\documentclass[10pt]{amsart}

\usepackage{amsfonts}
\usepackage{mathrsfs}
\usepackage{amsmath}
\usepackage{amsthm}
\usepackage{amssymb}
\usepackage{latexsym}
\usepackage[all]{xy}
\usepackage[plainpages,urlcolor=blue]{hyperref}
\usepackage{verbatim}
\usepackage{txfonts}
\usepackage{graphicx}
\usepackage{amscd ,amsbsy}

\usepackage{tikz}
\usetikzlibrary{arrows,calc}
\tikzset{
>=stealth',
help lines/.style={dashed, thick},
axis/.style={<->},
important line/.style={thick},
connection/.style={thick, dotted},
}

\newtheorem{theorem}{Theorem}[section]
\newtheorem{proposition}{Proposition}[section]

\newtheorem{lemma}{Lemma}[section]
\newtheorem{definition}{Definition}[section]
\newtheorem{remark}{Remark}[section]

\newcommand{\al}{\alpha}
\newcommand{\eps}{\epsilon}

\newcommand{\la}{\lambda}

\newcommand{\h}{\mathfrak{h}}

\newcommand{\mbv}{\mathbf{v}}

\newcommand{\mcA}{\mathcal{A}}

\newcommand{\mfb}{\mathfrak{b}}

\newcommand{\mfD}{\mathfrak{D}}

\newcommand{\mfg}{\mathfrak{g}}
\newcommand{\mfgl}{\mathfrak{g}\mathfrak{l}}
\newcommand{\mfh}{\mathfrak{h}}

\newcommand{\mfn}{\mathfrak{n}}

\newcommand{\mfsl}{\mathfrak{s}\mathfrak{l}}

\newcommand{\mfU}{\mathfrak{U}}

\newcommand{\g}{\mathfrak{g}}

\newcommand{\msD}{\mathsf{D}}

\newcommand{\msH}{\mathsf{H}}

\newcommand{\msP}{\mathsf{P}}

\newcommand{\msX}{\mathsf{X}}

\newcommand{\End}{\mathrm{End}}

\newcommand{\iso}{\stackrel{\sim}{\longrightarrow}}

\newcommand{\wh}{\widehat}
\newcommand{\wt}{\widetilde}
\newcommand{\lra}{\longrightarrow}

\newcommand{\C}{\mathbb{C}}
\newcommand{\Z}{\mathbb{Z}}

\newcommand{\ot}{\otimes}
\newcommand{\ol}{\overline}

\DeclareMathOperator{\KZB}{KZB}

\newcommand {\Omit}[1]{}

\newcommand{\N}{\mathbb{N}}

\newenvironment{romenum}
{

\begin{enumerate}}{\end{enumerate}}

\DeclareMathOperator{\ad}{ad}

\author[N.~Guay]{Nicolas~Guay}
\address{Department of Mathematics,
University of Alberta,
CAB 632, Edmonton, AB T6G 2G1, Canada}
\email{nguay@ualberta.ca}

\author[Y.~Yang]{Yaping~Yang}
\address{Department of Mathematics and Statistics,
University of Massachusetts Amherst,
Amherst, MA 01003-9305, USA}
\email{yaping@math.umass.edu}
\begin{document} 

\setlength{\pdfpagewidth}{8.5in}
\setlength{\pdfpageheight}{11in}

\begin{center}
\Large{\textbf{On deformed double current algebras for simple Lie algebras}}

\bigskip

Nicolas Guay, Yaping Yang
\end{center}

\bigskip

\bigskip
Abstract: We prove the equivalence of two presentations of deformed double current algebras associated to a complex simple Lie algebra $\mfg$, the first one obtained via a degeneration of affine Yangians while the other one naturally appeared in the construction of the  elliptic Casimir connection. We also construct a specific central element of these algebras and, in type A, show that they contain a very large center for certain values of their parameters.

\bigskip

\tableofcontents

\section{Introduction}

Deformed double current algebras were introduced by the first named author in \cite{G1,G2,G3}. They are deformations of the enveloping algebra of the universal central extension of the double current algebra $\g \ot_{\C} \C[u, v]$ where $\mfg$ is a finite dimensional, simple, complex Lie algebra. The deformed double current algebra $\mfD(\g)$ is the quantum algebra analog of the rational Cherednik algebra introduced in \cite{EG}, at least for two reasons: when $\mfg= \mfsl_n$, there is a Schur-Weyl type duality between the rational Cherednik algebra of type $A$ and $\mfD(\mathfrak{sl}_n)$ (\cite{G1}, Theorem 8.4 and \cite{G2}, Theorem 13.1); moreover, $\mfD(\g)$ can be obtained by degenerating twice the quantum toroidal algebra of $\mfg$ \cite{G2,G3}, in a way similar to how rational Cherednik algebras can be viewed as two-steps degenerations of elliptic Cherednik algebras (i.e. double affine Hecke algebras). 

The first presentation of $\mfD(\g)$ (which is similar to the Kac-Moody presentation of an affine Lie algebra) is obtained from the defining relations of the affine Yangian $Y(\widehat{\g})$ via a certain degeneration: see \cite{G2}, Theorem 12.1 and \cite{G3}, Theorem 5.5. A second presentation of $\mfD(\mathfrak{sl}_n)$, for $n\ge 4$, is given in \cite{G2} and is useful to establish the Schur-Weyl functor with rational Cherednik algebras; it involves two subalgebras which are enveloping algebras of current algebras in one variable and which play a symmetric role. In this paper, we extend this second presentation to the deformed double current algebra $\mfD(\g)$ for an arbitrary simple Lie algebra $\g$ of rank $\ge 3$: see Definition \ref{ddca double loop} and Theorem~\ref{thm:main theorem}.  (Deformed double current algebras for $\mfsl_2$ and $\mfsl_3$ were not defined in \cite{G2} due to some complications when the rank of $\mfg$ is small: we will propose a reasonable definition in \cite{GuYa}. We also do not know what is the correct definition in type $B_2$ and $G_2$.)

The presentation of $\mfD(\g)$ in terms of two current subalgebras came up naturally in the joint work of the second named author with Valerio Toledano Laredo (see \cite{TLY1,TLY2}). In \cite{CEE}, Calaque-Enriquez-Etingof constructed the universal Knizhnik-Zamolodchikov-Bernard (KZB) connection, which is a flat connection on the configuration space of $n$ points on an elliptic curve. Their work, which is for type ${\mathsf A}_n$, was generalized by V. Toledano Laredo and the second named author in \cite{TLY1} to any finite root system $\Phi$. The universal KZB connection $\nabla_{\KZB}$ obtained in \cite{TLY1} is a flat connection on the regular locus of the elliptic configuration space associated to $\Phi$ with values in a holonomy algebra. Two concrete incarnations of the KZB connection were obtained in \cite{TLY1,TLY2} by mapping the holonomy algebra to a rational Cherednik algebra and to a deformed double current algebra. This second incarnation is called the elliptic Casimir connection and is an elliptic analog of the rational Casimir connection \cite{FMTV, MTL, TL1} and of the trigonometric Casimir connection \cite{TL2}.  The construction of the elliptic Casimir connection relies crucially on the second presentation and on some of the properties of the deformed double current algebra obtained in the current paper.

Our first theorem (which is Theorem \ref{thm:main theorem}) states that both presentations of $\mfD(\g)$ alluded to above are equivalent. The other results of this paper concern central elements in $\mfD(\g)$. In sections \ref{sec: central element} and  \ref{Sec: sl_n two par}, we construct a certain central element of $\mfD(\mfg)$ which is essential for the construction in \cite{TLY2} of a homomorphism from the holonomy Lie algebra to the deformed double current algebra: see Theorem~\ref{central ele any g} and Theorem~\ref{central ele sln}. When $\mfg=\mfsl_n$, the definition of $\mfD(\g)$ involves two deformation parameters, so section \ref{Sec: sl_n two par} is specially devoted to the construction of that central element in this case. In the last section, for certain specific values of the two deformation parameters, we prove that the center of $\mfD(\mfsl_n)$ contains two subrings isomorphic to the ring of polynomials in infinitely many variables and we identify precisely two infinite sets of central elements which are algebraically independent. This result was inspired by an analogous one about rational Cherednik algebras, namely Proposition 3.6 in \cite{G} which states that the center of the rational Cherednik algebra when the parameter $t=0$ contains the subalgebra $\C[\h]^W\otimes \C[\h^*]^W$, $\h$ being here the reflection representation of the Weyl group $W$. Theorem \ref{PYang} is also inspired by a similar result for the rational Cherednik algebra $\msH_{t,c}(S_l)$ associated to the symmetric group, namely Proposition 4.3 in \cite{EG} which states that $\msH_{t,c}(S_l)$ contains a subalgebra isomorphic to the degenerate affine Hecke algebra.

\textbf{Acknowledgement}
The first named author acknowledges the financial support received from the Natural Sciences and Engineering Research Council of Canada via its Discovery Grant program. We are grateful to Pavel Etingof for useful suggestions and comments and to Valerio Toledano Laredo for numerous helpful discussions. We warmly thank the referee for a careful reading of our manuscript.  Part of the work was done when the second named author visited the Mathematical Sciences Research Institute and the Max Planck Institute for Mathematics. She wishes to acknowledge the hospitality of MSRI and MPIM.

\section{The deformed double current algebra of \texorpdfstring{$\g$}{g}}\label{rankge3}

Let $\g$ be a finite-dimensional, simple Lie algebra over $\mathbb{C}$. In this paper, we will assume that the rank of $\mfg$ is $\ge 3$, the reason being that we do not know what is the correct definition of deformed double current algebra when $\mfg$ is of Dynkin type $A_1,A_2,B_2$ or $G_2$: we expect that it will involve more complicated relations. Most of the results of this paper quite likely hold also when the rank of $\mfg$ is 1 or 2, but the required computations would probably be more daunting.

Let $(\cdot,\cdot)$ be the Killing form on $\mfg$ and let $X_i^{\pm},H_i, 1\le i \le N$, be the Chevalley generators of $\mfg$ normalized so that $(X_i^+,X_i^-)=1$ and $[X_i^+,X_i^-]=H_i$. Let $\Delta$ be the set of  roots for $\mfg$ and $\Delta^+$ be a set of positive roots. For each positive root $\alpha$, we choose generators $X_{\alpha}^{\pm}$ of $\mfg_{\pm\alpha}$ such that $(X_{\alpha}^{+},X_{\alpha}^{-})=1$ and $X_{\alpha_i}^{\pm}=X_i^{\pm}$. (As usual, $\al_i, 1\le i\le N$, denote the simple roots of $\mfg$.) If $\al>0$, set $X_{\al} = X_{\al}^+$ and $H_{\al} = [X_{\al}^+, X_{\al}^-]$; if $\al<0$, set $X_{\al} = X_{-\al}^-$ and $H_{\al} = -H_{-\al}$. Let $C=(c_{ij})_{i,j=0}^N$ be the affine Cartan matrix of the affine Lie algebra $\wh{\mfg}$ and the scalars $d_0,d_1,\ldots, d_N$ be such that $(d_ic_{ij})_{i,j=0}^N$ is a symmetric matrix. The index $0$ corresponds to the extending vertex in the Dynkin diagram of $\wh{\mfg}$. 
  
Set $\g[u] = \g \otimes_{\C} \C[u]$ and $\g[u, v]=\g\otimes_{\C}\C[u, v]$, which we call a double current Lie algebra.  For general results about the universal central extension of $\mfg[u,v]$ (e.g. that it is isomorphic to $\mfg[u,v] \oplus \frac{\Omega^1(\C[u,v])}{d\C[u,v]}$, the center being isomorphic to the quotient vector space $\frac{\Omega^1(\C[u,v])}{d\C[u,v]}$ of all 1-forms in the plane modulo the exact forms), see \cite{Ka} and \cite{KaLo}.

\begin{proposition}[\cite{G3}, Lemmas 4.1, 4.2]\label{prop:centr ext}
The universal central extension $\wh{\g[u, v]}$ of $\g[u, v]$ is isomorphic to the Lie algebra generated by elements $X_{i, r}^{\pm}$, $H_{i, r}$ for $1\le i\le N$, $r=0, 1$ and $X_{0, 0}^+$, $X_{0, 1}^+$ subject to the following relations:
\begin{gather}
[H_{i_1, r_1}, H_{i_2, r_2}]=0, \ [H_{i_1, 0}, X_{i_3, r_3}^{\pm}]=\pm d_{i_1}c_{i_1, i_3}X_{i_3, r_3}^{\pm}, \mbox{ for } 1\leq i_1, i_2 \leq N,\, 0\leq i_3 \leq N,\, r_1, r_2, r_3= \hbox{$0$ or $1$} \label{rel:current rel0}\\[1.2ex]
[H_{i_1, 1}, X_{i_2, 0}^{\pm}]=[H_{i_1, 0}, X_{i_2, 1}^{\pm}], \, 1\leq i_1 \leq N, \, 0\leq i_2 \leq N, \ [X_{i_1, 1}^{\pm}, X_{i_2, 0}^{\pm}]=[X_{i_1, 0}^{\pm}, X_{i_2, 1}^{\pm}], \, 0\leq i_1, i_2 \leq N \label{rel:current rel1}\\[1.2ex]
 [X_{i_1, r_1}^+, X_{i_2, r_2}^-]=\delta_{i_1 i_2}H_{i_1, r_1+r_2}, \, 0\leq i_1 \leq N, \, 1\leq i_2 \leq N, r_1+r_2=0, 1\\[1.2ex]
\Big[X_{i_1, 0}^\pm, \big[X_{i_1, 0}^\pm, \dots, [X_{i_1, 0}^\pm, X_{i_2, 0}^\pm]\dots\big]\Big]=0, \,\ 0\leq i_1, i_2 \leq N, 
\text{where $X_{i_1, 0}^\pm$ appears $1-c_{i_1, i_2}$ times}. \label{rel:current rel2}
\end{gather}
In  \eqref{rel:current rel1} and \eqref{rel:current rel2}, when $i_1=0$ and $i_2=0$, there is a relation only when $\pm = +$. The same applies to \eqref{rel:current rel0} when $i_3=0$.
\end{proposition}

We assume for the remainder of this section that $\mfg$ is not of Dynkin type $A$. (When $\mfg=\mfsl_n$, the results of this section are contained in \cite{G2}; moreover, in this case, Definition \ref{ddca} has to be modified because, in the Dynkin diagram of affine type $\wh{A}$, the extending vertex is connected to two other vertices, not just to one.) Then there is a unique $k\in \{1, \dots, N\}$ such that $c_{0k}\neq 0$. In other words, $k$ is the label of the unique vertex in the Dynkin diagram of $\wh{\mfg}$ to which the zero node is connected. Let $\theta$ be the highest root of $\g$. If $a,b$ are any two elements in an algebra $\mcA$, we set $S(a,b) = ab + ba$. Let
\begin{equation}
\omega_i^{\pm} = \pm\frac{1}{4}\sum_{\alpha\in\Delta^+} S\big(
[X_i^{\pm},X_{\alpha}^{\pm}],X_{\alpha}^{\mp}\big)-\frac{1}{4}S(X_i^{\pm},H_i) \in \mfU(\g) \text{ and } \nu_i = [\omega_i^+,X_i^-] = \frac{1}{4} \sum_{\alpha\in\Delta^+} (\alpha_i, \alpha)S(X_{\alpha}^+, X_{\alpha}^-)-\frac{H_i^2}{2} \in \mfU(\mfg).\label{w1}
\end{equation}
We will write also $\omega_i^{\pm}$ and $\nu_i$ to denote the corresponding elements in Definition \ref{ddca} below via the homomorphism $\mfU(\mfg) \lra \msD(\mfg)$ given by $X_i^{\pm},H_i \mapsto \msX_{i,0}^{\pm},\msH_{i,0}$ for $1\le i \le N$.

\begin{definition}[\cite{G3}, Definition 5.3 with $\lambda=1$]\label{ddca}
Let $\la\in\C$. The deformed double current algebra $\msD(\mfg)$ is the $\C$-algebra
 generated
by $\msX_{i,r}^{\pm}, \msH_{i,r}$ and $\msX_{0,r}^+$ for $1\le i\le N,\, r=0,1$
subjected to the same relations as those in Proposition \ref{prop:centr ext}, except
 that the following relations involving $\msX_{0,r}^+$ must be modified:
\begin{equation} [\msX_{k,1}^+, \msX_{0,0}^{+}] - [\msX_{k,0}^+,
\msX_{0,1}^{+}] = \frac{d_0 c_{0k} \la}{2} S(\msX_{k,0}^+,\msX_{\theta}^-)
+ \la[\omega_k^+,\msX_{\theta}^-] + \la[\msX_{k,0}^+,\omega_0^+]
\label{ddca1} \end{equation}
\begin{equation} [\msH_{k,1}, \msX_{0,0}^{+}] - [\msH_{k,0},
\msX_{0,1}^{+}] =  \frac{d_0 c_{0k} \la}{2}S(\msH_{k,0},\msX_{\theta}^-)
+ d_0 c_{0k} \la\omega_0^+ + \la[\nu_k,\msX_{\theta}^-]\label{ddca2}
\end{equation} 
\begin{equation} [\msX_{0,0}^+,\msX_{k,0}^-] = 0 , \;\; [\msX_{0,1}^+,\msX_{k,0}^-] =
\lambda[\msX_{k,0}^-,\omega_0^+], \;\; [\msX_{0,1}^{+}, \msX_{0,0}^{+}] =
2d_0 \la \msX_{0,0}^+ \msX_{\theta}^-  \label{ddca3} \end{equation}
\begin{equation} [\msX_{0,0}^+,\msX_{i,1}^{\pm}] =
\la [ \msX_{\theta}^-, \omega_i^{\pm}], \;\;
[\msX_{0,1}^+,\msX_{i,0}^{\pm}] = - \la [\omega_0^+,
\msX_{i,0}^{\pm}] \,\ \mbox{ for } i\neq 0,k. \label{ddca4}
\end{equation} The elements $\msX_{\theta}^-$ and $\omega_0^+$ are defined in the following way.  We write $X_{\theta}^-$ as $X_{\theta}^- = [X_{k,0}^-,X_{\theta-\al_k}^-]$ (it may be necessary to rescale $X_{\theta-\al_k}^{\pm}$ to achieve this) and set $\msX_{\theta}^- = [\msX_{k,0}^-,\msX_{\theta-\al_k}^-] \in \msD(\mfg)$. (Here, $ \msX_{\theta-\al_k}^- \in \mfg \subset \msD(\mfg)$.) We set $\omega_0^+ = -[\omega_k^-, X_{\theta-\al_k}^-]$.
\end{definition}
Applying $[\cdot, X_{k,0}^-]$ to \eqref{ddca2} and using \eqref{ddca3} lead to an expression for $[X_{0,0}^+,X_{k,1}^-]$ as an element of $\mfU(\mfg)$. The relations \eqref{ddca1} - \eqref{ddca4} were arrived at in \cite{G3} after considering a certain degeneration of the affine Yangian of $\mfg$, taking Proposition \ref{prop:centr ext} into account.

The next definition appeared naturally in the work \cite{TLY2} and was first given in \cite{G2}  in the case of $\mathfrak{sl}_n$ for $n\ge 4$.

\begin{definition}\label{ddca double loop}
Let $\la\in\C$. The $\C$-algebra $\mfD(\mathfrak{g})$ is generated by elements $X,K(X), Q(X),P(X)$ for all  $X\in \mfg$ such that
\begin{itemize}
\item The assignment $X \mapsto X, \, X\otimes v \mapsto K(X)$ (resp. $X \mapsto X, \,X\otimes u \mapsto Q(X)$) extends to an algebra homomorphism $\mfU(\g[u]) \lra \mfD(\mathfrak{g})$ (resp. $\mfU(\g[v]) \lra \mfD(\mathfrak{g})$);
\item $P(X)$ is linear in $X$, and for any $X, X'\in \g$, $[P(X), X']=P([X, X'])$,
\end{itemize}
and the following relation holds for all root vectors $X_{\beta_1},X_{\beta_2}\in\g$ with $\beta_1\neq -\beta_2$:
\begin{equation}
[K(X_{\beta_1}),Q(X_{\beta_2})]  = P([X_{\beta_1},X_{\beta_2}]) - \frac{(\beta_1,\beta_2)\lambda}{4}  S(X_{\beta_1},X_{\beta_2}) + \frac{ \lambda}{4} \sum_{\alpha\in\Delta} S([X_{\beta_1},X_{\alpha}],[X_{-\alpha},X_{\beta_2}]).\label{2rel}
\end{equation}
\end{definition}

The name deformed double current algebra seems a priori more appropriate for  $\mfD(\mathfrak{g})$ since this algebra is clearly built from two current algebras $\mfg[u]$ and $\mfg[v]$: Theorem \ref{thm:main theorem} below states that $\mfD(\mathfrak{g})$ and $\msD(\mfg)$ are isomorphic, hence we can also call $\mfD(\mathfrak{g})$ a deformed double current algebra. The algebra $\msD(\mfg)$ was obtained in \cite{G3} as a degenerate version of the affine Yangian of $\mfg$, so one consequence of our main theorem just below is that $\mfD(\mathfrak{g})$ is also such a degeneration. Moreover, affine Yangians can be obtained from quantum toroidal algebras via a similar type of degeneration \cite{GuMa,GTL}, so a deformed double current algebra can be viewed as a two-step degeneration of a quantum toroidal algebra.

The following proposition will be useful to avoid repeating parts of certain proofs.
\begin{proposition}\cite[Proposition 12.1]{G2}\label{Prop: auto}
There is an automorphism of $\mfD(\mathfrak{g})$ which is given by 
\[
X\mapsto X, \,\ K(X)\mapsto -Q(X), 
\,\ Q(X)\mapsto K(X), \,\ P(X)\mapsto -P(X)
\]for any $X\in \g$.
\end{proposition}

\begin{theorem} \label{thm:main theorem}
There exists an algebra isomorphism
$\varphi: \msD(\mfg) \to \mfD(\mfg)$
 given by \[ \varphi(\msX_{i,0}^{\pm}) = X_{i}^{\pm}, \; \varphi(\msH_{i,0}) = H_{i},  \; \varphi(\msX_{i,1}^{\pm}) = Q(X_{i}^{\pm}), \;  \varphi(\msH_{i,1}) = Q(H_{i}) \hbox{ for } 1\le i\le N, \]  \[  \varphi(\msX_{0,0}^+) = K(X_{\theta}^-),  \;  \varphi(\msX_{0,1}^{+}) = P(X_{\theta}^-) - \lambda\omega_0^+.  \]
\end{theorem}

The proof of this theorem will be given in the following two subsections.

\subsection{The map $\varphi$ is a homomorphism of algebras}
The first step is to prove that the assignment $\varphi$ given in Theorem \ref{thm:main theorem} extends to a homomorphism of algebras.

\begin{lemma}\label{omegai} The element $\omega_i^{\pm}$ of $\mfU(\g)$ defined in \eqref{w1} can be rewritten in the following way: \begin{equation}
\omega_i^{\pm} = \mp\frac{1}{4}\sum_{\alpha\in \Delta^+}
S\big([X_i^{\pm},X_{\alpha}^{\mp}],X_{\alpha}^{\pm}\big). \label{w2}
\end{equation}
\end{lemma}
\begin{proof}
The Casimir element of $\mfg$ is \[ \Omega = \sum_{\al\in\Delta^+} S(X_{\al}^+, X_{\al}^-) + \sum_{i=1}^N \wt{h}^i \wt{h}_i   \] where $\{ \wt{h}_1,\ldots,\wt{h}_N \}$ is a  basis of the Cartan subalgebra $\mfh$ and $\{ \wt{h}^1,\ldots,\wt{h}^N \}$ is the dual basis with respect to the Killing form. $\Omega$ is in the center of $\mfU(\mfg)$, so \[ 0 = \pm\frac{1}{4} [X_{i}^{\pm},\Omega] = \omega_i^{\pm} \pm  \frac{1}{4}\sum_{\al\in\Delta^+} S(X_{\al}^{\pm}, [X_{i}^{\pm},X_{\al}^{\mp}]) \] because $\sum_{j=1}^N [X_{i}^{\pm},\wt{h}^j] \wt{h}_j = \mp \sum_{j=1}^N (H_i,\wt{h}^j) X_{i}^{\pm} \wt{h}_j = \mp X_{i}^{\pm} H_i$. 
\end{proof}

We have to verify that $\varphi$ respects the defining relations of $\msD(\mfg)$. Let us start with \eqref{ddca1}. We have to check that
\begin{equation*}
[Q(X_k^+), K(X_{\theta}^-)] -[X_k^+, P(X_{\theta}^-)] = \frac{d_0 c_{0k}\lambda}{2} S(X_k^+, X_{\theta}^-) - \frac{\lambda}{4} \sum_{\al\in\Delta^+} S([X_k^+, X_{\al}^-], [X_{\al}^+, X_{\theta}^-]) - \frac{\lambda}{4} S\big(\big[[X_k^+, X_{k}^-] , X_{\theta}^-\big] , X_{k}^+\big) . 
\end{equation*}
(Observe that $[X_{k}^-, X_{\theta}^-]=0$ and that if $\big[ [X_k^+, X_{\al}^-], X_{\theta}^-\big]$ is a root vector (where $\al$ is positive), then $\al_k-\al-\theta\ge -\theta$, so $\al\le \al_k$: since $\al_k$ is simple and both are positive, this forces $\al$ to be equal to $\al_k$.)

 The previous equality is equivalent to \[ [K(X_{\theta}^-),Q(X_k^+)] = P([X_{\theta}^-,X_k^+]) - \frac{d_0 c_{0k} \lambda}{4} S(X_k^+, X_{\theta}^-) + \frac{\lambda}{4} \sum_{\al\in\Delta^+} S([X_k^+, X_{\al}^-], [X_{\al}^+, X_{\theta}^-])  \] because $\big[ [X_k^+, X_{k}^-], X_{\theta}^-\big] = d_k c_{k0} X_{\theta}^- = d_0 c_{0k}  X_{\theta}^-$.
 This is true in $\mfD(\mfg)$ because $d_0 c_{0k} = (-\theta,\alpha_k)$. Thus, $\varphi$ preserves relation \eqref{ddca1}.

\eqref{ddca2} follows from \eqref{ddca1} because $[X_{\theta}^-, X_{k,0}^{-}] =0$. For the same reason, the first and second relation in \eqref{ddca3} hold.

For the third relation in \eqref{ddca3}, we need to compute $[P(X_{\theta}^-) - \lambda\omega_0^+, K(X_{\theta}^-)]$. Notice that if $[X_{\al}^-, X_{\theta-\al_k}^-]\neq 0$ for a positive root $\al$, then $[X_{\al}^-, X_{\theta-\al_k}^-]$ is a root vector in $\mfg$ for the root $-\al-\theta+\al_k$, so $-\al-\theta+\al_k\ge -\theta$ and thus $\al\le \al_k$: this forces $\al=\al_k$ because $\al_k$ is a simple root. Similarly, if $\Big[ \big[ [X_k^-, X_{\al}^+], X_{\theta}^-\big] ,X_{\theta-\al_k}^-\Big] \neq 0$ then it is a root vector in $\mfg$ for the root $-\al_k+\al-\theta-\theta+\al_k$, so $-\al_k+\al-\theta-\theta+\al_k\ge -\theta$ hence $\al\ge \theta$: this implies that $\al=\theta$. These two observations serve to obtain the third equality below. Starting with Lemma \ref{omegai}, we obtain: \begin{align}
[\omega_0^+, K(X_{\theta}^-)] & =   - \big[ [\omega_k^-,X_{\theta-\al_k}^-],K(X_{\theta}^-)\big] = -\frac{1}{4} \sum_{\al\in\Delta^+} \Big[ S\big(\big[ [X_k^-, X_{\al}^+], K(X_{\theta}^-)\big], X_{\al}^-\big),X_{\theta-\al_k}^-\Big]    \notag \\
& =  -\frac{1}{4}  S\Big(\Big[\big[ [X_k^-, X_{\theta}^+], K(X_{\theta}^-)\big],X_{\theta-\al_k}^-\Big] , X_{\theta}^-\Big) - \frac{1}{4}S\Big(\big[ [X_k^-, X_{\al_k}^+], K(X_{\theta}^-)\big] , [X_{\al_k}^-,X_{\theta-\al_k}^-]\Big)    \notag \\
& =  -\frac{(\al_k,\theta)}{4}   S([K(X_k^-),X_{\theta-\al_k}^-] , X_{\theta}^-) -\frac{(\al_k,\theta)}{4}  S(K(X_{\theta}^-) , X_{\theta}^-)  =  -(\theta, \alpha_k) K(X_{\theta}^-) X_{\theta}^- \label{b}
\end{align}

We show that  $[P(X_{\theta}^-), K(X_{\theta}^-)] = \lambda( (\theta, \theta)-(\alpha_k, \theta)) K(X_{\theta}^-) X_{-\theta}$ as follows.
Start with relation \eqref{2rel} with $\beta_1 = -\al_k$ and $\beta_2 = -\theta+\al_k$:
\begin{equation}
[K(X_{\al_k}^-),Q(X_{\theta-\al_k}^-)] = P(X_{\theta}^-) -  \frac{(\al_k,\theta-\al_k)\lambda}{4} S(X_{\al_k}^-,X_{\theta-\al_k}^-) + \frac{\lambda }{4} \sum_{\alpha\in\Delta} S\big( [X_{\al_k}^-,X_{\alpha}],[X_{-\alpha},X_{\theta-\al_k}^-]\big) \label{a}
\end{equation}
Observe that if $[X_{\al_k}^-,X_{\al}]\neq 0$ and $\big[ [X_{-\al},X_{\theta-\al_k}^-],X_{\theta}^-\big]\neq 0$, then both are roots vectors for the roots $-\al_k+\al$ and $-\al-\theta+\al_k-\theta$, so $-\al_k+\al \ge -\theta$ and $-\al-\theta+\al_k-\theta \ge -\theta$, hence $\al \ge -\theta + \al_k$, $-\theta+\al_k \ge \al$ and thus $\al = -\theta+\al_k$. Similarly, if $\big[ [X_{\al_k}^-,X_{\alpha}],X_{\theta}^-\big] \neq 0$ and $[X_{-\alpha},X_{\theta-\al_k}^-] \neq 0$, then $\al=\al_k$ by a similar argument.

Using these two observations, we now apply $[\cdot,K(X_{\theta}^-)]$ to \eqref{a} to get: \begin{align}
\big[ K(X_{\al_k}^-),[Q(X_{\theta-\al_k}^-),K(X_{\theta}^-)]\big]  = {} &  [P(X_{\theta}^-),K(X_{\theta}^-)]  + \frac{\lambda }{4}  S\Big( [X_{\al_k}^-,X_{-\theta + \al_k}],\big[ [X_{\theta - \al_k},X_{\theta-\al_k}^-],K(X_{\theta}^-)\big]\Big) \notag \\
& {} + \frac{\lambda }{4} S\Big(\big[ [X_{\al_k}^-,X_{\alpha_k}],K(X_{\theta}^-)\big],[X_{-\alpha_k},X_{\theta-\al_k}^-]\Big)  \notag \\
 = {} &  [P(X_{\theta}^-),K(X_{\theta}^-)]  + \frac{\lambda }{4}  S\big( X_{-\theta},[H_{\theta - \al_k},K(X_{\theta}^-)]\big) +  \frac{(\alpha_k, \theta)\lambda}{4}  S(K(X_{\theta}^-),X_{\theta}^-)  \notag \\
 = {} &  [P(X_{\theta}^-),K(X_{\theta}^-)] + \lambda\frac{2(\alpha_k, \theta)-(\theta, \theta)}{4}  S(K(X_{\theta}^-),X_{\theta}^-). \label{c}
\end{align}

We compute the left-hand side of the above equality \eqref{c}.
 \begin{align*}
[Q(X_{\theta-\al_k}^-),K(X_{\theta}^-)] & =  \frac{(\theta-\al_k,\theta)\lambda}{4}S(X_{\theta-\al_k}^-,X_{\theta}^-) - \frac{\lambda}{4} S([X_{\theta}^-,X_{\alpha_k}],[X_{-\alpha_k},X_{\theta-\al_k}^-]) \\
& =  \frac{(\theta-\al_k,\theta)\lambda}{4}S(X_{\theta-\al_k}^-,X_{\theta}^-) - \frac{\lambda}{4} S([X_{\theta}^-,X_{\alpha_k}],X_{\theta}^-)
\end{align*}
so
 \begin{align}
 \big[ K(X_{\al_k}^-),[Q(X_{\theta-\al_k}^-),K(X_{\theta}^-)]\big] & =   \Big[ K(X_{\al_k}^-), \frac{(\theta-\al_k,\theta)\lambda}{4}S(X_{\theta-\al_k}^-,X_{\theta}^-) - \frac{\lambda}{4} S([X_{\theta}^-,X_{\alpha_k}],X_{\theta}^-) \Big] \notag \\
 & =   \frac{(\theta-\al_k,\theta)\lambda}{4}S(K(X_{\theta}^-),X_{\theta}^-)  + \frac{(\alpha_k, \theta) \lambda}{4} S(K(X_{\theta}^-), X_{\theta}^-) = \frac{(\theta,\theta)\lambda}{4}S(K(X_{\theta}^-),X_{\theta}^-) . \label{d}
 \end{align}

Substituting \eqref{d} into \eqref{c} yields $[P(X_{\theta}^-), K(X_{\theta}^-)] = \lambda( (\theta, \theta)-(\alpha_k, \theta)) K(X_{\theta}^-) X_{-\theta}$ and it follows from \eqref{b} that $[P(X_{\theta}^-)-\lambda \omega_0^+, K(X_{\theta}^-)]= 
\lambda (\theta, \theta) K(X_{\theta}^-) X_{-\theta}$ as desired.

Finally, we check that $\varphi$ respects the first relation in \eqref{ddca4}, so with $i\neq 0,k$ (hence $(\theta,\al_i)=0$):
\[
[K(X_{\theta}^-), Q(X_i^{-})] = \frac{\lambda}{4} \sum_{\al\in\Delta^+} S\big( [X_i^-, X_{\al}^-], [X_{\al}^+, X_{\theta}^-]\big) = -\lambda[\omega_i^-,X_{\theta}^-] 
\]
using \eqref{w1} and
\[
[K(X_{\theta}^-), Q(X_i^{+})] = \frac{\lambda}{4} \sum_{\al\in\Delta^+} S\big( [X_i^+, X_{\al}^-], [X_{\al}^+, X_{\theta}^-]\big) = -\lambda[\omega_i^+,X_{\theta}^-] 
\]
this time using \eqref{w2}.   

The second relation  in \eqref{ddca4} is also satisfied since $[P(X_{\theta}^-), X_{i,0}^{\pm}] = P([X_{\theta}^-, X_{i,0}^{\pm}])= 0$ when $i\neq k$.

\subsection{The map $\varphi$ is an isomorphism of algebras.}

In this section, we construct the inverse  map $\psi: \mathfrak{D}(\mathfrak{g}) \rightarrow \mathsf{D}(\mathfrak{g})$ of $\varphi$.

Let us first describe the images $\psi(X)$, $\psi(K(X))$, $\psi(Q(X))$ and $\psi(P(X))$ of the generators of $\mathfrak{D}(\mathfrak{g})$ under the map $\psi$, where $X$ is any element in $\mathfrak{g}$. The following lemma is a consequence of \cite{GK} and the standard theorem of Serre regarding presentations of semisimple Lie algebras. 
\begin{lemma}\label{lem:current algebra}
The current Lie algebra $\mfg[u]$ is isomorphic to the Lie algebra generated by elements $e_i,f_i,h_i$, $1\leq i\leq N$, and $e_0$ subject
to the following relations for $1 \leq i\leq N$:
\begin{eqnarray*}
&&\hbox{$e_i,f_i,h_i$ satisfy the usual relations of Serre's Theorem for $\mathfrak{g}$;}\\[1.2ex]
&&\hbox{$[h_i,e_0] = d_i c_{i0} e_0$,\, \ $[e_0,f_i]=0$};\\[1.2ex]
&&\hbox{$ad(e_i)^{1-c_{i0}}(e_0)=0$ and $ad(e_0)^{1-c_{0i}}(e_i)=0$.}
\end{eqnarray*}
\end{lemma}

In the defining relations of $\mathfrak{D}(\mathfrak{g})$, $\{K(X), X\mid X\in\mfg\}$ generate a subalgebra which is a quotient of $\mfU(\mathfrak{g}[v])$ (conjecturally, they are isomorphic). By Lemma \ref{lem:current algebra}, we have a homomorphism $\mfU(\mathfrak{g}[v]) \rightarrow \mathsf{D}(\mathfrak{g})$, given by $e_i \mapsto \msX_{i,0}^+, \, f_i \mapsto \msX_{i,0}^-, \, h_i \mapsto \msH_{i,0}$ for $1\le i \le N$, and $e_0 \mapsto \msX_{0, 0}^+$.
Likewise, we have a homomorphism $\mfU(\mathfrak{g}[u]) \rightarrow \mathsf{D}(\mathfrak{g})$, given by
$e_i \mapsto \msX_{i,0}^+, \, f_i \mapsto \msX_{i,0}^-, \, h_i \mapsto \msH_{i,0}$ for $1\le i \le N$, and $X_i^{\pm} \ot u \mapsto \msX_{i, 1}^{\pm}$, for $1\leq i\leq N$.
These two homomorphisms tell us how to define $\psi(K(X))$ and $\psi(Q(X))$ in $\mathsf{D}(\mathfrak{g})$, for any $X\in \g$;  in particular, $\psi(K(X_{\theta}^-)) = \msX_{0, 0}^+$ and $\psi(Q(X_i^{\pm}))= \msX_{i, 1}^{\pm}$.
\begin{lemma}\label{msPlem}
There is an $\ad(\mfg)$-module morphism $\msP: \mfg \to \msD(\g)$, such that $X_{-\theta} \mapsto \msP(X_{-\theta}):= \msX_{0, 1}^++\lambda \omega_0^+.$
That it is an $\ad(\mfg)$-module morphism means that:
\begin{enumerate}
 \item $\msP(X)$ is linear in $X$;
 \item $[X, \msP(X')]=\msP([X, X'])$, for any $X, X'\in \g$.
\end{enumerate}
\end{lemma}
\begin{proof}
\footnote{This proof was suggested to us by Valerio Toledano Laredo.}
Let $\g=\mfn^-\oplus \mfh\oplus \mfn^+$ be a triangular decomposition of $\g$, and $\mfb^-:=\mfn^-\oplus \mfh$.
Let $M(-\theta):=\mfU(\g)\otimes_{\mfU(\mfb^-)} \C_{-\theta}$ be the Verma module with lowest weight $-\theta$,
where $\C_{-\theta}$ is the 1--dimensional $\mfb^-$--module with trivial
$n^-$-action.
Thanks to the PBW Theorem, $M(-\theta) \cong \mfU(\mfn^+)\otimes \C_{-\theta}$. Denote the lowest weight vector of $M(-\theta)$
by $v^-$, so $v^-=1\otimes 1$. Then, $M(-\theta)$ as a $\mfU(\g)$ is generated by $v^-$, such that
\begin{equation}\label{lowest weight}
\hbox{$\mfn^-\cdot v^-=0$ and $h\cdot v^-= -\theta(h) v^-$, for all $h\in \mfh$.}
\end{equation}
We first construct a morphism $\widetilde{\msP}: M(-\theta)\to \msD(\g)$ given by $v^-\mapsto \msP(X_{-\theta}):= \msX_{0, 1}^++\lambda \omega_0^+$ and then show that $\widetilde{\msP}$ factors through the adjoint representation of $\g$, which gives us the commutative diagram  
$\xymatrix@C=0.8em@R=0.8em{
M(-\theta) \ar[rr]^{\widetilde{\msP}} \ar[dr] & &\msD(\g)\\
&\g \ar@{-->}[ru]_{\msP}&
}$.

To show the morphism $\widetilde{\msP}$ is well-defined, it suffices to verify that $\msP(X_{-\theta})$ satisfies \eqref{lowest weight}. Rewriting the defining relation \eqref{ddca3} of $\msD(\g)$, we have:
$[\msP(X_{-\theta}), \msX_{k, 0}^{-}]=0$, and relation \eqref{ddca4} gives $[\msP(X_{-\theta}), \msX_{i, 0}^{\pm}]=0$ for $i\neq k$. 

We check now that $[\msH_{i, 0}, \msP(X_{-\theta})]=(\alpha_i, -\theta)\msP(X_{-\theta})$: it follows from the relation in Proposition \ref{prop:centr ext} that $[\msH_{i,0}, \msX_{0, 1}^{+}]=d_ic_{i, 0}\msX_{0, 1}^{+}=(\alpha_i, -\theta)\msX_{0, 1}^{+}$; using that $\omega_0^+ = -[\omega_k^-, X_{\theta-\alpha_k}^-]$, we obtain $[H_{i, 0}, \omega_0^+] = -(\alpha_i, \theta)\omega_0^+$.

Thus, we have a (non-zero) $\g$-equivariant map from the Verma module $M(-\theta)$ with
lowest weight $-\theta$ to $\msD(\g)$ mapping the lowest weight vector $v^-$ to $\msP(X_{-\theta})$.

The next step is to show that $\msD(\g)$ is locally finite as an $\ad(\g)$-module. Note that the algebra $\msD(\g)$ is bigraded with the following $\Z_{\ge 0} \times \Z_{\ge 0}$-grading:
\begin{itemize}
  \item $\deg(\msX_{i, 0}^{\pm})=(0, 0)$, for $i=1, \dots, N$.
  \item $\deg(\msX_{i, 1}^{\pm})=(1, 0)$, for $i=1, \dots, N$ and $\deg(\msX_{0, 0}^+)=(0, 1)$.
  \item $\deg(\lambda)=(1, 1)$ and $\deg(\msX_{0, 1}^+)=(1, 1).$ 
\end{itemize} 
Now each graded piece is invariant under the adjoint action of $\g$ since the degree of the elements in $\g$ is $(0, 0)$.
The degree $(0, 0)$ piece of $\msD(\g)$ coincides with the enveloping algebra $\mfU(\g)$ and each graded piece as a $\mfU(\g)$-module is finitely generated. Since the enveloping algebra $\mfU(\g)$ is locally finite as $\ad(\g)$-module, each graded piece of $\msD(\g)$ is also locally finite.

Therefore, the $\g$-equivariant map $\widetilde{P}$ must factor through
a finite dimensional quotient of the Verma module $M(-\theta)$, and
there is only one such quotient, namely the adjoint representation of $\g$.
\end{proof}

The remainder of the proof consists in checking that, for any two roots $\beta_1\neq -\beta_2$,
\begin{equation*}
[\psi(K(X_{\beta_1})), \psi(Q(X_{\beta_2}))] = \psi(P([X_{\beta_1},X_{\beta_2}]) - \frac{(\beta_1,\beta_2)\lambda}{4}S(\psi(X_{\beta_1}),\psi(X_{\beta_2})) + \frac{ \lambda}{4} \sum_{\alpha\in\Delta} S([\psi(X_{\beta_1}),\psi(X_{\alpha})],[\psi(X_{-\alpha}),\psi(X_{\beta_2})]) .
\end{equation*}
From the defining relations of $\mathsf{D}(\mathfrak{g})$, this is known to be true in the cases when $\beta_1=-\theta$ and $\beta_2 = \pm \alpha_1, \ldots, \pm \alpha_N$. In order to see that it's true in general, we can use the standard operators $s_i$, whose definition is recalled below, which have the property that $s_i(X_{\alpha})$ is a root vector for the root $s_i(\alpha)$.

In order to simplify the notation, we denote $\psi(K(X_{\beta}))$ simply by $K(X_{\beta})$, and similarly for $\psi(Q(X_{\beta}))$ and $\psi(P(X_{\beta}))$. As observed in the previous paragraph, we know that  \eqref{2rel} holds in $\msD(\mfg)$ when $\beta_1 = -\theta$ and $\beta_2=\pm\al_1,\ldots, \pm\al_N$. The goal is to show, using this assumption along with $[K(X_1),X_2]=K([X_1,X_2])$, $[Q(X_1),X_2]=Q([X_1,X_2])$ and $[P(X_1),X_2]=P([X_1,X_2])$ (see Lemma \ref{msPlem}), that \eqref{2rel} must hold in full generality in $\msD(\mfg)$ for any two roots $\beta_1,\beta_2$ with $\beta_1\neq -\beta_2$.

Let $m:\mfU(\mfg) \ot_{\C} \mfU(\mfg) \lra \mfU(\mfg)$ be the multiplication map. View $\Omega$ as an element of $\mfg \ot_{\C} \mfg$. The following observation will be useful below:
\begin{align}
\sum_{\alpha\in\Delta} S([X_{\beta_1},X_{\alpha}],[X_{-\alpha},X_{\beta_2}])  = {} &  m \left( \sum_{\alpha\in\Delta} \big[ [X_{\beta_1} \ot 1, X_{\alpha} \ot X_{-\alpha}], 1 \ot X_{\beta_2} \big]\right) + m \left( \sum_{\alpha\in\Delta} \big[ [1 \ot X_{\beta_1} , X_{-\alpha} \ot X_{\alpha}],  X_{\beta_2} \ot 1 \big]\right) \notag \\
 = {} &  m \left( \big[ [X_{\beta_1} \ot 1, \Omega], 1 \ot X_{\beta_2} \big] - \sum_{i=1}^N \big[ [X_{\beta_1} \ot 1, \wt{h}^i \ot \wt{h}_i], 1 \ot X_{\beta_2} \big] \right) \notag \\
& {} + m \left( \big[ [ 1 \ot X_{\beta_1} , \Omega], X_{\beta_2} \ot 1 \big] - \sum_{i=1}^N \big[ [ 1 \ot X_{\beta_1}, \wt{h}^i \ot \wt{h}_i], X_{\beta_2} \ot 1 \big] \right). \label{SXbXa}
\end{align}
It is known that $[\Omega, X \ot 1 + 1 \ot X] =0$ for any $X\in\mfg$ and, consequently,  we have:
\begin{align}
\sum_{\alpha\in\Delta} \big[ S([X_{\beta_1},X_{\alpha}],[X_{-\alpha}, X_{\beta_2}]), X_\gamma\big]
 & -  \sum_{\alpha\in\Delta} S\big(\big[ [X_{\beta_1},X_\gamma],X_{\alpha}\big] ,[X_{-\alpha},X_{\beta_2}]\big)
 - \sum_{\alpha\in\Delta} S\big( [X_{\beta_1},X_{\alpha}],\big[ X_{-\alpha},[X_{\beta_2},X_\gamma]\big]\big)  \notag \\
& = -(\gamma, \beta_2)S([X_{\beta_1}, X_\gamma], X_{\beta_2})-(\gamma, \beta_1)S(X_{\beta_1}, [ X_{\beta_2}, X_\gamma]). \label{m}
\end{align}
Here, either $X_\gamma \in \mfg_\gamma$ for $\gamma\in \Delta$ or $\gamma=0$ and $X_{\gamma}$ is to be interpreted as an element in $\mfh$.  
\begin{remark}
Equality \eqref{m} is also valid if $X_{\beta_1}$ (resp. $X_{\beta_2}$) is replaced by an element of $\mfh$ and if the factor $(\gamma, \beta_1)$ (resp. $(\gamma, \beta_2)$) is replaced by $0$. When referring to \eqref{m}, we may also be considering these two cases.
\end{remark}
Equality \eqref{m} follows from \eqref{SXbXa} via the following computation:
\begin{align*}
\sum_{\alpha\in\Delta} & \big[S([X_{\beta_1},X_{\alpha}],[X_{-\alpha}, X_{\beta_2}]), X_\gamma\big]
  -  \sum_{\alpha\in\Delta} S\big(\big[ [X_{\beta_1},X_\gamma],X_{\alpha}\big] ,[X_{-\alpha},X_{\beta_2}]\big)
 - \sum_{\alpha\in\Delta} S\big( [X_{\beta_1},X_{\alpha}],\big[ X_{-\alpha},[X_{\beta_2},X_\gamma]\big]\big)  \notag \\
& = - \sum_{i=1}^N m\left(\Big[ \big[ X_{\beta_1} \ot 1, [\wt{h}^i \ot \wt{h}_i, X_\gamma \ot 1 + 1 \ot X_\gamma ]\big], 1 \ot X_{\beta_2} \Big] \right) \notag
  - \sum_{i=1}^N m \left( \Big[ \big[ 1 \ot X_{\beta_1}, [\wt{h}^i \ot \wt{h}_i, X_\gamma \ot 1 + 1 \ot X_\gamma ]\big] , X_{\beta_2} \ot 1 \Big] \right)\notag \\
 & = - m\left(\Big[ \big[ X_{\beta_1} \ot 1, X_\gamma \ot H_\gamma+H_\gamma \ot X_\gamma \big], 1 \ot X_{\beta_2} \Big] \right) \notag
  - m \left( \Big[ \big[ 1 \ot X_{\beta_1}, X_\gamma \ot H_\gamma+H_\gamma \ot X_\gamma \big] , X_{\beta_2} \ot 1 \Big] \right)\\
&   = - m\left(\big[ [X_{\beta_1}, X_\gamma] \ot H_\gamma+ [X_{\beta_1}, H_\gamma] \ot X_\gamma, 1 \ot X_{\beta_2} \big] \right) - m \left( \big[ X_\gamma \ot  [X_{\beta_1}, H_\gamma]+H_\gamma \ot  [X_{\beta_1}, X_\gamma] , X_{\beta_2} \ot 1 \big] \right)\\
& = - m\left((\gamma, \beta_2)[X_{\beta_1}, X_\gamma] \ot X_{\beta_2}+(\gamma, \beta_1) X_{\beta_1} \ot [X_{\beta_2}, X_\gamma] \right) \notag
  - m \left( (\gamma, \beta_1)[ X_{\beta_2}, X_\gamma] \ot  X_{\beta_1}+(\gamma, \beta_2)X_{\beta_2} \ot  [X_{\beta_1}, X_\gamma] \right)\\
& = -(\gamma, \beta_2)S([X_{\beta_1}, X_\gamma], X_{\beta_2})-(\gamma, \beta_1)S(X_{\beta_1}, [ X_{\beta_2}, X_\gamma]).
\end{align*}

We are supposing that \eqref{2rel} holds when $\beta_1 = -\theta$ and $\beta_2 = \pm\al_k$. For convenience, let's write down these two relations here:
\begin{equation} [K(X_{\theta}^-),Q(X_k^+)] = P([X_{\theta}^-,X_k^+]) + \frac{(\theta,\al_k)\lambda}{4} S(X_{\theta}^-,X_k^+) + \frac{\lambda}{4} \sum_{\al\in\Delta} S([X_{\theta}^-,X_{\al}],[X_{-\al},X_k^+]) \label{r1}
\end{equation} and
\begin{equation} [K(X_{\theta}^-),Q(X_k^-)] = - \frac{(\theta,\al_k)\lambda}{4} S(X_{\theta}^-,X_k^-) + \frac{\lambda}{4} \sum_{\al\in\Delta} S([X_{\theta}^-,X_{\al}],[X_{-\al},X_k^-]). \label{rr1} \end{equation}
We will need the following observation: $([X_{\theta}^+,X_k^-],X_{-\theta+\al_k}) = (X_{\theta}^+,[X_k^-,X_{-\theta+\al_k}]) = (X_{\theta}^+,X_{\theta}^-)=1 \Longrightarrow [X_{\theta}^+,X_k^-] = X_{\theta-\al_k}$. Let's apply $[X_{\theta}^+,\cdot]$ to relation \eqref{rr1} to obtain, using \eqref{m}:
\begin{align}
[K(H_{\theta})  ,Q(X_k^-)] & + [K(X_{\theta}^-),Q(X_{\theta-\al_k})] \notag \\ 
& = \frac{\lambda}{4} \sum_{\al\in\Delta} S([H_{\theta},X_{\al}],[X_{-\al},X_k^-]) + \frac{\lambda}{4} \sum_{\al\in\Delta} S([X_{\theta}^-,X_{\al}],[X_{-\al},X_{\theta-\al_k}]) +
 \frac{((\theta, \theta)-(\theta, \alpha_k))\lambda}{4} S(X_{\theta}^-,X_{\theta-\al_k}^+) \label{r2}.
\end{align}

Set $\wt{e}_i = \sqrt{\frac{2}{(\alpha_i,\alpha_i)}} X_{i}^+$ and $\wt{f}_i = \sqrt{\frac{2}{(\alpha_i,\alpha_i)}} X_{i}^-$ for $i\neq 0$. Consider the operators in the adjoint representation of $\mfg$ given by $s_{i} = \exp(\mathrm{ad}(\wt{f}_i)) \exp(-\mathrm{ad}(\wt{e}_i)) \exp(\mathrm{ad}(\wt{f}_i))$. It is known that $X\in\mfg_{\al} \Longrightarrow s_i(X_{\al}) \in \mfg_{s_i(\al)}$ (where $\mfg_{\al}$ is the root subspace of $\mfg$ for the root $\al$) and $s_i(H_j) = H_j - \frac{2(\al_i,\al_j)}{(\al_i,\al_i)}H_i $. We will use $s_i$ to denote either a simple reflection in the Weyl group $W$ of $\mfg$ or the corresponding operator in its adjoint representation. (This is an abuse of notation since those operators $s_i$ do not provide an action of $W$ on $\mfg$.)

Let $w_0$ be the longest element of the Weyl group. Then $-w_0$ is a permutation of the simple roots. Let us express $w_0$ as a product of simple reflections $w_0 = s_{i_1} \cdots s_{i_{\ell}}$ and let us denote also by $w_0$ the corresponding operator in the adjoint representation of $\mfg$ for this choice of decomposition.  Applying $w_0$ to \eqref{r1} shows that the following relation also holds (after rescaling if needed):
\begin{equation}
[K(X_{\theta}^+),Q(X_k^-)] = P([X_{\theta}^+,X_k^-]) + \frac{(\theta,\al_k)\lambda}{4} S(X_{\theta}^+,X_k^-) + \frac{\lambda}{4} \sum_{\al\in\Delta} S([X_{\theta}^+,X_{\al}],[X_{-\al},X_k^-]). \label{tkm}
\end{equation}
(For any $w\in W$, one can choose a decomposition into simple transpositions and obtain a corresponding operator in the adjoint representation with the property that $w(X_{\al})$ is a scalar multiple of $X_{w(\al)}$; moreover, since the Killing form is invariant under the adjoint action, $(w(X_{\al}),w(X_{-\al})) = (X_{\al},X_{-\al})=1$, so $w(X_{\al}) = aX_{w(\al)},w(X_{-\al}) = a^{-1} X_{-w(\al)} $ for some non-zero scalar $a$ and $S([X_{\theta}^+,X_{w(\al)}],[X_{-w(\al)},X_k^-]) = S([X_{\theta}^+,w(X_{\al})],[w(X_{-\al}),X_k^-])$.)

Applying $[\cdot,X_{\theta}^-]$ to \eqref{tkm} and using again \eqref{m} gives \begin{equation} [K(H_{\theta}),Q(X_k^-)] = P([H_{\theta},X_k^-]) + \frac{\lambda}{4} \sum_{\al\in\Delta} S([H_{\theta},X_{\al}],[X_{-\al},X_k^-]). \label{r3}
\end{equation}

Combining \eqref{r3} with \eqref{r2} yields
\begin{equation}
[K(X_{\theta}^-),Q(X_{\theta-\al_k})] = P([X_{\theta}^-,X_{\theta-\al_k}]) -  \frac{(\theta,\al_k-\theta)\lambda}{4} S(X_{\theta}^-,X_{\theta-\al_k}) + \frac{\lambda}{4} \sum_{\al\in\Delta} S([X_{\theta}^-,X_{\al}],[X_{-\al},X_{\theta-\al_k}]).
\label{r4}
\end{equation}
To obtain this relation, it was necessary to determine that $[X_{\theta}^-,X_{\theta-\al_k}] =  (\theta, \alpha_k) X_{k}^-$. $[X_{\theta}^-,X_{\theta-\al_k}]$ is a scalar multiple of $X_{k}^-$ with scalar given by $([X_{\theta}^-,X_{\theta-\al_k}],X_{k}^+)$ and 
\[ ([X_{\theta}^-,X_{\theta-\al_k}],X_{k}^+)  = - (X_{\theta-\al_k},[X_{\theta}^-,X_{k}^+])  =  (\theta, \alpha_k) (X_{\theta-\al_k},X_{\theta-\al_k}^-) = (\theta, \alpha_k) = -d_0c_{0k}, \]
where the second equality follows from $[X_{\theta}^-,X_{k}^+]=-(\theta, \alpha_k)X_{\theta-\al_k}^-$, which can be checked as follows:
\[
([X_\theta^-, X_{k}^+],X_{\theta-\al_k})  = - (X_{k}^+,[X_{\theta}^-,X_{\theta-\al_k}]) = - (X_{k}^+,[X_{\theta}^-,[X_\theta^+, X_k^-]]) =  - (X_{k}^+,[[X_{\theta}^-,X_\theta^+],X_k^-]) = -(\theta, \alpha_k)(X_{k}^+, X_k^-) = -(\theta, \alpha_k).
\]
Let $\beta$ be any positive root different from $\theta-\al_k$. There exist simple roots $\beta_1,\ldots, \beta_{\ell}$ such that $\Big[ X_{\beta_{\ell}}^-,\big[ \cdots [X_{\beta_1}^-,X_{\theta-\al_k}] \cdots \big]\Big]$ is a root vector in $\mfg_{\beta}$; let's denote it by $\wt{X}_{\beta}$.
(This is true for $X_{\theta}$ instead of $X_{\theta-\al_k}$ because $\theta$ is the highest root of $\mfg$, but the only simple root $\beta_i$ such that $[X_{\beta_i}^-,X_{\theta}^+] \neq 0$ is $\beta_i=\al_k$, so we might as well start with $X_{\theta-\al_k}$.) Applying $[X_{\beta_i}^-,\cdot]$ to \eqref{r4} repeatedly, we obtain, after perhaps rescaling:  \begin{equation}
[K(X_{\theta}^-),Q(X_{\beta})] = P([X_{\theta}^-,X_{\beta}]) + \frac{(\theta,\beta)\lambda}{4} S(X_{\theta}^-,X_{\beta})  + \frac{\lambda}{4} \sum_{\al\in\Delta} S([X_{\theta}^-,X_{\al}],[X_{-\al},X_{\beta}]). \label{r5}
\end{equation}
Let's prove this by induction on $\ell$. We already know that this is true when $\ell=0$ by \eqref{r4}, so let's assume it's true for $\ell-1$. Set  $\wt{X}_{\wt{\beta}} = \Big[ X_{\beta_{\ell-1}}^-,\big[ \cdots [X_{\beta_1}^-,X_{\theta-\al_k}] \cdots \big]\Big]$, so $\wt{X}_{\beta} =[X_{\beta_{\ell}}^-,\wt{X}_{\wt{\beta}}]$ and, by the inductive assumption, \begin{equation*}
[K(X_{\theta}^-),Q(\wt{X}_{\wt{\beta}})] = P([X_{\theta}^-,\wt{X}_{\wt{\beta}}]) + \frac{(\theta,\wt{\beta})\lambda}{4} S(X_{\theta}^-,\wt{X}_{\wt{\beta}})  + \frac{\lambda}{4} \sum_{\al\in\Delta} S([X_{\theta}^-,X_{\al}],[X_{-\al},\wt{X}_{\wt{\beta}}]).
\end{equation*} We apply $[X_{\beta_{\ell}}^-,\cdot]$ to both sides and use the fact that $\beta = \wt{\beta} - \beta_{\ell}$ along with the following consequence of \eqref{m}:  \[ \sum_{\al\in\Delta} \big[ X_{\beta_{\ell}}^-, S([X_{\theta}^-,X_{\al}],[X_{-\al},\wt{X}_{\wt{\beta}}])\big] - \sum_{\al\in\Delta} S([X_{\theta}^-,X_{\al}],[X_{-\al},\wt{X}_{\beta}]) = - (\beta_{\ell},\theta)S(X_{\theta}^-, [X_{\beta_{\ell}}^- , \wt{X}_{\wt{\beta}}]) = - (\beta_{\ell},\theta)  S(X_{\theta}^-,\wt{X}_{\beta}). \] The relation \eqref{r5} now follows by induction, using that $\beta = \wt{\beta} - \beta_{\ell}$.

We still have to prove \eqref{r5} when $\beta$ is a negative root. This can be done by writing $X_{\beta}$ as $X_{\beta} = \Big[ X_{\beta_{\ell}}^-,\big[ \cdots [X_{\beta_1}^-,X_{i}^-] \cdots \big]\Big]$ for some simple roots $\al_i,\beta_1, \ldots, \beta_{\ell}$,  starting with relation \eqref{2rel} with $\beta_1=-\theta$, $\beta_2=-\al_i$ and applying successively $[X_{\beta_j}^-,\cdot], j=1\ldots, \ell$ to both sides. We have now proved that \eqref{r5} holds for any root $\beta$ of $\mfg$ different from $\theta$. 

Let $\beta_1$ be any long root for $\mfg$. Then there exists $w\in W$ such that $\beta_1 = w(-\theta)$.  Let $w=s_{i_1} s_{i_2} \cdots s_{i_{\ell}}$ be a decomposition of $w$ into a product of simple reflections and let's denote also by $w$ the corresponding operator in the adjoint representation of $\mfg$ for this choice of decomposition. Relation \eqref{2rel} when $\beta_1$ is a long root now follows by applying $w$ to \eqref{r5} with $\beta = w^{-1}(\beta_2)$.
Therefore, \eqref{2rel} is true when $\beta_1$ is a long root. 

It remains to deal with the case when $\beta_1$ is a short root and $\beta_1\neq -\beta_2$. Let's apply $[X_{\theta}^+, \cdot]$ to relation \eqref{r5} in the case that $\beta$ is any positive root different from $\theta$. Then, using \eqref{m}, 
\begin{equation}
[K(H_{\theta})  ,Q(X_\beta)]  =  P([H_\theta, X_\beta]) + \frac{\lambda}{4} \sum_{\al\in\Delta} S([H_{\theta},X_{\al}],[X_{-\al},X_\beta]). \label{KHQX}
\end{equation}

In the case that $\beta$ is a negative root with $\beta\neq -\theta$, we can apply $[\cdot, X_{-\theta}]$ to relation \eqref{2rel} with $\beta_1=\theta$, and $\beta_2=\beta$; the same argument as above shows that \eqref{KHQX} holds for any root $\beta$ such that $\beta\neq \pm \theta$.

We are in a situation similar to one we had before, so we can use the action of the Tits extension $\widetilde{W}$ of $W$ on $\g$ to obtain that, for any long root $\gamma$ and for any root $\beta\neq \pm\gamma$,
\begin{equation}\label{r8}
[K(H_{\gamma})  ,Q(X_\beta)]=P([H_\gamma, X_\beta])
 + \frac{\lambda}{4} \sum_{\al\in\Delta} S([H_{\gamma},X_{\al}],[X_{-\al},X_\beta]).
\end{equation}

For any root $\eta\in \Delta$ of $\mfg$, there exists a long root $\gamma\in \Delta$, such that $(\gamma, \eta)\neq 0$. For $\eta\neq -\beta$, we apply $[\cdot, X_{\eta}]$ to relation \eqref{r8} with such a $\gamma$ and let $a \in \C$ be such that $[X_\beta, X_\eta]=aX_{\beta+\eta}$. (If $\beta+\eta$ is not a root, then $X_{\beta+\eta}=0$ and we set $a=0$ also.) Then, using \eqref{m},
\begin{align}
 \big[ [K(H_{\gamma})  ,Q(X_\beta)], X_{\eta}\big] = {} & (\gamma, \eta)P([X_\eta, X_\beta])+aP([H_\gamma, X_{\beta+\eta}])
 + \frac{(\gamma, \eta)\lambda}{4} \sum_{\al\in\Delta}S([X_\eta,X_{\al}],[X_{-\al},X_\beta])\notag\\
 & {} +\frac{a\lambda}{4} \sum_{\al\in\Delta} S([H_{\gamma},X_{\al}],[X_{-\al},X_{\beta+\eta}])
  -\frac{(\gamma, \eta)(\beta, \eta)\lambda}{4} S(X_\eta, X_\beta)\label{r9} \;\; \text{ by \eqref{m}}.
\end{align}

On the other hand, if $\beta+\eta\neq \pm \gamma$, using \eqref{r8} with $\beta$ replaced by $\beta+\eta$, we get:
\begin{align}
\big[ [K(H_{\gamma})  ,Q(X_\beta)], X_{\eta}\big] & =(\gamma, \eta)[K(X_{\eta})  ,Q(X_\beta)]+a[K(H_{\gamma}), Q(X_{\beta+\eta})] \notag\\
& =(\gamma, \eta)[K(X_{\eta})  ,Q(X_\beta)]+aP([H_{\gamma}, X_{\beta+\eta}])+ \frac{a\lambda}{4} \sum_{\al\in\Delta} S([H_{\gamma},X_{\al}],[X_{-\al},X_{\beta+\eta}]). \label{r10}
\end{align}

Combining \eqref{r9} with \eqref{r10}, we see that, if $\beta+\eta\neq \pm \gamma$ and $\beta+\eta\neq 0$:
\begin{equation}\label{r11}
[K(X_{\eta})  ,Q(X_\beta)]=P([X_{\eta}, X_\beta]) - \frac{(\beta, \eta)\lambda}{4}S(X_\eta, X_\beta)+\frac{\lambda}{4}\sum_{\al\in\Delta}S([X_\eta,X_{\al}],[X_{-\al},X_\beta]).
\end{equation}
In order to obtain \eqref{r11} for any two roots $\eta$ and $\beta$ with $\eta\neq -\beta$, it was only necessary to choose a long root $\gamma$ such that $(\eta, \gamma)\neq 0$ and $\gamma \neq \pm (\beta+\eta)$. 

In conclusion, the relation \eqref{2rel} holds in $\msD(\g)$.

\section{Other useful relations}

In this short section, we establish some new relations among the generators of $\mfD(\g)$ which will be useful for our computations later.

\begin{lemma}\label{lem:K(h)Q(x)}
For any roots $\beta_1, \beta_2$, the following relations hold in the deformed double current algebra $\mfD(\g):$
\begin{equation}\label{rel:K(h)Q(x)}
[K(H_{\beta_1}), Q(X_{\beta_2})]=P([H_{\beta_1}, X_{\beta_2}])+ \frac{\la}{4} \sum_{\alpha\in\Delta} S([H_{\beta_1},X_{\alpha}],[X_{-\alpha},X_{\beta_2}]),
\end{equation}
and
\begin{equation}\label{rel:K(X)Q(h)}
[K(X_{\beta_1}), Q(H_{\beta_2})]=P([X_{\beta_1}, H_{\beta_2}])+ \frac{\la}{4} \sum_{\alpha\in\Delta} S([X_{\beta_1},X_{\alpha}],[X_{-\alpha},H_{\beta_2}]).
\end{equation}
\end{lemma}
\begin{proof}
In the case $\beta_1\neq -\beta_2$,  we can write
$[X_{\beta_2}, X_{-\beta_1}]=aX_{\beta_2-\beta_1}$ for some $a\in \mathbb{C}$. (If $\beta_2-\beta_1$ is not a root, $a=0$). Now applying $[\cdot, X_{-\beta_1} ]$ to relation \eqref{2rel} and using \eqref{m}, we then obtain:
\begin{multline*}
[K(H_{\beta_1})  ,Q(X_{\beta_2})] + [K(X_{\beta_1})  ,aQ(X_{\beta_2-\beta_1})] \\
=P([H_{\beta_1}, X_{\beta_2}])+P([X_{\beta_1}, aX_{\beta_2-\beta_1}])
- \frac{a(\beta_1,\beta_2-\beta_1)\la}{4}  S(X_{\beta_1}, X_{\beta_2-\beta_1})\\
+ \frac{\la}{4} \sum_{\alpha\in\Delta} S([H_{\beta_1},X_{\alpha}],[X_{-\alpha},X_{\beta_2}])
+ \frac{a\la}{4} \sum_{\alpha\in\Delta} 
S([X_{\beta_1},X_{\alpha}],[X_{-\alpha},X_{\beta_2-\beta_1}]).
\end{multline*}
Relation \eqref{2rel} gives the following:
\[
[K(X_{\beta_1})  ,Q(X_{\beta_2-\beta_1})]=P([X_{\beta_1}, X_{\beta_2-\beta_1}])- \frac{(\beta_1,\beta_2-\beta_1)\la}{4}  S(X_{\beta_1},X_{\beta_2-\beta_1}) + \frac{\la}{4} \sum_{\alpha\in\Delta} S([X_{\beta_1},X_{\alpha}],[X_{-\alpha},X_{\beta_2-\beta_1}]).
\]
Combining the above calculations, we obtain \eqref{rel:K(h)Q(x)} in the case $\beta_1\neq -\beta_2$. The general case for \eqref{rel:K(h)Q(x)} follows from linearity of the factor $H_{\beta_1}$. (If $\beta_1 = \beta_2$, we can write $\beta_1 = \beta + \wt{\beta}$ with $\beta\neq \beta_1$ and $\wt{\beta}\neq \beta_1$.) The second relation of this lemma follows from the first one using the automorphism in Proposition \ref{Prop: auto}.
\end{proof}

\begin{lemma}\label{lem:HH}
For any roots $\beta_1, \beta_2$ such that $(\beta_1, \beta_2)=0$, the following relation holds in $\mfD(\g):$
\begin{equation*}
[K(H_{\beta_1}), Q(H_{\beta_2})] =\frac{\la}{2} \sum_{\alpha\in\Delta^+}(\beta_1, \alpha)(\beta_2, \alpha) S(X_{\alpha}, X_{-\alpha}).
\end{equation*}
\end{lemma}
\begin{proof}
The assumption $(\beta_1, \beta_2)=0$ implies that $[H_{\beta_1},X_{\beta_2}^-]=0$, so applying $[\cdot, X_{\beta_2}^-]$ to \eqref{rel:K(h)Q(x)} and using \eqref{m} yields the desired relation:
\begin{equation*}
[K(H_{\beta_1}) , Q(H_{\beta_2})]  =  P([H_{\beta_1},H_{\beta_2}]) + \frac{\la}{4} \sum_{\alpha\in\Delta} \Big[ S([H_{\beta_1},X_{\alpha}],[X_{-\alpha},X_{\beta_2}]), X_{\beta_2}^- \Big] = \frac{\la}{4} \sum_{\alpha\in\Delta} S([H_{\beta_1},X_{\alpha}],[X_{-\alpha},H_{\beta_2}])  \text{ by \eqref{m}}.
\end{equation*}
\end{proof}

\section{A central element of \texorpdfstring{$\mfD(\mfg)$}{D(g)} }\label{sec: central element}
In \cite{TLY1}, V. Toledano Laredo and the second author generalized the universal Knizhnik--Zamolodchikov--Bernard (KZB) connection $\nabla_{\KZB}$ in \cite{CEE} to any finite root system $\Phi$. This connection is valued in a holonomy Lie algebra. The elliptic Casimir connection \cite{TLY2} is a flat connection on the regular locus of the elliptic configuration space associated to $\Phi$ with values in a deformed double current algebra. It is obtained from $\nabla_{\KZB}$ via a homomorphism from the holonomy Lie algebra to $\mfD(\g)$. This construction requires a certain central element in $\mfD(\g)$. We introduce it in this section and show that it is indeed central (see Theorem \ref{central ele any g}). Actually, what we do is introduce various elements in $\mfD(\g)$, show that they are all scalar multiples of one another and then prove that they are central.

Set
\[
C(\beta_1, \beta_2)=[K(H_{\beta_1}), Q(H_{\beta_2})]-\frac{\la}{4} \sum_{\alpha\in\Delta} S([H_{\beta_1},X_{\alpha}],[X_{-\alpha},H_{\beta_2}]),\; C(\beta)=C(\beta, \beta),
\]
and set
\[
B(\beta)= [K(X_{\beta})  ,Q(X_{-\beta})]-P(H_\beta)- \frac{(\beta,\beta) \la}{4}  S(X_{\beta},X_{-\beta}) - \frac{\la}{4} \sum_{\alpha\in\Delta} S([X_{\beta},X_{\alpha}],[X_{-\alpha},X_{-\beta}]).
\]
\begin{proposition}\label{prop:central}
\begin{romenum}
  \item \label{prop: CB ele 1} The following equalities hold:
  \[
  C(\beta_1, \beta_2) = (\beta_1,\beta_2) B(\beta_2)=C(\beta_2, \beta_1),  \,\ \hbox{for any roots $\beta_1,\beta_2 \in \Delta$.}
\]
  In particular, when $\beta_1 = \beta_2 = \beta$, we obtain $\frac{C(\beta)}{(\beta, \beta)}=B(\beta)$.
  \item \label{prop: CB ele 2} For any two roots $\alpha, \beta$ in $\Delta$, we have
  \[
  \frac{C(\alpha)}{(\alpha, \alpha)}=\frac{C(\beta)}{(\beta, \beta)}.
  \]
\end{romenum}
\end{proposition}
\begin{proof}
To show \eqref{prop: CB ele 1},
apply $[X_{\beta_2}, \cdot]$ to the relation
\[
[K(H_{\beta_1}), Q(X_{-\beta_2})]=-(\beta_1, \beta_2)P(X_{-\beta_2})+\frac{\la}{4}\sum_{\alpha\in\Delta}S([H_{\beta_1}, X_{\alpha}], [X_{-\alpha}, X_{-\beta_2}]) \text{ (see \eqref{rel:K(h)Q(x)}).}
\] We then get, using \eqref{m}:
\begin{align*}
[K(H_{\beta_1}), Q(H_{\beta_2})]-(\beta_1, \beta_2)[K(X_{\beta_2}), Q(X_{-\beta_2})] = {} & -(\beta_1, \beta_2)P(H_{\beta_2})
+\frac{\la}{4}\sum_{\alpha\in\Delta}S([H_{\beta_1}, X_{\alpha}], [X_{-\alpha}, H_{\beta_2}]) \\
& -\frac{(\beta_1,\beta_2)\la}{4} \left(\sum_{\alpha\in\Delta} S([X_{\beta_2}, X_{\alpha}], [X_{-\alpha}, X_{-\beta_2}]) -
(\beta_2,\beta_2)S(X_{\beta_2}, X_{-\beta_2}) \right). \\
\end{align*}
Rewriting the above equality, we obtain  $C(\beta_1, \beta_2)= (\beta_1,\beta_2) B(\beta_2)$ for any roots $\beta_1,\beta_2\in \mfh^*$.
The proof of the second equality in claim \eqref{prop: CB ele 1} is similar.

Let us now turn to the proof of claim \eqref{prop: CB ele 2}. For any roots $\alpha, \beta$ such that $(\alpha, \beta)\neq 0$, by \eqref{prop: CB ele 1}, we have
\[
C(\alpha, \beta)=(\alpha, \beta) B(\beta) \,\ \hbox{and $C(\alpha, \beta)=(\beta, \alpha) B(\alpha)$}.
\]
Thus, $B(\alpha)=B(\beta)$ and $ \frac{C(\alpha)}{(\alpha, \alpha)}=\frac{C(\beta)}{(\beta, \beta)}$  if $(\alpha, \beta)\neq 0$.

Since the Dynkin diagram of the simple Lie algebra $\mathfrak g$ is connected, it follows that $ \frac{C(\alpha_i)}{(\alpha_i, \alpha_i)}$ is the same constant when $\alpha_i$ is a simple root. Now for any root $\alpha$, there exists a simple root $\alpha_i$, such that $(\alpha, \alpha_i)\neq 0$. This implies that $ \frac{C(\alpha)}{(\alpha, \alpha)}=\frac{C(\beta)}{(\beta, \beta)}$ for any two roots $\alpha, \beta$. 
\end{proof}

\begin{theorem}\label{central ele any g}
For any root $\beta$ of $\g$,  $C(\beta)$ is a central element of the algebra $\mfD(\mfg)$.
\end{theorem}
It suffices to show that $C(\beta)$ commutes with the generators of $\mfD(\g)$. This will follow from the next two lemmas.
\begin{lemma}\label{lem:C commu X}
For any $X\in \mfg$, we have $[C(\beta), X]=0$.
\end{lemma}
\begin{proof}
It suffices to take $X$ to be a root vector $X_\gamma$ for a root $\gamma$. We then have:
\begin{align*}
\big[ [K(H_{\beta}), Q(H_{\beta})], X_\gamma\big]
= {} & (\beta, \gamma)[K(X_{\gamma}), Q(H_{\beta})]+(\beta, \gamma)[K(H_{\beta}), Q(X_{\gamma})]\\
= {} & (\beta, \gamma)P([X_\gamma, H_\beta])+\frac{(\beta, \gamma)\la}{4}\sum_{\alpha\in \Delta}S([X_\gamma, X_\alpha], [X_{-\alpha}, H_\beta]) \\
&+(\beta, \gamma)P([H_\beta, X_\gamma])+\frac{(\beta, \gamma)\la}{4}\sum_{\alpha\in \Delta}S([H_\beta, X_\alpha], [X_{-\alpha}, X_\gamma]) \;\;   \text{ by \eqref{rel:K(h)Q(x)} and \eqref{rel:K(X)Q(h)}}\\
= {} & \frac{\la}{4}\sum_{\alpha\in \Delta}S\Big(\big[ [H_\beta, X_\gamma], X_\alpha\big] , [X_{-\alpha}, H_\beta]\Big)+\frac{\la}{4}\sum_{\alpha\in \Delta}S\Big( [H_\beta, X_\alpha], \big[ X_{-\alpha}, [H_\beta, X_\gamma]\big]\Big)\\
= {} &\frac{\la}{4}\sum_{\alpha\in \Delta} \Big[ S([H_\beta, X_\alpha], [X_{-\alpha}, H_\beta]), X_\gamma\Big] \;\;   \text{ by \eqref{m}}.
\end{align*}
Hence, the conclusion follows.
\end{proof}
\begin{lemma}\label{lem:C and K(H)}
If the two roots $\beta, \gamma$ are such that $(\beta, \gamma)=0$, then $[C(\beta), K(H_\gamma)]=0 = [C(\beta), Q(H_\gamma)]$.
\end{lemma}
\begin{proof}
Since $[K(H_\beta), K(H_\gamma)]=0$,
\begin{align*}
\big[ K(H_\gamma), [K(H_\beta), Q(H_\beta)]\big] & = \big[ K(H_\beta), [K(H_\gamma), Q(H_\beta)]\big] = \left[ K(H_\beta), \frac{\la}{4}\sum_{\alpha\in \Delta}(\gamma, \alpha)(\beta, \alpha)S(X_\alpha, X_{-\alpha})\right] \text{ by Lemma \ref{lem:HH}}\\
& =\frac{\la}{4}\sum_{\alpha\in \Delta}(\gamma, \alpha)(\beta, \alpha)^2 \Big( S(K(X_\alpha), X_{-\alpha})
-S( X_\alpha, K(X_{-\alpha}))\Big) = \frac{\la}{2}\sum_{\alpha\in \Delta}(\gamma, \alpha)(\beta, \alpha)^2S(K(X_\alpha), X_{-\alpha}).
\end{align*}
On the other hand,
\begin{align*}
\Bigg[ K(H_\gamma), \frac{\la}{4}\sum_{\alpha\in \Delta} & S([H_\beta, X_\alpha], [X_{-\alpha}, H_\beta]) \Bigg] = \left[ K(H_\gamma), \frac{\la}{4}\sum_{\alpha\in \Delta}(\beta, \alpha)^2S(X_\alpha, X_{-\alpha}) \right] \text{ by Lemma \ref{lem:HH}}\\
& = \frac{\la}{4}\sum_{\alpha\in \Delta}(\gamma, \alpha)(\beta, \alpha)^2 \Big( S(K(X_\alpha), X_{-\alpha})
-S(X_\alpha, K(X_{-\alpha})) \Big) = \frac{\la}{2}\sum_{\alpha\in \Delta}(\gamma, \alpha)(\beta, \alpha)^2S(K(X_\alpha), X_{-\alpha}).
\end{align*}
The assertion follows from these computations.  The proof that $[C(\beta), Q(H_\gamma)]=0$ is analogous.
\end{proof}
\begin{proof}[Proof of Theorem \ref{central ele any g}]
By Lemma \ref{lem:C and K(H)} and Proposition \ref{prop:central}, $C(\beta)$ commutes with $K(H_{\wt{\gamma}})$ for at least one root $\wt{\gamma}$ (since $\mfg$ has rank at least $\ge 3$ by assumption). Moreover, for any $X$ in $\mfg$,  $K(X)$ can be obtained by applying $\mathrm{ad}(X_1) \circ \mathrm{ad}(X_{2}) \circ \cdots \circ \mathrm{ad}(X_{l})$ to $K(H_{\wt{\gamma}})$ for certain elements $X_1,X_2,\ldots, X_l \in \mfg$. Therefore, $C(\beta)$ commutes with $K(X) \, \forall \, X\in\mfg$. 

Similar arguments show that $[C(\beta), Q(X)]=0$ for any $X\in \mfg$. $\mfD(\g)$ is generated by $X, K(X), Q(X)$ for all $X\in \mfg$: therefore, $C(\beta)$ is central in $\mfD(\g)$.
\end{proof}

\section{The deformed double current algebra of type A with two parameters}
\label{Sec: sl_n two par}
The deformed double current algebra of type $A$ with two parameters was introduced in \cite{G2} Definition 12.1. In this section, we obtain in this case results similar to those in the previous two sections - see Lemma \ref{lem:KEab QHcd} and Theorem \ref{central ele sln}. We also establish one connection with the Yangian of $\mfsl_n$ in Theorem \ref{PYang}. 

As usual in type $A$, $\mfh$ denotes the subspace of diagonal matrices of trace zero. Let  $\{\epsilon_1, \dots, \epsilon_n\}$ be the standard orthonormal basis of $\mathbb{C}^n$. The set $\Delta$ of roots of $\mathfrak{sl}_n$ can be identified with $\{\epsilon_i-\epsilon_j \mid 1\leq i\neq j\leq n\}$, with a choice of positive roots $\Delta^+$ given by $\{\epsilon_i-\epsilon_j \mid 1\leq i < j\leq n\}$ as usual. For $i\neq j$, set $\eps_{ij} = \eps_i - \eps_j$. The longest positive root $\theta$ equals $\epsilon_{1n}$. The elementary matrices will be written as $E_{ij}\in \mathfrak{sl}_n$, so $X_i^+=E_{i, i+1}$, $X_i^-=E_{i+1, i}$ and $H_i = E_{ii}-E_{i+1,i+1}$ for $1\leq i \leq n-1.$ We will assume in this section and the next that $n\ge 4$.
\begin{definition}
Let $\lambda, \beta\in \C$. We define $\mfD_{\lambda, \beta}(\mathfrak{sl}_n)$ to be the $\C$-algebra generated by elements
$X, K(X), Q(X), P(X)$ for $X\in \mathfrak{sl}_n$ subject to the following relations:
\begin{itemize}
\item The assignment $X \mapsto X, \, X\otimes v \mapsto K(X)$ (resp. $X \mapsto X, \,X\otimes u \mapsto Q(X)$) extends to an algebra homomorphism $\mfU(\mfsl_n[v]) \lra \mfD_{\lambda, \beta}(\mathfrak{sl}_n)$ (resp. $\mfU(\mfsl_n[u]) \lra \mfD_{\lambda, \beta}(\mathfrak{sl}_n)$);
\item $P(X)$ is linear in $X$, and for any $X, X'\in \mfsl_n$, $[P(X), X']=P([X, X'])$.
\end{itemize}
Moreover, for $a\neq b$, $c\neq d$, and $(a, b) \neq (d, c)$, the following relation holds:
\begin{align}\label{rel:two param}
[K(E_{ab}), Q(E_{cd})]=P([E_{ab}, E_{cd}])+\left(\beta-\frac{\lambda}{2}\right)(\delta_{bc}E_{ad}+&\delta_{ad}E_{cb})-
\frac{\lambda}{4}(\epsilon_{ab}, \epsilon_{cd})S(E_{ab}, E_{cd})\\
&+\frac{\lambda}{4}\sum_{1\leq i \neq j\leq n}S([E_{ab}, E_{ij}], [E_{ji}, E_{cd}]).\notag
\end{align}
\end{definition}

When $\beta=\frac{\lambda}{2}$, the relation \eqref{rel:two param} coincides with relation \eqref{2rel} of the deformed double current algebra $\mfD(\mathfrak{sl}_n)$ in Definition \ref{ddca double loop}.

We first list some relations of $\mfD_{\lambda, \beta}(\mathfrak{sl}_n)$ which are parallel to those found in Lemma \ref{lem:K(h)Q(x)}. The proof for $\mfD_{\lambda, \beta}(\mathfrak{sl}_n)$ is similar, so we omit it. In particular, the second relation follows from the first one using the automorphism of Proposition \ref{Prop: auto}.
\begin{lemma}\label{lem:KEab QHcd}
Let $H_{ab}=E_{aa}-E_{bb}$, for $a, b\in \mathbb{N}$. 
For any $a\neq b$, and $c\neq d$, the following relations hold in the algebra $\mfD_{\lambda, \beta}(\mathfrak{sl}_n)$:
\[[K(E_{ab}), Q(H_{cd})]=P([E_{ab}, H_{cd}])+\frac{\lambda}{4}\sum_{1\leq i \neq j\leq n}S([E_{ab}, E_{ij}], [E_{ji}, H_{cd}])+\left(\beta-\frac{\lambda}{2}\right)(\epsilon_a+\epsilon_b, \epsilon_{cd})E_{ab}.\]
and
\[[K(H_{ab}), Q(E_{cd})]=P([H_{ab}, E_{cd}])+\frac{\lambda}{4}\sum_{1\leq i \neq j\leq n}S([H_{ab}, E_{ij}], [E_{ji}, E_{cd}])+\left(\beta-\frac{\lambda}{2}\right)(\epsilon_{ab}, \epsilon_c+\epsilon_d)E_{cd}.\]
In particular,
$[K(E_{ab}), Q(H_{ab})]=-2P(E_{ab})+\frac{\lambda}{4}\sum_{1\leq i \neq j\leq n}S([E_{ab}, E_{ij}], [E_{ji}, H_{ab}]).$
\end{lemma}
In \cite{G2}, in the definition of $\mfD_{\lambda, \beta}(\mathfrak{sl}_n)$, the generators $P(X)$ were also imposed the condition that they had to satisfy the defining relations of the Yangian of $\mfsl_n$ as given in \cite{Dr1} in terms of elements $X, J(X) ,\, X\in\mfsl_n$.  It turns out that this is not necessary as explained in Theorem \ref{PYang} below. For the proof of that theorem, we will need certain elements $W_{ab}$, so we will now introduce these and a related central element (see Theorem \ref{central ele sln} below) similar to the one obtained in the previous section in the one-parameter case.

\subsection{Some central elements in $\mfD_{\lambda, \beta}(\mathfrak{sl}_n)$}

Set
\begin{equation}
Z_{ab, cd}=[K(H_{ab}), Q(H_{cd})]-\frac{\lambda}{4}\sum_{1\leq i \neq j\leq n}S([H_{ab}, E_{ij}], [E_{ji}, H_{cd}]). \label{defz}
\end{equation}
and denote $Z_{ab, ab}$ by $Z_{ab}$.
Set
\begin{equation}
W_{ab}=[K(E_{ab}), Q(E_{ba})]-P(H_{ab})-\frac{\lambda}{4}\sum_{1\leq i \neq j\leq n}S([E_{ab}, E_{ij}], [E_{ji}, E_{ba}])
-\frac{\lambda}{2}S(E_{ab}, E_{ba}). \label{defw}
\end{equation}
The following proposition is parallel to Proposition \ref{prop:central} and its proof is also similar, so we omit it.
\begin{proposition}\label{prop:two param} The following relations hold in $\mfD_{\la,\beta}(\mathfrak{sl}_n)$, $n\ge 4$:
\begin{romenum}
  \item \label{prop:two param item 1}
  For any $1\leq a \neq b \leq n$ and $1\leq c\neq b\leq n$, 
\[
  Z_{ab, cd}=(\epsilon_{ab}, \epsilon_{cd})W_{ab}+ \left(\beta-\frac{\lambda}{2}\right)(\epsilon_a+\epsilon_b, \epsilon_{cd})H_{ab}
  =(\epsilon_{ab}, \epsilon_{cd})W_{cd}+ \left(\beta-\frac{\lambda}{2}\right)(\epsilon_c+\epsilon_d, \epsilon_{ab})H_{cd}.
\]
In particular, we have $Z_{ab}=2 W_{ab}$, and when $a, b, c, d$ are distinct, $Z_{ab, cd}=0$.
\item \label{prop:two param item 2}
For $1\leq a\neq b \leq n$ and $1\leq c \neq d \leq n$,
\[
W_{ab}-W_{cd}=\left(\beta - \frac{\la}{2}\right)(H_{ac}+H_{bd}). \;\;\; (\text{If $a=c$, then $H_{ac}=0$.})
\]
\item \label{prop:two param item 3}
For $1\leq a\neq b \leq n$ and $1\leq c \neq d \leq n$,
\[Z_{ab}-Z_{cd}=2\left(\beta -\frac{\la}{2}\right) (H_{ac}+H_{bd}).\]
\end{romenum}
\end{proposition}

\begin{theorem}\label{central ele sln}
Set
\[
Z=\sum_{a=1}^{n}Z_{a, a+1}=\sum_{a=1}^{n}\left([K(H_{a}), Q(H_{a})]-\frac{\lambda}{4}\sum_{1\leq i \neq j\leq n}S([H_{a}, E_{ij}], [E_{ji}, H_{a}]) \right),
\]
where $H_a=E_{aa}-E_{a+1, a+1}$ when $1\leq a\leq n-1$, and $H_n=E_{nn}-E_{11}$.  
 The element $Z$ is central in $\mfD_{\lambda,\beta}(\mathfrak{sl}_n)$.
\end{theorem}
\begin{proof}
We first show that $Z$ commutes with all the elementary matrices $E_{cd}$ in $\mathfrak{sl}_n$. We have
\begin{equation}
[Z_{ab}, E_{cd}]=2\left(\beta-\frac{\lambda}{2}\right)(\epsilon_{ab}, \epsilon_{cd})(\epsilon_c+\epsilon_d, \epsilon_{ab})E_{cd}, \label{ZE}
\end{equation} since
\begin{align*}
\big[ [K(H_{ab}), & Q(H_{ab})], E_{cd}\big]\\ 
 = {} & (\epsilon_{ab}, \epsilon_{cd})[K(E_{cd}), Q(H_{ab})]+(\epsilon_{ab}, \eps_{cd})[K(H_{ab}), Q(E_{cd})]\\
 = {} & (\eps_{ab}, \eps_{cd})
\left(P([E_{cd}, H_{ab}])+\frac{\lambda}{4}\sum_{1\leq i \neq j\leq n}S([E_{cd}, E_{ij}], [E_{ji}, H_{ab}])+\left(\beta-\frac{\lambda}{2}\right)(\epsilon_c+\epsilon_d, \eps_{ab})E_{cd}\right) \\
& +(\eps_{ab}, \eps_{cd})
\left(P([H_{ab}, E_{cd}])+\frac{\lambda}{4}\sum_{1\leq i \neq j\leq n}S([H_{ab}, E_{ij}], [E_{ji}, E_{cd}])
+\left(\beta-\frac{\lambda}{2}\right)(\eps_{ab}, \epsilon_c+\epsilon_d)E_{cd}\right) \\
& \qquad  \qquad  \qquad  \qquad \text{ by Lemma \ref{lem:KEab QHcd}}\\
 = {} & \frac{\lambda}{4}\sum_{1\leq i \neq j\leq n}\Big[ S([H_{ab}, E_{ij}], [E_{ji}, H_{ab}]), E_{cd}\Big]
+2\left(\beta-\frac{\lambda}{2}\right)(\eps_{ab}, \eps_{cd})(\epsilon_c+\epsilon_d, \eps_{ab})E_{cd}.
\end{align*}
Thus,
\[
[Z, E_{i, i+1}]=\sum_{a=1}^n [Z_{a, a+1}, E_{i, i+1}]=2\left(\beta-\frac{\lambda}{2}\right)\sum_{a=1}^n(\epsilon_{a,a+1}, \epsilon_{i,i+1})(\epsilon_i+\epsilon_{i+1}, \epsilon_{a,a+1})E_{i, i+1} = 0 \text{ for any $1\leq i < n$.}\]
Similarly, it can be shown that $[Z, E_{i+1, i}]=0$ for any $1\leq i < n$, thus $Z$ commutes with any element in $\mathfrak{sl}_n$.

In the following, we show that $Z$ commutes with $K(H_{cd})$ for some diagonal matrix $H_{cd}$ of $\mathfrak{sl}_n$. Since $[K(h), h']=0$ for any two diagonal matrices $h, h'$, it follows from Proposition \ref{prop:two param} \eqref{prop:two param item 3} that \begin{equation} [Z_{ab}, K(h)]=[Z_{cd}, K(h)] \text{ for any  diagonal matrix $h$.} \label{ZKh} \end{equation}

Consider four distinct integers $a,b,c,d$ in $\{ 1,2,\ldots, n \}$. Let's check that $[Z_{ab}, K(H_{cd})]=0$. 
On the one hand,
\begin{align*}
\big[ K(H_{cd}), &  [K(H_{ab}), Q(H_{ab})]\big] = \big[ K(H_{ab}), [K(H_{cd}), Q(H_{ab})]\big] \\
& = \frac{1}{4}\sum_{1\leq i ,j \leq n} \left[ K(H_{ab}), (\eps_{cd}, \eps_{ij})(\eps_{ab}, \eps_{ij})S(E_{ij}, E_{ji})\right] \text{ since $Z_{ab,cd}=0$ by Proposition \ref{prop:two param} \eqref{prop:two param item 1}}\\
& =\frac{1}{2}\sum_{1\leq i \neq j \leq n}(\eps_{cd}, \eps_{ij})(\eps_{ab}, \eps_{ij})^2
S(K(E_{ij}), E_{ji}).
\end{align*}
On the other hand, we have:
\begin{align*}
\left[ K(H_{cd}), \frac{1}{4}\sum_{1\leq i \neq j \leq n}  S([H_{ab}, E_{ij}], [E_{ji}, H_{ab}])\right] & = \frac{1}{4}\sum_{1\leq i \neq j \leq n} \left[ K(H_{cd}), (\eps_{ab}, \eps_{ij})^2S(E_{ij}, E_{ji})\right]\\
& = \frac{1}{2}\sum_{1\leq i \neq j \leq n}(\eps_{cd}, \eps_{ij})(\eps_{ab}, \eps_{ij})^2 S(K(E_{ij}), E_{ji}).
\end{align*}

Combining the above computations, we conclude that $[K(H_{cd}), Z_{ab}]=0$. This is true in particular when $b=a+1$ and $a,a+1,c,d$ are all distinct. Thus, using \eqref{ZKh}, we obtain: \[ [K(H_{cd}), Z]= n [K(H_{cd}), Z_{a,a+1}] =0. \] It follows from this and $[Z,X]=0 \; \forall \, X\in\mfsl_n$ that $[Z, K(X)]=0$ for all $X\in \mathfrak{sl}_n$.  

A similar argument shows that $[Z, Q(X)]=0$ for all $X\in \mathfrak{sl}_n$. $\mfD_{\lambda, \beta}(\mathfrak{sl}_n)$ is generated by $X, K(X), Q(X)$ for $X\in \mathfrak{sl}_n$, hence $Z$ must be a central element in $\mfD_{\lambda, \beta}(\mathfrak{sl}_n)$.
\end{proof}

\subsection{Connection with the Yangian of $\mfsl_n$}

In this subsection, we will prove that the elements $X,P(X)$ for all $X\in\mfsl_n$ satisfy the defining relations of the Yangian $Y(\mfsl_n)$ as given in \cite{Dr1} (with $J(X)$ replaced by $P(X)$). In order to prove this, it will be easier to use a simpler presentation of $Y(\mfsl_n)$ and the isomorphism given in \cite{Dr2} between the two presentations, which is recalled below.

\begin{theorem}\cite{Le}\label{thm:Yangpres}
The Yangian $Y(\mfsl_n)$ of $\mfsl_n$ is isomorphic to the $\C$-algebra generated by elements $X_{i,r}^{\pm}, H_{i,r}$ for $1\le i\le n-1$ and $r=0,1$ which satisfy the following relations for $1\le i,j\le n$: \begin{equation}
[H_{i,r},H_{j,s}] = 0, \;\; [H_{i,0},X_{j,s}^{\pm}] = \pm c_{ij} X_{j,s}^{\pm} \;\;  \forall\, r,s\in\{0,1\}, \label{YLe1}
\end{equation}
\begin{equation}
[X_{i,1}^{\pm},X_{j,0}^{\pm}] - [X_{i,0}^{\pm},X_{j,1}^{\pm}] = \pm\frac{\la c_{ij}}{2} S(X_{i,0}^{\pm},X_{j,0}^{\pm}), \;\; [H_{i,1},X_{j,0}^{\pm}] - [H_{i,0},X_{j,1}^{\pm}] = \pm\frac{\la c_{ij}}{2} S(H_{i,0},X_{j,0}^{\pm}), \label{YLe2}
\end{equation}
\begin{equation}
[X_{i,r}^+, X_{j,s}^-] = \delta_{ij} H_{i,r+s}\; \text{ for } r+s=0,1, \;\; \big[ [X_{i,1}^+,X_{i,1}^-],H_{i,1}\big]=0, \label{YLe3}
\end{equation}
\begin{equation}
ad(X_{i,0}^{\pm})^{1-c_{ij}}(X_{j,0}^{\pm}) = 0. \label{YLe4}
\end{equation}
\end{theorem}
\begin{remark}
The most complicated relation is the second one in  \eqref{YLe3}. It turns out that it is needed only when $n=2$: since we are assuming that $n\ge 4$, we will disregard it.
\end{remark}

The isomorphism between this presentation of $Y(\mfsl_n)$ and the one given in \cite{Dr1} in terms of generators $X,J(X)$ for all $X\in\mfsl_n$ sends $X_{i,1}^{\pm}$ to $J(X_i^{\pm}) - \la \omega_i^{\pm}$: see \cite{Dr2}.

\begin{theorem}\label{PYang}
Set $X_{i,0}^{\pm} = X_i^{\pm}, H_{i,0} = H_i$, $X_{i,1}^{\pm} = P(X_i^{\pm}) - \la \omega_i^{\pm}$ and $H_{i,1} = P(H_i) - \la \nu_i$. These elements of  $\mfD_{\lambda, \beta}(\mathfrak{sl}_n)$ satisfy the defining relations of the Yangian $Y(\mfsl_n)$ as given in Theorem \ref{thm:Yangpres}.
\end{theorem}

\begin{proof}
Most of the relations in Theorem \ref{thm:Yangpres} follow directly from the fact that $[P(X),X'] = P([X,X'])$ for all $X,X' \in\mfsl_n$ and the isomorphism given in the paragraph just before the theorem. The main difficulty is in showing that $[H_{i,1},H_{j,1}]=0$ for all $i,j\in\{ 1,2,\ldots,n-1\}$. 

The relation $[H_{i, 1}, H_{j, 1}]=0$ is equivalent to 
$[P(H_i), P(H_j)]+\la^2[\nu_i, \nu_j]=0$. 
Indeed, 
\[
[H_{i, 1}, H_{j, 1}]
=[P(H_i)-\lambda \nu_i, P(H_j)-\lambda \nu_j]
=[P(H_i), P(H_j)]+\lambda^2[\nu_i, \nu_j]
\]
since \begin{align*}
[\nu_i, P(H_j)] & = \left[ \frac{1}{4} \sum_{\beta>0} (\alpha_i, \beta)S(X_{\beta}^+, X_{\beta}^-)-\frac{H_i^2}{2}, P(H_j) \right]\\
& =\frac{1}{4} \sum_{\beta>0} (\alpha_i, \beta)S
 (P[X_{\beta}^+, H_j], X_{\beta}^-)
 +\frac{1}{4} \sum_{\beta>0} 
 (\alpha_i, \beta)S(X_{\beta}^+, P[X_{\beta}^-, H_j])\\
& =-\frac{1}{4} 
 \sum_{\beta>0} (\alpha_i, \beta)(\alpha_j, \beta)S
 (P(X_{\beta}^+), X_{\beta}^-)
 +\frac{1}{4} \sum_{\beta>0} 
 (\alpha_i, \beta) (\alpha_j, \beta)S(X_{\beta}^+, P(X_{\beta}^-))\\
& = \frac{1}{4} \sum_{\beta>0} 
 (\alpha_i, \beta) (\alpha_j, \beta)\Big(S(X_{\beta}^+, P(X_{\beta}^-))-
 S(P(X_{\beta}^+), X_{\beta}^-)
 \Big).
\end{align*}
By symmetry, one has $[\nu_i, P(H_j)]=[\nu_j, P(H_i)]$ as desired.

That the equality $[P(H_i), P(H_j)]+\la^2[\nu_i, \nu_j]=0$ holds is the content of Lemma \ref{lem:P(H)}, which in turn depends on Lemmas \ref{lem:nu} and \ref{PKPQ}; these can all be found below. \end{proof}

For any three elements $z_1,z_2,z_3$ of $\mfsl_n$, set \[
\{z_1, z_2, z_3\}=\frac{1}{24}\sum_{\sigma\in \mathfrak{S}_3} z_{\sigma(1)}z_{\sigma(2)}z_{\sigma(3)}. \text{ ($\mathfrak{S}_3$ is the symmetric group on 3 letters.)}
\]
\begin{lemma}\label{lem:nu}
If $1\leq i\neq j \leq n$, 
we have:
\begin{align}
[\nu_i, \nu_j]
= {} & -\left(\sum_{l=1}^n\{ E_{li}, E_{i, j+1}, E_{j+1, l}\}
-\sum_{k=1}^n\{ E_{ik}, E_{k, j+1}, E_{j+1, i}\}
-\sum_{l=1}^n\{ E_{l, i+1}, E_{i+1, j+1}, E_{j+1, l}\}
+\sum_{k=1}^n\{ E_{i+1, k}, E_{k, j+1}, E_{j+1, i+1}\}\right) \notag
\\
&+\left(
\sum_{l=1}^n\{ E_{li}, E_{i, j}, E_{j, l}\}
-\sum_{k=1}^n\{ E_{ik}, E_{k, j}, E_{j, i}\}
-\sum_{l=1}^n\{ E_{l, i+1}, E_{i+1, j}, E_{j, l}\}
+\sum_{k=1}^n\{ E_{i+1, k}, E_{k, j}, E_{j, i+1}\}
\right). \label{nuinuj}
\end{align}
\end{lemma}
\begin{proof}
Let $(x , y)=\hbox{tr}(xy)$ be the inner product of $\mathfrak{sl}_n$ giving by the trace. Recall the following defining relation of the Yangian $Y_{\la}(\mathfrak{sl}_n)$ \cite{Dr1}: \[
\big[ J(x), J([y, z])\big] + \big[ J(z), [J([x, y])]\big] + \big[ J(y), J([z, x])\big]
=\lambda^2\sum_{a,b,c}\Big([x, x_a], \big[ [y, x_b], [z, x_c]\big]\Big)\{x^a, x^b, x^c\},
\]
where  $\{x_a\}, \{x^a\}$ are dual bases of $\mathfrak{sl}_n$ with respect to the inner product $(\cdot , \cdot)$. 
The isomorphism between the two presentations of the Yangian \cite{Dr2} gives the equality:
\[
\la^2[\nu_i, \nu_j]=-[J(H_i), J(H_j)]
=-\la^2\sum_{a,b,c}([H_i, x_a], [[X_j^+, x_b], [X_j^-, x_c]])\{x^a, x^b, x^c\}.
\]
Tedious but straightforward computations show that \begin{equation*}
-\sum_{a,b,c}\Big([H_i, x_a], \big[ [X_j^+, x_b], [X_j^-, x_c]\big]\Big)\{x^a, x^b, x^c\}
\end{equation*}
is the right hand side of \eqref{nuinuj}. Complete computations can be found in Section \S\ref{appendix1}.
\end{proof}

We will need one more lemma to treat the case $j=i+1$ below. The second identities below follows from the first one via the anti-automorphism of $\mfD_{\lambda, \beta}(\mathfrak{sl}_n)$ given by $X \mapsto X^t$, $K(X) \mapsto Q(X^t)$, $Q(X) \mapsto K(X^t)$ and $P(X) \mapsto P(X^t)$.
\begin{lemma}\label{PKPQ}
The following identities hold in $\mfD_{\lambda, \beta}(\mathfrak{sl}_n)$ for $1\le i \le n-2$: 
\begin{alignat*}{2}
 [P(H_i),K(E_{i+1,i+2})] = {} & - \frac{1}{2} [P(H_{i+1}),K(E_{i+1,i+2})] 
 +\left(\beta-\frac{\la}{2}\right) & K(E_{i+1, i+2}) + \frac{\la}{4} \Big
 ( S(K(E_{i+1,i}),E_{i,i+2})  
 + S(E_{i+1,i},K(E_{i,i+2})) 
\Big)\\
& & +\frac{\la}{8} \sum_{\stackrel{p=1}{p\neq i+1, i+2}}^n \Bigg(
S(K(E_{p, i+2}), E_{i+1, p})
+ S(K(E_{i+1, p}), E_{p, i+2})
\Bigg)
 \end{alignat*} 
and
\begin{alignat*}{2}
[P(H_i),Q(E_{i+2,i+1})] = {} & - \frac{1}{2} [P(H_{i+1}), Q(E_{i+2,i+1})] 
 -\left(\beta-\frac{\la}{2}\right) & Q(E_{i+2, i+1}) - \frac{\la}{4} \Big
 ( S(Q(E_{i, i+1}),E_{i+2, i})  
 + S(E_{i, i+1},Q(E_{i+2, i})) 
\Big)\\
& & -\frac{\la}{8} \sum_{\stackrel{p=1}{p\neq i+1, i+2}}^n\Bigg(
S(Q(E_{i+2, p}), E_{p, i+1})
+ S(Q(E_{p, i+1}), E_{i+2, p})
\Bigg).
 \end{alignat*} 
\end{lemma}
\begin{proof}
We only show the relation of $ [P(H_i),K(E_{i+1,i+2})]$, the other one $[P(H_i),Q(E_{i+2,i+1})]$ follows from a similar argument. 
Set $R_{ab,cd} = P([E_{ab},E_{cd}]) - [K(E_{ab}),Q(E_{cd})]$ and $\wt{W}_{ab} = P(H_{ab}) - [K(E_{ab}),Q(E_{ba})] + W_{ab}$ for simplicity, both of which are in $\mfU\mfsl_n$.  
We have:
\begin{align*}
 [P(H_{i,i+2}),K(E_{i+1,i+2})] 
 & =  \big[ [K(E_{i,i+2}),Q(E_{i+2,i})] + \wt{W}_{i,i+2} - W_{i,i+2} ,K(E_{i+1,i+2}) \big] \\
& {} =  \big[ K(E_{i,i+2}),[Q(E_{i+2,i}),K(E_{i+1,i+2})]\big] + [ \wt{W}_{i,i+2} - W_{i,i+2} ,K(E_{i+1,i+2})]  \\
& {} =  -[K(E_{i,i+2}),P(E_{i+1,i}) - R_{i+1,i+2,i+2,i}] + [ \wt{W}_{i,i+2} - W_{i,i+2} ,K(E_{i+1,i+2})] \\
 & =  - \big[ [E_{i,k},K(E_{k,i+2})],P(E_{i+1,i}) \big] + [K(E_{i,i+2}),R_{i+1,i+2,i+2,i}]  + [ \wt{W}_{i,i+2} - W_{i,i+2} ,K(E_{i+1,i+2})], 
 \end{align*}
where $k, i, i+1, i+2$ are distinct.
 
 For the term $-\big[ [E_{i,k},K(E_{k,i+2})],P(E_{i+1,i}) \big]$, we use the following trick:
 \begin{align*}
 -\big[ [E_{i,k},K(E_{k,i+2})],P(E_{i+1,i}) \big]
 = {} &  [P(E_{i+1,k}),K(E_{k,i+2})]  
 -\big[E_{ik}, [K(E_{k, i+2}), P(E_{i+1, i})]\big]\\
= {}  & \big[ P(E_{i+1,k}),[E_{k,i+1},K(E_{i+1,i+2})] \big] -\big[E_{ik}, 
[K(E_{k, i+2}), P(E_{i+1, i})]\big] \\
 = {} &  [P(H_{i+1,k}),K(E_{i+1,i+2})] +\big[E_{k, i+1},  [P(E_{i+1, k}), K(E_{i+1, i+2})]\big]-\big[E_{ik}, [K(E_{k, i+2}), P(E_{i+1, i})]\big] \\
 = {} &  [P(H_{i+1,i}),K(E_{i+1,i+2})] + [P(H_{ik}),K(E_{i+1,i+2})] \\
 &+\big[E_{k, i+1},  [P(E_{i+1, k}), K(E_{i+1, i+2})]\big]-\big[E_{ik}, [K(E_{k, i+2}), P(E_{i+1, i})]\big].
\end{align*}
Therefore,
\begin{align*}
 [P(H_{i,i+2}),K(E_{i+1,i+2})] 
  = {} & [P(H_{i+1,i}),K(E_{i+1,i+2})] + [P(H_{ik}),K(E_{i+1,i+2})] +\big[E_{k, i+1},  [P(E_{i+1, k}), K(E_{i+1, i+2})]\big] \\  
 &  -\big[E_{ik}, [K(E_{k, i+2}), P(E_{i+1, i})]\big]  + [K(E_{i,i+2}),R_{i+1,i+2,i+2,i}]  + [ \wt{W}_{i,i+2} - W_{i,i+2} ,K(E_{i+1,i+2})].
 \end{align*}
We now compute:
\begin{align}
[P(H_{i+1,i+2}),K(E_{i+1,i+2})] 
 = {} &  [P(H_{i+1,i}),K(E_{i+1,i+2})] + [P(H_{i,i+2}),K(E_{i+1,i+2})] \notag \\
 = {} &  -2[P(H_{i}),K(E_{i+1,i+2})] 
 + [P(H_{ik}),K(E_{i+1,i+2})] -\big[E_{ik}, [K(E_{k, i+2}), P(E_{i+1, i})]\big] \notag \\
  & +\big[E_{k, i+1},  [P(E_{i+1, k}), K(E_{i+1, i+2})]\big] + [K(E_{i,i+2}),R_{i+1,i+2,i+2,i}]  + [ \wt{W}_{i,i+2} - W_{i,i+2} ,K(E_{i+1,i+2})]. \label{PHi1i2}
 \end{align}
To compute $[W_{i,i+2},K(E_{i+1,i+2})]$, we can proceed in the following way. By Proposition \ref{prop:two param}, $W_{i,i+2} = \frac{1}{2}Z_{ik} + \left( \beta - \frac{\la}{2} \right)H_{i+2,k}$, so we are reduced to finding $[Z_{ik},K(E_{i+1,i+2})]$. It is established in the proof of Theorem \ref{central ele sln} that $[Z_{ik},K(H_{i+1,i+2})]=0$ since $i,i+1,i+2,k$ are all distinct. It follows that $[Z_{ik},K(E_{i+1,i+2})]= 0$.

The rest of the proof follows from long, but direct computations which can be found in Section \S\ref{appendix2}.
\end{proof}

As explained earlier, to establish Theorem \ref{PYang}, it is enough to prove the last lemma of this section. 

\begin{lemma}\label{lem:P(H)}
For any $1\leq i\neq j\leq n$, we have:
\begin{equation*}
[P(H_i), P(H_j)]+\la^2[v_i, v_j]=0.
\end{equation*}
\end{lemma}
\begin{proof}
To establish Lemma \ref{lem:P(H)}, we compute $[P(H_i), P(H_j)]$ explicitly and show that it equals the right-hand side of \eqref{nuinuj}. We consider two cases: $| i-j |>1$ and $| i-j |=1$. 

\noindent \textbf{Case 1:} $| i-j |>1$.

When $| i-j |>1$, we have that $i, i+1, j, j+1$ are all distinct. Therefore, 
$[X_i^+, X_j^+]=0$ and $[X_i^-, X_j^-]=0$. Moreover,
\begin{equation*}
[P(H_i), P(H_j)]
=\big[ P([X_i^+, X_i^-]), P([X_j^+, X_j^-]) \big]
=\Big[\big[ [P(X_i^+), P(X_j^+)], X_j^-\big],  X_i^-\Big].
\end{equation*}
To compute $[P(X_i^+), P(X_j^+)]$, we use
the equality:
\[
[K(X_i^+), Q(H_i)]=-2 P(X_i^+) +\frac{\lambda}{4} 
\sum_{\alpha\in\Delta} S([X_i^+, X_{\alpha}], [X_{-\alpha}, H_i]) \;\; \text{ (see \eqref{rel:K(X)Q(h)}).}
\]
Then:
\begin{align}
[-2 P(X_i^+), P(X_j^+)]
& = \big[ [K(X_i^+), Q(H_i)], P(X_j^+)\big]
-\frac{\lambda}{4} 
\sum_{\alpha\in\Delta} \big[ S([X_i^+, X_{\alpha}], [X_{-\alpha}, H_i]), P(X_j^+)\big] \notag \\
& = \big[ [K(X_i^+),  P(X_j^+)], Q(H_i)\big]
+\big[ K(X_i^+), [Q(H_i), P(X_j^+)]\big]
-\frac{\lambda}{4} 
\sum_{\alpha\in\Delta} \big[ S([X_i^+, X_{\alpha}], [X_{-\alpha}, H_i]), P(X_j^+)\big]. \label{PXiXj}
\end{align}
We now need the following relations in $D_{\lambda, \beta}(\mathfrak{sl}_n)$ when $[E_{ab}, E_{cd}]=0$:
\begin{align*}
&[K(E_{ab}), P(E_{cd})]=-\frac{\la}{4}\Big(S(K(E_{ad}), E_{cb})+S(K(E_{cb}), E_{ad})\Big), \\
&[Q(E_{ab}), P(E_{cd})]=\frac{\la}{4}\Big(S(Q(E_{ad}), E_{cb})+S(Q(E_{cb}), E_{ad})\Big).
\end{align*}
The given relations were obtained in step 2 of \S\ref{appendix2}; the second one follows from the first one using the automorphism in Proposition \ref{Prop: auto}.
Using these and the defining relations of $D_{\lambda, \beta}(\mathfrak{sl}_n)$, we can compute $[P(X_i^+), P(X_j^+)]$ explicitly. The result is provided by equation \eqref{PXiXj2} in Section \S\ref{appendix3}:
\begin{equation*}
[ P(X_i^+), P(X_j^+)]
=\frac{\la^2}{16}
S\Bigg(
 \sum_{k=1}^n S(E_{k, j+1}, E_{i, k}), E_{j, i+1}\Bigg)
- \frac{\la^2}{16}
S\Bigg(
\sum_{l=1}^n S(E_{jl}, E_{l, i+1}), E_{i, j+1} \Bigg).
\end{equation*}
The conclusion that Lemma \ref{lem:P(H)} is true when $|i-j|>1$ follows directly by applying $\big[ [\cdot, X_j^-],  X_i^-\big]$ to the previous formula. The detailed computation is in Section \S\ref{appendix3}.

\noindent \textbf{Case 2:} $j=i+1$

Since $n\ge 4$, we can choose $k \in \{  1,2,\ldots,n\}$ such that $i,i+1,i+2,k$ are all distinct. Set $S_{i} = [P(H_i),K(E_{i+1,i+2})] + \frac{1}{2} [P(H_{i+1}),K(E_{i+1,i+2})]$ and $\wt{S}_i = [P(H_i),Q(E_{i+2,i+1})] + \frac{1}{2} [P(H_{i+1}),Q(E_{i+2,i+1})]$.  
We have:
\begin{align}
[P(H_i),P(H_{i+1})] = {} & \big[ P(H_i),[K(E_{i+1,i+2}),Q(E_{i+2,i+1})] + \wt{W}_{i+1,i+2} - W_{i+1,i+2} \big] \notag \\
 = {} &  \big[ [P(H_i),K(E_{i+1,i+2})],Q(E_{i+2,i+1})\big]  + \big[ K(E_{i+1,i+2}),[P(H_i),Q(E_{i+2,i+1})]\big] + [P(H_i),\wt{W}_{i+1,i+2} - W_{i+1,i+2}]  \notag \\
= {} & - \frac{1}{2} \big[ [P(H_{i+1}),K(E_{i+1,i+2})],Q(E_{i+2,i+1})\big] + [S_i,Q(E_{i+2,i+1})]  \notag  \\
&   {} - \frac{1}{2}  \big[ K(E_{i+1,i+2}),[P(H_{i+1}),Q(E_{i+2,i+1})]\big] + [K(E_{i+1,i+2}),\wt{S}_i] + [P(H_i),\wt{W}_{i+1,i+2} - W_{i+1,i+2}] \notag  \\
 = {} &  - \frac{1}{2} \big[ P(H_{i+1}),[K(E_{i+1,i+2}),Q(E_{i+2,i+1})] \big]  + [S_i,Q(E_{i+2,i+1})]  + [K(E_{i+1,i+2}),\wt{S}_i] \notag \\
&  + [P(H_i),\wt{W}_{i+1,i+2} - W_{i+1,i+2}]  \notag \\
= {}  &  - \frac{1}{2} \big[ P(H_{i+1}), P(H_{i+1}) + W_{i+1,i+2} - \wt{W}_{i+1,i+2} \big] + [S_i,Q(E_{i+2,i+1})]  + [K(E_{i+1,i+2}),\wt{S}_i]   \notag \\
&   {} + [P(H_i),\wt{W}_{i+1,i+2} - W_{i+1,i+2}]  \notag  \\
 = {} &   [S_i,Q(E_{i+2,i+1})]  + [K(E_{i+1,i+2}),\wt{S}_i] + [P(H_i) + \frac{1}{2} P(H_{i+1}) ,\wt{W}_{i+1,i+2} - W_{i+1,i+2}]. \label{PHiPHi1}
 \end{align}
To compute $[P(H_i), W_{i+1,i+2}]$, we can use that 
$W_{i+1,i+2} = W_{k,i+2} + \left(\beta - \frac{\la}{2}\right) H_{i+1,k}$, 
which follows from Proposition \ref{prop:two param} (ii),
so we are reduced to computing $[P(H_i), W_{k,i+2}]$: \begin{equation} [P(H_i), W_{i+1,i+2}] = [P(H_i), W_{k,i+2}] = \big[ [K(E_{i,i+1}),Q(E_{i+1,i})\big] + \wt{W}_{i,i+1} - W_{i,i+1} , W_{k,i+2}] \label{PHW}  \end{equation} and $[K(E_{i,i+1}), W_{k,i+2}]= 0$, $[Q(E_{i+1,i}), W_{k,i+2}]= 0$ (as shown above in the proof of Lemma \ref{PKPQ}). Moreover, by Proposition \ref{prop:two param}, \[ [W_{i,i+1} , W_{k,i+2}] = [W_{k,i+2} + \left(\beta - \frac{\la}{2}\right) (H_{ik} + H_{i+1,i+2} ) , W_{k,i+2}] = 0 \] and it follows from \eqref{ZE} that $[ \wt{W}_{i,i+1},W_{k,i+2}]=0$. Therefore, by \eqref{PHW}, $[P(H_i), W_{i+1,i+2}] = 0$, and $[P(H_{i+1}), W_{i+1,i+2}]=0$ can be obtained similarly. 

Since $\wt{W}_{i+1,i+2} \in \mfU\mfsl_n$, $[P(H_i) + \frac{1}{2} P(H_{i+1}) ,\wt{W}_{i+1,i+2}]$ can be computed directly. Using the explicit formula of  $S_i, \wt{S}_i$ in Lemma \ref{PKPQ}, it is possible to determine $[S_i,Q(E_{i+2,i+1})]  + [K(E_{i+1,i+2}),\wt{S}_i]$ and hence the right-hand side of \eqref{PHiPHi1}. The conclusion that $[P(H_i), P(H_{i+1})]+\la^2[v_i, v_{i+1}]=0$ follows from a direct computation which can be found in Section \S\ref{appendix4}.
\end{proof}

\section{The center of the deformed double current algebra \texorpdfstring{$\mfD_{\lambda, \beta}(\mfsl_n)$}{D(sl)}}

In this section, we show that, when $n\lambda=\pm 4\left(\beta-\frac{\lambda}{2}\right)$, the center of $\mfD_{\lambda, \beta}(\mfsl_n)$ is very large. This is made precise in Theorem \ref{thm:center=polynomial rings} below.

If $X \ot u^s$ is an element of $\mfsl_n[u]$, then we denote its image in $\mfD_{\lambda, \beta}(\mfsl_n)$ under $\mfsl_n[u] \lra \mfD_{\lambda, \beta}(\mfsl_n)$ also by $X \ot u^s$ or by $X(u^s)$. The same applies to $\mfsl_n[v] \lra \mfD_{\lambda, \beta}(\mfsl_n)$. Set

\begin{equation}  Z_{ab,cd}(s)=[K(H_{ab}), H_{cd}(u^s)]-\frac{\lambda}{4}\sum_{1\leq i\neq j\leq n}
 \sum_{p+q=s-1}S\big( [H_{ab}, E_{ij}](u^p), [E_{ji}, H_{cd}](u^q)\big), \; Z_{ab}(s) = Z_{ab,ab}(s) \label{Zabcd} \end{equation} and \begin{equation}
 \wt{Z}_{ab,cd}(s)=[Q(H_{ab}), H_{cd}(v^s)]+\frac{\lambda}{4}\sum_{1\leq i\neq j\leq n}
 \sum_{p+q=s-1}S\big( [H_{ab}, E_{ij}](v^p), [E_{ji}, H_{cd}](v^q)\big), \; \wt{Z}_{ab}(s) = \wt{Z}_{ab,ab}(s). \label{wtZabcd}
 \end{equation}

\begin{theorem}\label{thm:center=polynomial rings}
Assume that $n\ge 4$ and $n\lambda=\pm 4\left(\beta-\frac{\lambda}{2}\right)$. The center of the deformed double current algebra $\mfD_{\lambda, \beta}(\mathfrak{sl}_n)$ contains two subrings isomorphic to the polynomial rings in infinitely many variables $\C[a_1,a_2, a_3, \ldots]$ and $\C[b_1,b_2, b_3, \ldots]$; the isomorphism sends $a_s$ (resp. $b_s$) to $Z(s)$ (resp. $\wt{Z}(s)$) where $Z(s)=Z_{12}(s)+Z_{23}(s)+\cdots+Z_{n1}(s)$ (resp. $\wt{Z}(s)=\wt{Z}_{12}(s)+\wt{Z}_{23}(s)+\cdots+\wt{Z}_{n1}(s)$).
\end{theorem}

This theorem was inspired by an analogous result about rational Cherednik algebras, namely Proposition 3.6 in \cite{G} which states that the center of the rational Cherednik algebra when the parameter $t=0$ contains the subalgebra $\C[\h]^W\otimes \C[\h^*]^W$, $\h$ being here the reflection representation of the Weyl group $W$. In type $A$, the condition $t=0$ corresponds to the condition $n\lambda = - 4\left(\beta - \frac{\lambda}{2}\right)$ when the parameters $\la,\beta$ are related to the parameter $t,c$ of the rational Cherednik algebra as in \cite{G1,G2}, namely $\beta = \frac{t}{2} - \frac{c(n-2)}{4}$ and $\la=c$. These two conditions on the parameters are necessary for the construction of the Schur-Weyl functor in \textit{loc. cit.} If $M$ is a right module over a rational Cherednik algebra $\msH_{t,c}(S_l)$ of type $\mfgl_l$ ($S_l$ being the symmetric group on $l$ letters), then $M\ot_{\C[S_l]} (\C^n)^{\ot l}$ is a left module over $\mfD_{\lambda, \beta}(\mfsl_n)$. If $z$ is a central element of $\msH_{0,c}(S_l)$, then the  multiplication by $z$ on $M\ot_{\C[S_l]} (\C^n)^{\ot l}$ given by $m \ot \mbv \lra mz\ot \mbv$ produces an element $R_z$ of $\End_{\mfD_{\lambda, \beta}(\mfsl_n)}\left( M\ot_{\C[S_l]} (\C^n)^{\ot l} \right)$. Proposition 3.6 in \cite{G} states that the center of $\msH_{0,c}(S_l)$ contains $\C[x_1,\ldots, x_l]^{S_l} \ot_{\C} \C[y_1,\ldots, y_l]^{S_l}$.  If $z=x_1^s + x_2^s + \cdots + x_l^s$, then lengthy computations show that $R_z(m\ot\mbv) = Z(s)(m\ot\mbv)$. $Z(s)$ can also be viewed was an element of $\End_{\mfD_{\lambda, \beta}(\mfsl_n)}\left( M\ot_{\C[S_l]} (\C^n)^{\ot l}\right)$ because $Z(s)$ is central, so we can say that $R_z = Z(s)$ in $\End_{\mfD_{\lambda, \beta}(\mfsl_n)}\left( M\ot_{\C[S_l]} (\C^n)^{\ot l}\right)$. A similar remark applies to $y_1^s + y_2^s + \cdots + y_l^s$ and $\wt{Z}(s)$. 

\subsection{Higher commutation relations}

We will need the following proposition, which is also of independent interest.

\begin{proposition}\label{Ps}
For any $s\ge 0$ and any $X\in\mfsl_n$, there exists in $\mfD_{\lambda, \beta}(\mathfrak{sl}_n)$ an element $P_s(X)$ with the property that the assignment $X \mapsto P_s(X)$ is linear, $[P_s(X),X'] = P_s([X,X'])$ for any $X'\in\mfsl_n$, and such that, for any integers $1\leq a, b, c, d\leq n$ with $a\neq b$, $c \neq d$ and $(a, b)\neq (d, c)$, the following relation holds:
\begin{align}\label{rel:higher degree rel}
[K(E_{ab}), E_{cd}(u^s)] = {} & P_s([E_{ab}, E_{cd}]) + s\left(\beta-\frac{\lambda}{2}\right)(\delta_{bc}E_{ad}+\delta_{ad} E_{cb})(u^{s-1}) - \frac{(\epsilon_{ab}, \epsilon_{cd})\lambda}{4} \sum_{p+q=s-1}S(E_{ab}(u^p), E_{cd}(u^{q}))\notag\\
& + \frac{\lambda}{4}\sum_{1\leq i\neq j \leq n} \sum_{p+q=s-1}S\big( [E_{ab}, E_{ij}](u^p), [E_{ji}, E_{cd}](u^q)\big) .
\end{align}
\end{proposition}
\begin{remark}
When $s=0$, $P_0(X) = K(X)$ and the right-hand side equals $K([E_{ab}, E_{cd}])$.
\end{remark}

The following identities are essential for the proof of the previous proposition, which is given below. They are generalizations of \eqref{m} and can be obtained in the same way. As for \eqref{m}, let $\gamma$ be either a root and $X_{\gamma}$ a corresponding root vector, or let $\gamma=0$ and $X_{\gamma}$ be interpreted as an element of $\mfh$.  Then, for any roots $\beta_1, \beta_2$, we have:
\begin{align}
 \sum_{\alpha\in\Delta} \Big[ S\big( [X_{\beta_1},X_{\alpha}(u^p)], & [X_{-\alpha}(u^q), X_{\beta_2}]\big) , X_\gamma\Big] \notag \\
= {} & \sum_{\alpha\in\Delta} S\Big(\big[ [X_{\beta_1},X_\gamma],X_{\alpha}(u^p)\big],[X_{-\alpha}(u^q),X_{\beta_2}]\Big)+ \sum_{\alpha\in\Delta} S\Big( [ X_{\beta_1},X_{\alpha}(u^p)],\big[ X_{-\alpha}(u^q),[X_{\beta_2},X_\gamma]\big]\Big)  \notag \\
-&(\gamma, \beta_2)S\big([X_{\beta_1}, X_{\gamma}](u^p), X_{\beta_2}(u^q)\big)-(\gamma, \beta_1)S\big(X_{\beta_1}(u^p), [ X_{\beta_2}, X_{\gamma}](u^q)\big), \label{mpq}
\end{align}
\begin{align}
\sum_{\al\in\Delta} \sum_{p+q = s-1} \Big[ S\big( [X_{\beta_1}, X_{\al}(u^p)], & [X_{-\al}(u^q), X_{\beta_2}]\big), X_{\gamma}(u)  \Big] \notag \\
 = {} & \sum_{\al\in\Delta} \sum_{p+q = s} \left( S\Big(\big[[X_{\beta_1},X_{\gamma}], X_{\al}(u^p)\big],[X_{-\al}(u^q), X_{\beta_2}]\Big) + S\Big([X_{\beta_1}, X_{\al}(u^p)],\big[X_{-\al}(u^q), [X_{\beta_2}, X_{\gamma}]\big]\Big)  \right) \notag \\
& - \sum_{p+q = s} \left( (\gamma, \beta_2) S\big([X_{\beta_1}, X_{\gamma}](u^p),X_{\beta_2}(u^q)\big) +  (\gamma, \beta_1)S\big( X_{\beta_1}(u^p), [X_{\beta_2}, X_{\gamma}](u^q) \big) \right) \notag \\
& - \sum_{\al\in\Delta} \left( S\Big(\big[[X_{\beta_1}, X_{\al}], X_{\gamma}\big],[X_{-\al}, X_{\beta_2}](u^s)\Big) + S\Big( [X_{\beta_1}, X_{\al}](u^s),\big[ [X_{-\al}, X_{\beta_2}], X_{\gamma} \big]\Big) \right).  \label{mu}
\end{align}

\begin{proof}[Proof of Proposition \ref{Ps}]
Let's start with relation \eqref{rel:higher degree rel} in the case when $a,b,c$ and $d$ are all distinct. Applying $[Q(H_{cd}), \cdot]$ to $[K(E_{ab}), E_{cd}(u^{s-1})]$ and using Lemma \ref{lem:KEab QHcd} along with \eqref{mu} allow us to obtain, by induction on $s$, the formula for $[K(E_{ab}), E_{cd}(u^s)]$. The case when $[E_{ab},E_{cd}]=0$ follows from this. (For instance, if $a,b,c$ are all distinct and $d$ is chosen different from $a,b$ and $c$, then we have $[K(E_{ab}), E_{cb}(u^s)] = \big[ [K(E_{ab}), E_{cd}(u^s)], E_{db}\big]$ and the right-hand side can be computed using \eqref{mpq}.)

\noindent \textbf{Step 1.1:} The proof of the existence of $P_s(X)$ is by induction on $s$, the case $s=1$ being given by the definition of $\mfD_{\lambda, \beta}(\mathfrak{sl}_n)$. Let $s\ge 2$ and assume that $P_{s-1}(X)$ satisfying \eqref{rel:higher degree rel} is known to exists for all $X\in\mfsl_n$ and with the property that the assignment $X \mapsto P_{s-1}(X)$ is linear and $[P_{s-1}(X),X'] = P_{s-1}([X,X'])$ for any $X'\in\mfsl_n$. Set \begin{align} P_s(E_{n1}) = {} & [K(E_{n,n-1}),E_{n-1,1}(u^s)] - s\left(\beta - \frac{\la}{2}\right) E_{n1}(u^{s-1}) - \frac{\la}{4} \sum_{p+q = s-1} S(E_{n,n-1}(u^p), E_{n-1,1}(u^q)) \notag \\
& - \frac{\la}{4} \sum_{1\le i\neq j \le n} \sum_{p+q = s-1} S\big([E_{n,n-1}, E_{ij}](u^p), [E_{ji}, E_{n-1,1}](u^q)\big). \label{PsEn1}
\end{align}

Consider the finite-dimensional $\mfU(\mfsl_n)$-submodule $V$ of $\mfD_{\la,\beta}(\mfsl_n)$ generated by $P_s(E_{n1})$ under the adjoint action of $\mfsl_n$ on $\mfD_{\la,\beta}(\mfsl_n)$. Let us see why $P_s(E_{n1})$ is a lowest weight vector.  $P_s(E_{n1})$ is a weight vector since it is a sum of weight vectors of the same weight $\eps_n - \eps_1$. Suppose that $e>f$. Then $[E_{ef},P_s(E_{n1})] = 0$: this can be checked quickly using \eqref{mpq} if $e$ and $f$ are both $\neq n-1$. What is less clear is that $[E_{ef},P_s(E_{n1})] = 0$ when $e=n-1$ or $f=n-1$. For instance, if $e=n-1$ and $f<n-1$, then \begin{align*}
[E_{n-1,f},P_s(E_{n1})] = {} &  - [K(E_{n,f}),E_{n-1,1}(u^s)] + \frac{\la}{4} \sum_{p+q = s-1} S(E_{nf}(u^p), E_{n-1,1}(u^q)) \\
& + \frac{\la}{4} \sum_{1\le i\neq j \le n} \sum_{p+q = s-1} S\big([E_{nf}, E_{ij}](u^p), [E_{ji}, E_{n-1,1}](u^q)\big) - \frac{\la}{4} (1+\delta_{f1}) \sum_{p+q = s-1} S(E_{nf}(u^p), E_{n-1,1}(u^q)) \\
= {} &  0.
\end{align*}
Here, we used the expression \eqref{rel:higher degree rel} for $[K(E_{nf}),E_{n-1,1}(u^s)]$ which was established previously since $[E_{nf}, E_{n-1,1}]=0$. Similarly, $[E_{ef},P_s(E_{n1})] = 0$ when $e=n$ and $f=n-1$, so $P_s(E_{n1})$ is indeed a lowest weight vector. 

\noindent \textbf{Step 1.2:} The next step is to check that $\mathrm{ad}(E_{i,i+1})^{1-(\eps_{n1},\al_i)}(P_s(E_{n1}))=0$ for $i=1,2,\ldots,n-1$. When $i\neq 1, n-1$, $(\eps_{n1},\al_i)=0$, so we have to verify that $[E_{i,i+1},P_s(E_{n1})]=0$ in those cases. This is quite clear if $i\neq n-2$ (and $i\neq 1, n-1$) using \eqref{mpq}. If $i=n-2$, we compute the following using again \eqref{mpq}: \begin{align*}
[E_{n-2,n-1},P_s(E_{n1})] = {} & [K(E_{n,n-1}),E_{n-2,1}(u^s)] - \frac{\la}{4} \sum_{p+q = s-1} S(E_{n,n-1}(u^p), E_{n-2,1}(u^q)) \\
& - \frac{\la}{4} \sum_{1\le i\neq j \le n} \sum_{p+q = s-1} S\big([E_{n,n-1}, E_{ij}](u^p), [E_{ji}, E_{n-2,1}](u^q)\big) + \frac{\la}{4} \sum_{p+q = s-1} S( E_{n,n-1}(u^p),  E_{n-2,1}(u^q))
\end{align*}
and the right-hand side vanishes as a consequence of relation \eqref{rel:higher degree rel} for $[K(E_{n,n-1}),E_{n-2,1}(u^s)]$, which was established earlier since $[E_{n,n-1},E_{n-2,1}]=0$. 

\noindent \textbf{Step 1.3:} We also want to check that $\mathrm{ad}(E_{12})^2(P_s(E_{n1}))=0$ and $\mathrm{ad}(E_{n-1,n})^2(P_s(E_{n1}))=0$. We give the details in the second case. We start with: 

\begin{align*}
[E_{n-1,n},P_s(E_{n1})] = {} & [K(H_{n-1,n}),E_{n-1,1}(u^s)] - s\left(\beta - \frac{\la}{2}\right) E_{n-1,1}(u^{s-1}) - \frac{\la}{4} \sum_{p+q = s-1} S(H_{n-1,n}(u^p), E_{n-1,1}(u^q)) \\
& - \frac{\la}{4} \sum_{1\le i\neq j \le n} \sum_{p+q = s-1} S\big( [H_{n-1,n}, E_{ij}](u^p), [E_{ji}, E_{n-1,1}](u^q)\big) + \frac{\la}{4} \sum_{p+q = s-1} S(H_{n-1,n}(u^p),  E_{n-1,1}(u^q)),
\end{align*}
which follows from \eqref{mpq}. Applying $\mathrm{ad}(E_{n-1,n})$ again yields:  \begin{align*}
\big[ E_{n-1,n},[E_{n-1,n},P_s(E_{n1})]\big] = {} & -2[K(E_{n-1,n}),E_{n-1,1}(u^s)] + \frac{\la}{2} \sum_{p+q = s-1} S(E_{n-1,n}(u^p), E_{n-1,1}(u^q)) \\
& + \frac{\la}{2} \sum_{1\le i\neq j \le n} \sum_{p+q = s-1} S\big([E_{n-1,n}, E_{ij}](u^p), [E_{ji}, E_{n-1,1}](u^q)\big) \\ 
& - \frac{\la}{2} \sum_{p+q = s-1} S(E_{n-1,n}(u^p), E_{n-1,1}(u^q)) - \frac{\la}{2} \sum_{p+q = s-1} S(E_{n-1,n}(u^p),  E_{n-1,1}(u^q)).
\end{align*}
The right-hand side vanishes due to \eqref{rel:higher degree rel}, which has been established already because $[E_{n-1,n},E_{n-1,1}]=0$.

\noindent \textbf{Step 1.4:} Once all this has been proved, we can conclude that the finite dimensional $\mfU(\mfsl_n)$-submodule $V$ of $\mfD_{\la,\beta}(\mfsl_n)$ generated by $P_s(E_{n1})$ is isomorphic to the adjoint representation of $\mfsl_n$ on itself. This implies that if $X \in \mfsl_n$, then $P_s(X)$ can be defined as the element in $V$ corresponding to $X$ under the isomorphism $V\iso \mfsl_n$ of $\mfsl_n$-modules which sends $P_s(E_{n1})$ to the lowest root vector $E_{n1}$.

\noindent \textbf{Step 2.1:} It remains to prove that $P_s(X)$ satisfies the formula \eqref{rel:higher degree rel} for all $X\in\mfsl_n$. By definition, this is true when $X=E_{n1}$, $E_{ab} = E_{n,n-1}$ and $E_{cd} =E_{n-1,1}$. Relation \eqref{rel:higher degree rel} has also already been established when $[E_{ab},E_{cd}]=0$. It remains to check that it holds when $b=c$ and $a,b,d$ are all distinct, and also when $a=d$ and $a,b,c$ are all distinct.  For the first case, we employ an idea that was used in section \ref{rankge3}, namely start with formula \eqref{PsEn1} for $P_s(E_{n1})$ and apply a sequence of operators $s_i$ which sends $E_{n1}$ to $\pm E_{ad}$, $E_{n,n-1}$ to $\pm E_{ab}$ and $E_{n-1,1}$ to $\pm E_{bd}$. 

To deal with the case $a=d$ and $a,b,c$ all distinct, we see that we only need to show that \eqref{rel:higher degree rel} holds in one of these cases since then we can apply the same argument involving the operators $s_i$ as in the previous paragraph.

\noindent \textbf{Step 2.2:}  By induction on $s$, using \eqref{rel:higher degree rel}, we can assume that the following identity holds in $\mfD_{\la,\beta}(\mfsl_n)$: \begin{align*}
[K(E_{n-1,1}), & E_{n,n-1}(u^{s-1})] + [K(E_{n,n-1}), E_{n-1,1}(u^{s-1})] =  2(s-1)\left(\beta - \frac{\la}{2}\right) E_{n1}(u^{s-2}) \\
& + \frac{\la}{2} \sum_{p+q = s-2} S(E_{n,n-1}(u^p), E_{n-1,1}(u^q))  +  \frac{\la}{2} \sum_{1 \le i\neq j \le n} \sum_{p+q = s-2}S\big( [E_{n,n-1},E_{ij}](u^p),[E_{ji},E_{n-1,1}](u^q)\big) 
\end{align*}

Let us apply $[Q(H_{12}), \cdot]$ to both sides of the previous equality to obtain, using \eqref{mu}: 
\begin{align*}
 \big[ [Q(H_{12}) & ,K(E_{n-1,1})] ,  E_{n,n-1}(u^{s-1})\big] + \big[ [Q(H_{12}),K(E_{n,n-1})],  E_{n-1,1}(u^{s-1})\big] - [K(E_{n,n-1}),  E_{n-1,1}(u^{s})]  \\
 = {} &  -2(s-1)\left(\beta - \frac{\la}{2}\right) E_{n1}(u^{s-1}) - \frac{\la}{2} \sum_{p+q = s-2}  S(E_{n,n-1}(u^p), E_{n-1,1}(u^{q+1}))  \\
 & -\frac{\la}{2}\sum_{1 \le i\neq j \le n} \sum_{p+q = s-1}
S\big( [E_{n,n-1},E_{ij}](u^p), [E_{ji},E_{n-1,1}](u^{q})\big) \\
&-\frac{\la}{2} \sum_{1 \le i\neq j \le n} \left( 
S\Big(\big[ H_{12}, [E_{n,n-1},E_{ij}]\big] ,[E_{ji},E_{n-1,1}](u^{s-1})\Big) + 
S\Big( [E_{n,n-1},E_{ij}](u^{s-1}), \big[ H_{12}, [E_{ji},E_{n-1,1}]\big]\Big) \right).
\end{align*}

We want to rewrite the left-hand side using the following consequence of Lemma \ref{lem:KEab QHcd}: \begin{equation*}
[K(E_{n-1,1}),Q(H_{12})] + [K(E_{n-1,1}),Q(H_{n-1,2})] = \frac{\la}{4} \sum_{1\le i\neq j \le n} \big( S([E_{n-1,1},E_{ij}],[E_{ji},H_{12} + H_{n-1,2}]) + 2\left(\beta - \frac{\la}{2}\right) E_{n-1,1}
\end{equation*}
Substituting this into the left-hand side of the previous expression and using $[Q(H_{12}),K(E_{n,n-1})]=\frac{\la}{2}\big(S(E_{n1}, E_{1,n-1}) - S(E_{n2}, E_{2,n-1})\big)$ yields: 

\begin{align}
- \big[ [Q(H_{n-1,2}) & ,K(E_{n-1,1})]  ,E_{n,n-1}(u^{s-1})\big]  -[K(E_{n,n-1}),E_{n-1,1}(u^{s})] \notag \\ 
 = {} & \frac{\la}{4} \sum_{1\le i\neq j \le n} \left[ S([E_{n-1,1},E_{ij}],[E_{ji},H_{12}]) + S([E_{n-1,1},E_{ij}],[E_{ji},H_{n-1,2}]) ,  E_{n,n-1}(u^{s-1})\right] \notag\\
&-\frac{\lambda}{2} \Big(S(E_{n1}, H_{1, n-1}(u^{s-1}))-S(E_{n2}, E_{21}(u^{s-1}))\Big) 
 -2s\left(\beta - \frac{\la}{2}\right) E_{n1}(u^{s-1})  \notag \\
 &- \frac{\la}{2} \sum_{p+q = s-2}  S(E_{n,n-1}(u^p), E_{n-1,1}(u^{q+1})) 
  -\frac{\la}{2}\sum_{1 \le i\neq j \le n} \sum_{p+q = s-1}
S\big( [E_{n,n-1},E_{ij}](u^p), [E_{ji},E_{n-1,1}](u^{q})\big) \notag\\
&-\frac{\la}{2} \sum_{1 \le i\neq j \le n} \left(
S\Big(\big[ H_{12}, [E_{n,n-1},E_{ij}]\big] ,[E_{ji},E_{n-1,1}](u^{s-1})\Big) +
S\Big([E_{n,n-1},E_{ij}](u^{s-1}), \big[ H_{12}, [E_{ji},E_{n-1,1}]\big]\Big) \right).
\label{KEn1}
\end{align} 
\noindent \textbf{Step 2.3:}  The left-hand side of the previous expression is equal to \begin{equation} -[K(E_{n,n-1}),E_{n-1,1}(u^{s})] - \big[Q(H_{n-1,2}),[K(E_{n-1,1}),E_{n,n-1}(u^{s-1})] \big] +  \big[ K(E_{n-1,1}),[Q(H_{n-1,2}),E_{n,n-1}(u^{s-1})] \big]. \label{KHE}  \end{equation} The last term is equal to $-[K(E_{n-1,1}),E_{n,n-1}(u^{s})]$, which is exactly (up to a sign) what we want to determine. The first term was determined earlier when we treated the case $b=c$ and $a,b,d$ all distinct, so \eqref{rel:higher degree rel} holds for it. By induction on $s$, we can assume that relation \eqref{rel:higher degree rel} holds for $[K(E_{n-1,1}),E_{n,n-1}(u^{s-1})]$, so the second term in \eqref{KHE} can be determined using \eqref{mu}: 
\begin{align}
\big[Q(H_{n-1,2}&),  [K(E_{n-1,1}),E_{n,n-1}(u^{s-1})] \big] \notag \\
= {} & -[Q(H_{n-1,2}),P_{s-1}(E_{n1})]
+\frac{\lambda}{4}S(E_{n-1,1}(u^{s-1}), E_{n,n-1})
-\frac{\lambda}{4}S(E_{n-1,1}, E_{n,n-1}(u^{s-1}))\notag\\
& {} + \frac{\lambda}{4}\sum_{i\neq j}\left( \sum_{p+q=s-1}
S\Big(\big[[H_{n-1,2}, E_{n-1,1}], E_{ij}\big](u^{p}), [E_{ji}, E_{n,n-1}](u^{q})\Big)
- S\Big(\big[ H_{n-1,2}, [E_{n-1,1}, E_{ij}]\big] , [E_{ji}, E_{n,n-1}](u^{s-1})\Big) \right) \notag\\
& {} + \frac{\lambda}{4}\sum_{i\neq j} \left(\sum_{p+q=s-1}
S\Big( [E_{n-1,1}, E_{ij}](u^{p}), \big[E_{ji}, [H_{n-1,2},  E_{n,n-1}]\big](u^{q})\Big)
- S\Big( [E_{n-1,1}, E_{ij}](u^{s-1}), \big[ H_{n-1,2}, [E_{ji}, E_{n,n-1}]\big]\Big) \right) \notag\\
 = {} & -[Q(H_{n-1,2}),P_{s-1}(E_{n1})]
+\frac{\lambda}{4}S(E_{n-1,1}(u^{s-1}), E_{n,n-1})
-\frac{\lambda}{4}S(E_{n-1,1}, E_{n,n-1}(u^{s-1}))\notag\\
& - \frac{\lambda}{4}\sum_{i\neq j} \left(
S\Big(\big[ H_{n-1,2}, [E_{n-1,1}, E_{ij}]\big] , [E_{ji}, E_{n,n-1}](u^{s-1})\Big) + 
S\Big( [E_{n-1,1}, E_{ij}](u^{s-1}), \big[ H_{n-1,2}, [E_{ji}, E_{n,n-1}]\big]\Big) \right). \label{QKE}
\end{align}

$[Q(H_{n-1,2}),P_{s-1}(E_{n1})]$ can be determined because $[H_{n-1,2}, E_{n1}] = 0$. Indeed, we can use Lemma \ref{lem:higher KEab QHcd} below (which we can assume holds by induction) with $a=n=d, \, b=1=c$ to obtain an expression for $P_{s-1}(E_{n1})$ and then use this expression, along with \eqref{mu}, to compute $[Q(H_{n-1,2}),P_{s-1}(E_{n1})]$:
\begin{align}
[Q(H_{n-1,2})  ,P_{s-1}(E_{n1})] 
 = {} & \frac{\lambda}{4} \Big(
S(E_{n, n-1}(u^{s-1}), E_{n-1, 1})
+S(E_{n, n-1}, E_{n-1, 1}(u^{s-1}))
-S(E_{n2}(u^{s-1}) , E_{21})
-S(E_{n2}, E_{21}(u^{s-1}) )\Big)
\notag \\
&+\frac{\lambda}{8}\sum_{1\leq i \neq j\leq n} \left( 
S\Big(\big[ H_{n-1,2}, [E_{n1}, E_{ij}]\big] , [E_{ji}, H_{1n}](u^{s-1})\Big) + 
S\Big([E_{n1}, E_{ij}](u^{s-1}), \big[ H_{n-1,2}, [E_{ji}, H_{1n}]\big] \Big) \right) \label{QPs}
\end{align}

\noindent \textbf{Step 2.4:} Substituting \eqref{QPs} into \eqref{QKE} and then substituting the resulting expression into \eqref{KHE}, we obtain an expression for the left-hand side of \eqref{KEn1} from which we can isolate $[K(E_{n-1,1}),E_{n,n-1}(u^{s})]$. Relation \eqref{rel:higher degree rel} has already been established for $[K(E_{n,n-1}), E_{n-1,1}(u^{s})]$;  using it and performing a few elementary simplifications, we obtain:
 \begin{align*}
[K(E_{n-1,1}) & ,E_{n,n-1}(u^{s})] = -P_{s}(E_{n1})+s\left(\beta - \frac{\la}{2}\right) E_{n1}(u^{s-1})\\
&+ \frac{\la}{4} \sum_{p+q = s-1}  S(E_{n,n-1}(u^p), E_{n-1,1}(u^{q}))
+\frac{\lambda}{4}\sum_{i\neq j} \sum_{p+q=s-1} S\big([E_{n, n-1}, E_{ij}](u^p), [E_{ji}, E_{n-1, 1}](u^q)\big)\\
&
+\frac{\lambda}{4} 
\Big(
-S(E_{n2}(u^{s-1}) , E_{21})
-S(E_{n2}, E_{21}(u^{s-1}) )\Big)
\notag\\
&+\frac{\lambda}{8}\sum_{1\leq i \neq j\leq n} \left(
S\Big(\big[ H_{n-1,2}, [E_{n1}, E_{ij}]\big] , [E_{ji}, H_{1n}](u^{s-1})\Big) + 
S\Big([E_{n1}, E_{ij}](u^{s-1}), \big[ H_{n-1,2}, [E_{ji}, H_{1n}]\big]\Big) \right) \notag\\
&+ \frac{\lambda}{4}\sum_{i\neq j} \left(
S\Big(\big[ H_{n-1,2}, [E_{n-1,1}, E_{ij}]\big] , [E_{ji}, E_{n,n-1}](u^{s-1})\Big) + 
S\Big( [E_{n-1,1}, E_{ij}](u^{s-1}), \big[ H_{n-1,2}, [E_{ji}, E_{n,n-1}]\big]\Big) \right) \\
&-\frac{\la}{4} \sum_{1\le i\neq j \le n}  \left[
S([E_{n-1,1},E_{ij}],[E_{ji},H_{12}+H_{n-1, 2}]) ,  E_{n,n-1}(u^{s-1}) \right] \notag\\
&+
\frac{\lambda}{2} \Big(S(E_{n1}, H_{1, n-1}(u^{s-1}))-S(E_{n2}, E_{21}(u^{s-1}))\Big)
\notag \\
&+\frac{\la}{2} \sum_{1 \le i\neq j \le n} \left(
S\Big( \big[ H_{12}, [E_{n,n-1},E_{ij}]\big] ,[E_{ji},E_{n-1,1}](u^{s-1})\Big) +
S\Big( [E_{n,n-1},E_{ij}](u^{s-1}), \big[ H_{12}, [E_{ji},E_{n-1,1}]\big] \Big) \right).\\
\end{align*}
It can be checked that the previous long relation simplifies to: 

\begin{align*}
[ K(E_{n-1,1}),E_{n,n-1}(u^{s}) ] = {} &  -P_s(E_{n1})  + s\left(\beta-\frac{\lambda}{2}\right) E_{n1}(u^{s-1})  + \frac{\lambda}{4}\sum_{p+q=s-1}S(E_{n-1,1}(u^p), E_{n,n-1}(u^{q}))\notag\\
& + \frac{\lambda}{4}\sum_{1\leq i\neq j \leq n} \sum_{p+q=s-1}S\big( [E_{n-1,1}, E_{ij}](u^p), [E_{ji}, E_{n,n-1}](u^q)\big).
\end{align*}

This shows that \eqref{rel:higher degree rel} holds also for at least one case with $a=d$ and $a,b,c$ all distinct. As mentioned earlier, this is enough to complete the proof.
\end{proof}

The identities that we proved in Section \ref{Sec: sl_n two par} still hold more generally in type $A$ and with only a slightly modified proof (using for instance \eqref{mpq} instead of \eqref{m}), which we omit in this section. The following lemma is parallel to Lemma \ref{lem:KEab QHcd}.
\begin{lemma}\label{lem:higher KEab QHcd}
For any $a\neq b$ and $c\neq d$, and for any $s\ge 0$, the following relations hold in $\mfD_{\lambda, \beta}(\mathfrak{sl}_n)$:
\[[K(E_{ab}), H_{cd}( u^s)]=P_s([E_{ab}, H_{cd}])+
\frac{\lambda}{4}\sum_{1\leq i \neq j\leq n}\sum_{p+q=s-1}S\big( [E_{ab}, E_{ij}](u^p), [E_{ji}, H_{cd}](u^q)\big) +
s\left(\beta-\frac{\lambda}{2}\right) (\epsilon_a+\epsilon_b, \eps_{cd})E_{ab}(u^{s-1})\] 
and
\[[K(H_{ab}), E_{cd}(u^s)]=P_s([H_{ab}, E_{cd}])+\frac{\lambda}{4}\sum_{1\leq i \neq j\leq n}\sum_{p+q=s-1}S\big( [H_{ab}, E_{ij}](u^p), [E_{ji}, E_{cd}](u^q)\big) +s\left(\beta-\frac{\lambda}{2}\right) (\eps_{ab}, \epsilon_c+\epsilon_d)E_{cd}(u^{s-1}).\]
\end{lemma}

\subsection{Proof of Theorem \ref{thm:center=polynomial rings}}

The elements $Z_{ab}(s)$ and $\wt{Z}_{ab}(s)$ (see \eqref{Zabcd} and \eqref{wtZabcd}) can be expressed in a different way (compare with $Z_{ab}$ and $W_{ab}$ in \eqref{defz} and \eqref{defw}). Set \[ W_{ab}(s) = [K(E_{ab}), E_{ba}(u^s)]-P_s(H_{ab})-\frac{\lambda}{4}\sum_{1\leq i \neq j\leq n} \sum_{p+q = s-1}S\big( [E_{ab}, E_{ij}](u^p), [E_{ji}, E_{ba}](u^q)\big) -\frac{\lambda}{2}\sum_{p+q = s-1}S(E_{ab}(u^p), E_{ba}(u^q)).  \]
The next proposition is parallel to Proposition \ref{prop:two param} and its proof is analogous.
\begin{proposition}\label{prop: two param higher case}
\begin{romenum}
  \item \label{prop:two param item 1 s}
  For any $1\leq a \neq b \leq n$ and $1\leq c\neq b\leq n$, 
\[
  Z_{ab, cd}(s)=(\eps_{ab}, \epsilon_{cd})W_{ab}(s)+
   s\left(\beta-\frac{\lambda}{2}\right)(\epsilon_a+\epsilon_b, \eps_{cd})H_{ab}(u^{s-1})=(\eps_{ab}, \epsilon_{cd})W_{cd}(s)+
   s\left(\beta-\frac{\lambda}{2}\right)(\epsilon_c+\epsilon_d, \eps_{ab})H_{cd}(u^{s-1}).
  \]
In particular, we have $Z_{ab}(s)=2 W_{ab}(s)$, and when $a, b, c, d$ are distinct,
$Z_{ab, cd}(s)=0$.
\item \label{prop:two param item 2 s}
For $1\leq a\neq b \leq n$ and $1\leq c \neq d \leq n$,
\[
W_{ab}(s)-W_{cd}(s)=s\left(\beta - \frac{\la}{2}\right)(H_{ac}+H_{bd})(u^{s-1}).
\]
\item \label{prop:two param item 3 s}
For $1\leq a\neq b \leq n$ and $1\leq c \neq d \leq n$,
\[Z_{ab}(s)-Z_{cd}(s)=2s\left(\beta -\frac{\la}{2}\right) (H_{ac}+H_{bd})(u^{s-1}).\]
\end{romenum}
\end{proposition}

\begin{lemma}\label{ZsXKX}
For any $s\ge 1$ and any $\lambda, \beta \in\C$, the element $Z(s)$ commutes with the subalgebra $\mfU(\mfsl_n[u])$ of $\mfD_{\lambda, \beta}(\mfsl_n)$.
\end{lemma}
\begin{proof}
It suffices to show that for any $X\in \mfsl_n$ the commutator $[Z(s), X]$ is zero and $[Z(s), Q(X')]=0$ for at least one non-zero element $X'\in \mfsl_n$ because $\{X, Q(X')\mid X\in \mfsl_n\}$ is a set of generators for the image of $\mfU(\mfsl_n[u]) \lra \mfD_{\lambda, \beta}(\mfsl_n)$.

Slightly modifying the proof of Theorem \ref{central ele sln} and using Lemma \ref{lem:higher KEab QHcd},  we have, for any
$a\neq b$ and $c\neq d$,
\[
  [Z_{ab}(s), E_{cd}]=2 s\left(\beta-\frac{\lambda}{2}\right)(\eps_{ab}, \eps_{cd})(\eps_{ab}, \epsilon_c+\epsilon_d) E_{cd}(u^{s-1}).
  \]
 Therefore,
 \[
 [Z(s), E_{i, i+1}]=2 s\left(\beta-\frac{\lambda}{2}\right)\sum_{a=1}^n(\epsilon_{a,a+1}, \epsilon_{i,i+1})(\epsilon_{a,a+1}, \epsilon_i+\epsilon_{i+1}) E_{i, i+1}(u^{s-1})=0,
 \]
and similarly $[Z(s), E_{i+1, i}]=0$. Thus, $Z(s)$ commutes with $\mfsl_n$. 

By our assumption that $n\geq 4$, there exist $1\leq c, d \leq n$, such that $a, b, c, d$ are distinct.
We claim that $[Z_{ab}(s), Q(H_{cd})]$ equals 0, so $X'$ can be taken to be $H_{cd}$.  Indeed, by Proposition \ref{prop: two param higher case} \eqref{prop:two param item 3 s}, we have $Z(s)-n Z_{ab}(s)\in \h\otimes_{\C} \C [u^{s-1}]$. Therefore, 
\[
[Z(s), Q(H_{cd})]=n[Z_{ab}(s), Q(H_{cd})]=0.
\]
We now show the claim that $[Z_{ab}(s), Q(H_{cd})]=0$ by direct calculations. On the one hand, 
\begin{align}
 \big[ [  K(H_{ab}), H_{ab}(u^s)], & Q(H_{cd})\big] = \big[ [K(H_{ab}), Q(H_{cd})], H_{ab}(u^s)\big] \notag \\
 = {} & \frac{\lambda}{4}\sum_{1\leq i\neq j\leq n }(\epsilon_{ab}, \epsilon_{ij})(\epsilon_{cd}, \epsilon_{ij})[S(E_{ij}, E_{ji}) , H_{ab}(u^s)] \text{ since $a, b, c, d$ are distinct, see Proposition \ref{prop:two param} \eqref{prop:two param item 1}.} \notag \\
 = {} & -\frac{\lambda}{2}\sum_{1\leq i\neq j\leq n }(\epsilon_{ab}, \epsilon_{ij})^2(\epsilon_{cd}, \epsilon_{ij})
S(E_{ij}(u^s), E_{ji}). \label{ZabGcd1}
\end{align}
On the other hand,
\begin{align}
\sum_{1\leq i\neq j\leq n} 
 \sum_{p+q=s-1} & \Big[ S\big( [H_{ab}, E_{ij}](u^p),  [E_{ji}, H_{ab}](u^q)\big), Q(H_{cd})\Big] \;  = \sum_{1\leq i\neq j\leq n} \sum_{p+q=s-1} (\epsilon_{ab}, \epsilon_{ij})^2 [S(E_{ij}(u^p), E_{ji}(u^q)), Q(H_{cd})] \notag \\
 = {} & \sum_{1\leq i\neq j\leq n}
 \sum_{p+q=s-1} (\epsilon_{ab}, \epsilon_{ij})^2(\epsilon_{cd}, \epsilon_{ij}) \Big(-S(E_{ij}(u^{p+1}), E_{ji}(u^q)) + S(E_{ij}(u^p), E_{ji}(u^{q+1}))\Big) \notag \\
 = {} &  -2\sum_{1\leq i\neq j\leq n} (\epsilon_{ab}, \epsilon_{ij})^2(\epsilon_{cd}, \epsilon_{ij}) S(E_{ij}(u^{s}), E_{ji}). \label{ZabGcd2}
\end{align}
By the definition of the element $Z_{ab}(s)$, we conclude from \eqref{ZabGcd1} and \eqref{ZabGcd2} that $[Z_{ab}(s), Q(H_{cd})]=0$.
\end{proof}
In the remainder of this section, we compute the commutator $[K(H_{cd}), Z(s)]$ and deduce that $[K(H_{cd}), Z(s)]=0$ if and only if $n\lambda=\pm 4\left(\beta-\frac{\lambda}{2}\right)$, which implies, in light of the previous lemma, that $Z(s)$ is in the center of $\mfD_{\lambda, \beta}(\mathfrak{sl}_n)$ (see Theorem \ref{thm:center=polynomial rings}).

\begin{lemma} \label{lem:K(Hcd) Zab s} Assume that $a,b,c,d$ are distinct integers. For $s\geq 2$, in $\mfD_{\lambda, \beta}(\mfsl_n)$, we have the relation:
\begin{multline*}
[K(H_{cd}), Z_{ab}(s)]  =-\binom{s}{2} n\lambda^2 H_{cd}(u^{s-2}) +\lambda s\left(\beta-\frac{\lambda}{2}\right)
\sum_{p+q=s-2} 
\Big(-S( E_{ac}(u^{p}), E_{ca}(u^q) )
+S( E_{ad}(u^{p}), E_{da}(u^q))\\
-S( E_{bc}(u^{p}), E_{cb}(u^q))
+S( E_{bd}(u^{p}), E_{db}(u^q))\Big).
\end{multline*}
\end{lemma}
\begin{proof} \textbf{Step 1:} On the one hand,
\begin{align*}
\big[ K(H_{cd}), & [K(H_{ab}), H_{ab}(u^s)]\big]
=\big[ K(H_{ab}), [K(H_{cd}), H_{ab}(u^s)]\big]\\
 = {} & \frac{\lambda}{4}\sum_{1\leq i \neq j\leq n} \sum_{p+q=s-1} \Big[ K(H_{ab}), 
S\big( [H_{cd}, E_{ij}](u^p), [E_{ji}, H_{ab}](u^q)\big) \Big]
\,\ \text{by Proposition \ref{prop: two param higher case} \eqref{prop:two param item 1 s};}\\
 = {}  & \frac{\lambda}{4}\sum_{1\leq i \neq j\leq n} \sum_{p+q=s-1} (\epsilon_{cd}, \epsilon_{ij}) (\epsilon_{ab}, \epsilon_{ij}) \Big[ K(H_{ab}), S(E_{ij}(u^p), E_{ji}(u^q))\Big]\\
 = {} & \frac{\lambda}{4}\sum_{1\leq i \neq j\leq n}
\sum_{p+q=s-1} (\epsilon_{cd}, \epsilon_{ij}) (\epsilon_{ab}, \epsilon_{ij}) 
\Big( S\big( [K(H_{ab}), E_{ij}(u^p)], E_{ji}(u^q)\big) + S\big( E_{ij}(u^p), [K(H_{ab}), E_{ji}(u^q)]\big) \Big) \\
 = {}  & \frac{\lambda}{2}\sum_{1\leq i \neq j\leq n}
\sum_{p+q=s-1} (\epsilon_{cd}, \epsilon_{ij}) (\epsilon_{ab}, \epsilon_{ij}) 
S\Bigg(
(\epsilon_{ab}, \epsilon_{ij})P_p(E_{ij})+\frac{\lambda}{4}\sum_{k\neq l}\sum_{e+f=p-1}S\big( [H_{ab}, E_{kl}](u^e), [E_{lk}, E_{ij}](u^f)\big) \\
&  {} + p\left(\beta-\frac{\lambda}{2}\right)(\eps_{ab}, \epsilon_i+\epsilon_j) E_{ij}(u^{p-1})
, E_{ji}(u^q)\Bigg)  \;\; \text{by Lemma \ref{lem:higher KEab QHcd};} \\
 = {} & \frac{\lambda}{2}\sum_{1\leq i \neq j\leq n}
\sum_{p+q=s-1} (\epsilon_{cd}, \epsilon_{ij}) (\epsilon_{ab}, \epsilon_{ij})^2
S(P_p(E_{ij}), E_{ji}(u^q))\\
&+\frac{\lambda^2}{8}\sum_{1\leq i \neq j\leq n}\sum_{p+q=s-1} \sum_{k\neq l}\sum_{e+f=p-1}
(\epsilon_{cd}, \epsilon_{ij}) (\epsilon_{ab}, \epsilon_{ij})  (\epsilon_{ab}, \epsilon_{kl}) 
S\Big(S\big(E_{kl}(u^e), [E_{lk}, E_{ij}](u^f)\big), E_{ji}(u^q)\Big)\\
&+\frac{\lambda}{2}\sum_{1\leq i \neq j\leq n}
\sum_{p+q=s-1} p\left(\beta-\frac{\lambda}{2}\right) (\eps_{ab}, \epsilon_i+\epsilon_j)(\epsilon_{cd}, \epsilon_{ij}) (\epsilon_{ab}, \epsilon_{ij}) S\big( E_{ij}(u^{p-1}), E_{ji}(u^q)\big).
\end{align*}
On the other hand,
\begin{align*}
\Bigg[ K&(H_{cd}),  \frac{\lambda}{4}\sum_{1\leq i \neq j\leq n}
\sum_{p+q=s-1} (\epsilon_{ab}, \epsilon_{ij})^2S(E_{ij}(u^p), E_{ji}(u^q))\Bigg]\\
 = {} & \frac{\lambda}{4}\sum_{1\leq i \neq j\leq n}
\sum_{p+q=s-1} (\epsilon_{ab}, \epsilon_{ij})^2S\big( [K(H_{cd}), E_{ij}(u^p)], E_{ji}(u^q)\big) +\frac{\lambda}{4}\sum_{1\leq i \neq j\leq n}
\sum_{p+q=s-1} (\epsilon_{ab}, \epsilon_{ij})^2S\big( E_{ij}(u^p), [K(H_{cd}), E_{ji}(u^q)]\big) \\
 = {}  &  \frac{\lambda}{2}\sum_{1\leq i \neq j\leq n}
\sum_{p+q=s-1} (\epsilon_{ab}, \epsilon_{ij})^2S\Bigg(
(\epsilon_{cd}, \epsilon_{ij})P_p(E_{ij})+
\frac{\lambda}{4}\sum_{1 \le k\neq l \le n}\sum_{e+f=p-1}S\big([H_{cd}, E_{kl}](u^e), [E_{lk}, E_{ij}](u^f)\big)\\
&  {} +p\left(\beta-\frac{\lambda}{2}\right)(\epsilon_{cd}, \epsilon_{i}+\epsilon_{j})E_{ij}(u^{p-1})
, E_{ji}(u^q)\Bigg) \,\ \text{by Lemma \ref{lem:higher KEab QHcd};}\\
 = {} & \frac{\lambda}{2}\sum_{1\leq i \neq j\leq n}
\sum_{p+q=s-1} (\epsilon_{ab}, \epsilon_{ij})^2(\epsilon_{cd}, \epsilon_{ij})
S(P_p(E_{ij}), E_{ji}(u^q))\\
&+\frac{\lambda^2}{8}\sum_{1\leq i \neq j\leq n}
\sum_{p+q=s-1} \sum_{1\le k\neq l \le n}\sum_{e+f=p-1}
(\epsilon_{ab}, \epsilon_{ij})^2 (\epsilon_{cd}, \epsilon_{kl})S\Big( S\big( E_{kl}(u^e), [E_{lk}, E_{ij}](u^f)\big) ,
E_{ji}(u^q)\Big)\\
&+\frac{\lambda}{2}\sum_{1\leq i \neq j\leq n}
\sum_{p+q=s-1} p\left(\beta-\frac{\lambda}{2}\right)(\epsilon_{cd}, \epsilon_{i}+\epsilon_{j})(\epsilon_{ab}, \epsilon_{ij})^2S(E_{ij}(u^{p-1})
, E_{ji}(u^q)).
\end{align*}
Substituting the above computations into the formula \eqref{Zabcd} for $Z_{ab}(s)$ and cancelling out the obvious terms, we get:

\begin{align}
[K(H_{cd}), &  Z_{ab}(s)] \notag\\
 = {} & \frac{\lambda^2}{8}\sum_{1\leq i \neq j, k\neq l \leq n}\sum_{e+f+q=s-2}
\big(
(\epsilon_{cd}, \epsilon_{ij}) (\epsilon_{ab}, \epsilon_{ij})  (\epsilon_{ab}, \epsilon_{kl}) 
-
(\epsilon_{ab}, \epsilon_{ij})^2 (\epsilon_{cd}, \epsilon_{kl})\big)
S\Big( S\big( E_{kl}(u^e), [E_{lk}, E_{ij}](u^f)\big) , E_{ji}(u^q)\Big) 
 \label{E}\\
&+\frac{\lambda}{2}\sum_{1\leq i \neq j\leq n}
\sum_{p+q=s-1} p\left(\beta-\frac{\lambda}{2}\right) (\epsilon_{ab}, \epsilon_{ij}) \big( (\eps_{ab}, \epsilon_i+\epsilon_j)(\epsilon_{cd}, \epsilon_{ij}) - (\epsilon_{cd}, \epsilon_{i}+\epsilon_{j})(\epsilon_{ab}, \epsilon_{ij}) \big) S( E_{ij}(u^{p-1})
, E_{ji}(u^q)).  \label{F}
\end{align}
\noindent \textbf{Step 2:} Let's simplify the expression \eqref{E}. It can be written in a more general form as follows:
\begin{equation*}
\eqref{E}  = \frac{\lambda^2}{8}\sum_{e+f+q=s-2}\sum_{\alpha, \beta\in \Delta}\big((\gamma, \alpha)(\gamma', \alpha)(\gamma, \beta)-(\gamma, \alpha)^2(\gamma', \beta)\big)
 S\Big( S\big( X_{\beta}(u^e), [X_{-\beta}, X_{\alpha}](u^f)\big) ,
X_{-\alpha}(u^q)\Big),
\end{equation*} where $\gamma, \gamma'\in \h^*$, and $\Delta$ is the root system of a Lie algebra $\g$ with root vectors $\{X_{\alpha}\mid \alpha\in \Delta\}$.
Let us simplify \eqref{E} using the following two observations.

Observation 1: Interchanging $\al \leftrightarrow -\beta$ yields:
\begin{align*}
\sum_{\alpha, \beta\in \Delta}\big((\gamma, \alpha)(\gamma', \alpha)(\gamma, \beta) & -(\gamma, \alpha)^2(\gamma', \beta)\big)
 S\Big( S\big( X_{\beta}(u^e), [X_{-\beta}, X_{\alpha}](u^f)\big) , X_{-\alpha}(u^q) \Big)\\
& =\sum_{\alpha, \beta\in \Delta}
\big((\gamma, \beta)(\gamma', \beta)(\gamma, \alpha)-(\gamma, \beta)^2(\gamma', \alpha)\big)
 S\Big( S \big( X_{-\alpha}(u^e), [X_{-\beta}, X_{\alpha}](u^f)\big) , X_{\beta}(u^q)\Big).
\end{align*}

Observation 2: Note that if $\alpha-\beta$ is not a root of the Lie algebra $\g$, then $[X_{-\beta}, X_{\alpha}]$ is automatically zero. Therefore, 
\begin{align*}
\sum_{\alpha, \beta\in \Delta} & \big((\gamma, \alpha)(\gamma', \alpha)(\gamma, \beta)  -(\gamma, \alpha)^2(\gamma', \beta)\big)
 S\Big(S\big( X_{\beta}(u^e), [X_{-\beta}, X_{\alpha}](u^f)\big) ,
X_{-\alpha}(u^q)\Big)\\
 = {}  &  \sum_{ \{\alpha, \beta\in \Delta \, | \, \alpha-\beta\in \Delta \}}\big((\gamma, \alpha)(\gamma', \alpha)(\gamma, \beta)-(\gamma, \alpha)^2(\gamma', \beta)\big)
 S\Big(S\big( X_{\beta}(u^e), [X_{-\beta}, X_{\alpha}](u^f)\big) ,
X_{-\alpha}(u^q)\Big)\\
 = {} & \sum_{ \{ \al,\wt{\alpha}, \beta\in \Delta \, | \, \beta-\alpha = \wt{\alpha}\}}
\big((\gamma, \beta-\wt{\alpha})
(\gamma', \beta-\wt{\alpha})(\gamma, \beta)-(\gamma, \beta-\wt{\alpha})^2(\gamma', \beta)\big)
 S\Big( S\big( X_{\beta}(u^e), [X_{-\beta}, X_{\beta-\wt{\alpha}}](u^f)\big) ,
X_{\wt{\alpha}-\beta}(u^q)\Big)\\
 = {} & -\sum_{\wt{\alpha}, \beta\in \Delta}
\big((\gamma, \beta)(\gamma', \beta)(\gamma, \wt{\alpha})-(\gamma, \beta)^2(\gamma', \wt{\alpha})\big)
 S\Big(S( X_{\beta}(u^e), X_{-\wt{\alpha}}(u^f),
 [X_{-\beta}, X_{\wt{\alpha}}](u^q)\Big)\\
& -\sum_{\wt{\alpha}, \beta\in \Delta}
\big((\gamma, \wt{\alpha})(\gamma', \wt{\alpha})(\gamma, \beta)-(\gamma, \wt{\alpha})^2(\gamma', \beta)\big)
 S\Big(S(X_{\beta}(u^e), X_{-\wt{\alpha}}(u^f)),
 [X_{-\beta}, X_{\wt{\alpha}}](u^q)\Big).
\end{align*}
For the above equality, we used the following computations.
Assume $ [X_{-\beta}, X_{\wt{\alpha}}]=k X_{\wt{\alpha}-\beta}$ and $\wt{\alpha} = \beta - \al$. Then
\[
 (X_{\wt{\alpha}},[X_{-\beta}, X_{\beta-\wt{\alpha}}])
 =([X_{\wt{\alpha}}, X_{-\beta}], X_{\beta-\wt{\alpha}})
 =-k( X_{\wt{\alpha}-\beta}, X_{\beta-\wt{\alpha}})
 =-k,
\]
which gives us $[X_{-\beta}, X_{\beta-\wt{\alpha}}]=-kX_{-\wt{\alpha}}$.

After splitting \eqref{E} into three identical expressions, then switching $\al\leftrightarrow -\beta$ in the second one and applying the two observations above to the third expression, we obtain: 

\begin{align*}
\frac{8}{\lambda^2}\cdot \eqref{E} = {} & \frac{1}{3}\sum_{e+f+q=s-2}\sum_{\alpha, \beta\in \Delta}\Big((\gamma, \alpha)(\gamma', \alpha)(\gamma, \beta)-(\gamma, \alpha)^2(\gamma', \beta)\Big)
 S\Big( S\big( X_{\beta}(u^e), [X_{-\beta}, X_{\alpha}](u^f)\big),
X_{-\alpha}(u^q)\Big)\\
&+\frac{1}{3}\sum_{e+f+q=s-2}\sum_{\alpha, \beta\in \Delta}
\Big((\gamma, \beta)(\gamma', \beta)(\gamma, \alpha)-(\gamma, \beta)^2(\gamma', \alpha)\Big)
 S\Big( S\big( X_{-\alpha}(u^e), [X_{-\beta}, X_{\alpha}](u^f)\big) ,
X_{\beta}(u^q)\Big) \\
&-\frac{1}{3}\sum_{e+f+q=s-2}\sum_{\alpha, \beta\in \Delta}
\Big((\gamma, \beta)(\gamma', \beta)(\gamma, \alpha)-(\gamma, \beta)^2(\gamma', \alpha)\Big)
 S\Big( S\big( X_{\beta}(u^e), X_{-\alpha}(u^f)\big) ,
 [X_{-\beta}, X_{\alpha}](u^q)\Big) \\
&-\frac{1}{3}\sum_{e+f+q=s-2}\sum_{\alpha, \beta\in \Delta}
\Big((\gamma, \alpha)(\gamma', \alpha)(\gamma, \beta)-(\gamma, \alpha)^2(\gamma', \beta)\Big)
 S\Big( S\big( X_{\beta}(u^e), X_{-\alpha}(u^f)\big) ,
 [X_{-\beta}, X_{\alpha}](u^q)\Big) \\
 = {}  & \frac{1}{3}\binom{s}{2}\sum_{\alpha, \beta\in \Delta}\Big((\gamma, \alpha)(\gamma', \alpha)(\gamma, \beta)-(\gamma, \alpha)^2(\gamma', \beta)\Big)
\Big[ X_{\beta}, \big[ [X_{-\beta}, X_{\alpha}], X_{-\alpha}\big] \Big] (u^{s-2})\\
&  {}  +\frac{1}{3}\binom{s}{2}\sum_{\alpha, \beta\in \Delta}
\Big((\gamma, \beta)(\gamma', \beta)(\gamma, \alpha)-(\gamma, \beta)^2(\gamma', \alpha)\Big)
\Big[ X_{-\alpha}, \big[ [X_{-\beta}, X_{\alpha}], X_{\beta}\big] \Big] (u^{s-2})\\
 = {} & -\frac{2}{3}\binom{s}{2}\sum_{\alpha, \beta\in \Delta}
\Big((\gamma, \alpha)(\gamma', \alpha)(\gamma, \beta)-(\gamma, \alpha)^2(\gamma', \beta)\Big)
\big([X_{\beta},  X_{-\alpha}] , [X_{-\beta}, X_{\alpha}]\big)H_{\beta}(u^{s-2})
\end{align*}
where the second equality uses the fact
\[
S\big( S(A, B), C\big)-S\big( S(A, C), B\big) = \big[ A, [B, C] \big] ,
\] and $\binom{s}{2}$ appears as the cardinality of the set $\{(e, f, q)\in \N^{3} \mid  e+f+q=s-2\}$. 

Now specialize $\gamma$ to $\epsilon_{ab}$, $\gamma'$ to $\epsilon_{cd}$, $\alpha$ to $\epsilon_{ij}$ and $\beta$ to $\epsilon_{kl}$; by direct computations, we obtain:
\[
\sum_{1\leq i\neq j, k\neq l\leq n}
\Big((\epsilon_{ab}, \epsilon_{ij})(\epsilon_{cd}, \epsilon_{ij})(\epsilon_{ab}, \epsilon_{kl})-(\epsilon_{ab}, \epsilon_{ij})^2(\epsilon_{cd}, \epsilon_{kl})\Big)
\big([E_{kl},  E_{ji}] , [E_{lk}, E_{ij}]\big) H_{kl}=   12n H_{cd}.
\]
Thus,
\begin{equation} 
\eqref{E}=-\frac{\lambda^2}{8}\frac{2}{3}\binom{s}{2} 12 n H_{cd}(u^{s-2})
=-\binom{s}{2} n\lambda^2 H_{cd}(u^{s-2}). \label{G}
\end{equation}
\noindent \textbf{Step 3:} We move on to computing the term \eqref{F}.
Under the assumption that $a,b,c,d$ are all distinct and $i\neq j$, we can compute the constants which appear in \eqref{F} by using the basic rule $(\epsilon_a,\epsilon_i) = \delta_{ai}$: \begin{align*} (\eps_{ab}, \epsilon_i+\epsilon_j) (\epsilon_{cd}, \epsilon_{ij}) (\epsilon_{ab}, \epsilon_{ij}) & {} = -\delta_{ai}\delta_{cj}+\delta_{ai}\delta_{dj} -\delta_{aj}\delta_{ci}+\delta_{aj}\delta_{di} -\delta_{bi}\delta_{cj}+\delta_{bi}\delta_{dj} -\delta_{bj}\delta_{ci}+\delta_{bj}\delta_{di}  \\
& = -(\epsilon_{cd}, \epsilon_{i}+\epsilon_{j})(\epsilon_{ab}, \epsilon_{ij})^2. \end{align*}
Substituting this into the expression $\eqref{F}$, we obtain:
\begin{align*}
\eqref{F} = {} &
\lambda \sum_{1\leq i \neq j\leq n}
\sum_{p+q=s-1} p\left(\beta-\frac{\lambda}{2}\right)
\Big(
-\delta_{ai}\delta_{cj}+\delta_{ai}\delta_{dj}
-\delta_{aj}\delta_{ci}+\delta_{aj}\delta_{di} \notag \\
&  - \delta_{bi}\delta_{cj}+\delta_{bi}\delta_{dj}
-\delta_{bj}\delta_{ci}+\delta_{bj}\delta_{di}
\Big)
S\big( E_{ij}(u^{p-1})
, E_{ji}(u^q)\big) \notag \\  
 = {}  & \lambda 
\sum_{p+q=s-1} p\left(\beta-\frac{\lambda}{2}\right)
\Big(-S\big( E_{ac}(u^{p-1}), E_{ca}(u^q)\big)
+S\big( E_{ad}(u^{p-1}), E_{da}(u^q)\big) -S\big( E_{ca}(u^{p-1}), E_{ac}(u^q)\big)  \notag \\
&   {}  + S\big( E_{da}(u^{p-1}), E_{ad}(u^q)\big) -S\big( E_{bc}(u^{p-1}), E_{cb}(u^q)\big)
+S\big( E_{bd}(u^{p-1}), E_{db}(u^q)\big)  \notag \\
& {} -S\big( E_{cb}(u^{p-1}), E_{bc}(u^q)\big) +S\big( E_{db}(u^{p-1}), E_{bd}(u^q)\big) \Big) \notag 
\end{align*}
\begin{align}
 = {}  &  \lambda 
\sum_{p+q=s-2} (p+1)\left(\beta-\frac{\lambda}{2}\right)
\Big(-S\big( E_{ac}(u^p), E_{ca}(u^q)\big)
+S\big( E_{ad}(u^p), E_{da}(u^q)\big)  \notag \\
&  {}  -S\big( E_{bc}(u^p), E_{cb}(u^q)\big)
+S\big( E_{bd}(u^p), E_{db}(u^q)\big)\Big) \notag \\
&  {} + \lambda\sum_{p+q=s-2} (q+1)\left(\beta-\frac{\lambda}{2}\right)
\Big(-S\big( E_{ac}(u^q), E_{ca}(u^p)\big)
+S\big( E_{ad}(u^q), E_{da}(u^p)\big)   \notag \\
&  {} -S\big( E_{bc}(u^q), E_{cb}(u^p)\big) +S\big( E_{bd}(u^q), E_{db}(u^p)\big) \Big) \notag  \\
 = {} & \lambda s\left(\beta-\frac{\lambda}{2}\right)
\sum_{p+q=s-2} 
\Big(-S\big( E_{ac}(u^p), E_{ca}(u^q)\big)
+S\big( E_{ad}(u^p), E_{da}(u^q)\big)\notag \\
 &  -S\big( E_{bc}(u^p), E_{cb}(u^q)\big)
+S\big( E_{bd}(u^p), E_{db}(u^q)\big)\Big). \label{H}
\end{align}
\noindent \textbf{Step 4:} We can combine the results of steps 2 and 3 to complete the proof of the Lemma \ref{lem:K(Hcd) Zab s}:
\begin{align*}
[K(H_{cd}), Z_{ab}(s)] & {} = \eqref{E}+\eqref{F} \\
& {} = -\binom{s}{2} n\lambda^2 H_{cd}(u^{s-2}) +\lambda s\left(\beta-\frac{\lambda}{2}\right)
\sum_{p+q=s-2} 
\Big(-S\big( E_{ac}(u^p), E_{ca}(u^q)\big)
+S\big( E_{ad}(u^p), E_{da}(u^q)\big)\\
& \qquad \qquad \qquad \qquad  -S\big( E_{bc}(u^p), E_{cb}(u^q)\big)
+S\big( E_{bd}(u^p), E_{db}(u^q)\big)\Big) \text{ by \eqref{G} and \eqref{H}}.
\end{align*}
\end{proof}

The main theorem of this section relies on the following proposition. 
\begin{proposition}
Assume $n\geq 4$, $s\in \N$ and $s\geq 2$. 
For any $X\in \mathfrak{sl}_n$, in the algebra $\mfD_{\lambda, \beta}(\mfsl_n)$, we have 
\[
[K(X), Z(s)]=\left( 
16\left(\beta-\frac{\lambda}{2}\right)^2- n^2\lambda^2 \right)\cdot  \binom{s}{2} \cdot X(u^{s-2}).
\]
\end{proposition}
As a consequence, we conclude that
\begin{equation*}
[K(X), Z(s)]=0 
\Longleftrightarrow  
n^2\lambda^2 =16\left(\beta-\frac{\lambda}{2}\right)^2
\Longleftrightarrow  
n\lambda=4\left(\beta-\frac{\lambda}{2}\right) \,\ \text{ or } \,\ n\lambda=-4\left(\beta-\frac{\lambda}{2}\right),
\end{equation*}
which completes the proof that $Z(s)$ is in the center of $\mfD_{\lambda, \beta}(\mathfrak{sl}_n)$ (see Theorem \ref{thm:center=polynomial rings}).
\begin{proof}
By Proposition \ref{prop: two param higher case}, we have:
\begin{equation}
Z(s)=\sum_{a=1}^n Z_{a, a+1}(s) =n Z_{12}(s)+2s\left( \beta-\frac{\lambda}{2}\right) \sum_{a=1}^n(H_{a1}+H_{a, 2})(u^{s-1}). \label{Zssum}
\end{equation}
By the same proposition, we also have:
\begin{align*}
\sum_{a=3}^n (Z_{34, a1}(s-1)+Z_{34, a2}(s-1))   {} = {}  & \sum_{a=3}^n 
(\epsilon_3-\epsilon_4, \epsilon_a-\epsilon_1)W_{34}(s-1)
+(s-1) \left(\beta-\frac{\lambda}{2}\right)(\epsilon_3+\epsilon_4, \epsilon_a-\epsilon_1) H_{34}(u^{s-2})\\
&  {} +\sum_{a=3}^n 
(\epsilon_3-\epsilon_4, \epsilon_a-\epsilon_2)W_{34}(s-1)
+(s-1) \left(\beta-\frac{\lambda}{2}\right)(\epsilon_3+\epsilon_4, \epsilon_a-\epsilon_2) H_{34}(u^{s-2}) \\
 = {} &4(s-1)\left(\beta-\frac{\lambda}{2}\right) H_{34}(u^{s-2}),
\end{align*}
which is equivalent to the identity
\begin{align}
\left[ K(H_{34}), \sum_{a=1}^n(H_{a1}+H_{a2})(u^{s-1}) \right]
 = {} & \frac{\lambda}{4}\sum_{i\neq j}\sum_{p+q=s-2} S\left( [H_{34}, E_{ij}](u^p), \left[ E_{ji}, \sum_{a=1}^n(H_{a1}+H_{a2})\right] (u^q)\right) \notag \\
&+4\left(\beta-\frac{\lambda}{2}\right) (s-1) H_{34}(u^{s-2}). \label{KH34H}
\end{align}
Substituting \eqref{Zssum} into the commutator $[K(H_{34}), Z(s)]$, we obtain:
\begin{align*}
[K(H_{34}), Z(s)]
 {} = {} &  n[K(H_{34}), Z_{12}(s)]+
2s\left(\beta-\frac{\lambda}{2}\right) 
\frac{\lambda}{4}\sum_{i\neq j}\sum_{p+q=s-2} S\left([H_{34}, E_{ij}](u^p), \left[ E_{ji}, \sum_{a=1}^n(H_{a1}+H_{a2})\right] (u^q)\right)  \\
&+8s\left(\beta-\frac{\lambda}{2}\right)^2 (s-1) H_{34}(u^{s-2})  \text{ by } \eqref{KH34H}
\\
 = {}  & n[K(H_{34}), Z_{12}(s)]+
s\left(\beta-\frac{\lambda}{2}\right) 
\frac{\lambda}{2}\sum_{i\neq j}\sum_{p+q=s-2} (\epsilon_3-\epsilon_4, \eps_{ij})
\left( 2\sum_{a=1}^n \epsilon_a-n\epsilon_1-n\epsilon_2, \eps_{ij}\right)
S(E_{ij}(u^p), E_{ji} (u^q))  \\
&  {} +8s\left(\beta-\frac{\lambda}{2}\right)^2 (s-1) H_{34}(u^{s-2})  \\
 = {} &  n[K(H_{34}), Z_{12}(s)] +8s\left(\beta-\frac{\lambda}{2}\right)^2 (s-1) H_{34}(u^{s-2})   \\
& - sn\left(\beta-\frac{\lambda}{2}\right) 
\frac{\lambda}{2}\sum_{i\neq j}\sum_{p+q=s-2} 
(\delta_{3i}-\delta_{3j}-\delta_{4i}+\delta_{4j})
(\delta_{1i}-\delta_{1j}+\delta_{2i}-\delta_{2j})
S(E_{ij}(u^p), E_{ji} (u^q))  \\
 = {} &  -\binom{s}{2} n^2\lambda^2 H_{34}(u^{s-2})
+8s\left(\beta-\frac{\lambda}{2}\right)^2 (s-1) H_{34}(u^{s-2})
\\
&+n\lambda s\left(\beta-\frac{\lambda}{2}\right)
\sum_{p+q=s-2} 
\big(-S\Big( E_{13}(u^p), E_{31}(u^q)\big)
+S\big( E_{14}(u^p), E_{41}(u^q)\big) \\ 
& -S\big( E_{23}(u^p), E_{32}(u^q)\big)
+S\big( E_{24}(u^p), E_{42}(u^q)\big)\Big) \text{ by Lemma \ref{lem:K(Hcd) Zab s}} \\
&-sn\left(\beta-\frac{\lambda}{2}\right) 
\frac{\lambda}{2}\sum_{p+q=s-2} 
\Big(-S(E_{31}(u^p), E_{13} (u^q))
-S(E_{32}(u^p), E_{23} (u^q))\\
&  -S(E_{13}(u^p), E_{31} (u^q))
-S(E_{23}(u^p), E_{32} (u^q)) +S(E_{41}(u^p), E_{14} (u^q))
+S(E_{42}(u^p), E_{24} (u^q)) \\
& 
+S(E_{14}(u^p), E_{41} (u^q))
+S(E_{24}(u^p), E_{42} (u^q))
\Big)\\
 = {} & \left( 16\left( \beta-\frac{\lambda}{2}\right)^2- n^2\lambda^2 \right) \binom{s}{2} H_{34}(u^{s-2}).
\end{align*}
We have shown that $Z(s)$ commutes with $X$ for any $X\in \mfsl_n$ (see Lemma \ref{ZsXKX}). 
Using $[K(X), Y]=K([X, Y])$, it follows that, for any $X\in \mfsl_n$,
\[
[K(X), Z(s)]=\left(  16\left( \beta-\frac{\lambda}{2}\right)^2- n^2\lambda^2 \right) \binom{s}{2} X(u^{s-2}).
\]
\end{proof}

\begin{proof}[Proof of Theorem \ref{thm:center=polynomial rings}]
In the preceding pages, we have shown that the elements $Z(s)$ are all in the center of $\mfD_{\lambda, \beta}(\mfsl_n)$.  The proof that $\wt{Z}(s)$ is central is similar. To complete the proof of Theorem \ref{thm:center=polynomial rings}, we have to see why they generate a subalgebra isomorphic to a polynomial ring in infinitely many variables. This is a consequence of the PBW Theorem for $\mfD_{\lambda, \beta}(\mfsl_n)$ established in \cite{G2} which states that the associated graded ring of $\mfD_{\lambda, \beta}(\mfsl_n)$ is isomorphic to the enveloping algebra of $\wh{\mfsl_n[u,v]}$, the universal central extension of $\mfsl_n[u,v]$. The center of $\wh{\mfsl_n[u,v]}$ is known to be isomorphic to $\Omega^1(\C[u,v])/d\C[u,v]$ (see \cite{KaLo}) and, as vector spaces, $\wh{\mfsl_n[u,v]} \cong \mfsl_n[u,v] \oplus \Omega^1(\C[u,v])/d\C[u,v]$. The associated graded ring $\mathrm{gr}(\mfD_{\lambda, \beta}(\mfsl_n))$ is obtained from the filtration $F_{\bullet}$ on $\mfD_{\lambda, \beta}(\mfsl_n)$ that assigns degree $1$ to $Q(X)$ and degree $0$ to $X$ and $K(X)$ for all $X\in\mfsl_n$. (Under the isomorphism of Theorem \ref{thm:main theorem}, the corresponding filtration on $\msD(\mfg)$ is given by assigning degree $r$ to $\msX_{i,r}^{\pm}$ and $\msH_{i,r}$ for $r=0,1$.) The central element $Z(s)$ has filtration degree $s$, so $\ol{Z(s)}$ is an element of $F_s/F_{s-1}$ which, under the isomorphism between the center of $\wh{\mfsl_n[u,v]}$ and $\Omega^1(\C[u,v])/d\C[u,v]$, corresponds to the central element given by $2nvu^{s-1}du$. The elements $vu^{s-1}du$ for $s\ge 1$ are all linearly independent in $\Omega^1(\C[u,v])/d\C[u,v]$ and are all algebraically independent in $\mfU(\wh{\mfsl_n[u,v]}) \cong \mathrm{gr}_{F_{\bullet}}(\mfD_{\lambda, \beta}(\mfsl_n))$.  Therefore, the elements $Z(s)$ for $s\ge 1$ must be algebraically independent in $\mfD_{\lambda, \beta}(\mfsl_n)$. \end{proof}

\appendix\section{}
\subsection{Computations for the proof of Lemma \ref{lem:nu}}
\label{appendix1}
Here is an alternate version of Lemma \ref{lem:nu}. 
\begin{lemma}\label{lem:nu2}
Set $S(x, y, z)=S(S(x, y), z)
+S(S(x, z), y)+S(S(y, z), x).$ 
Assuming that $i\neq j$, 
we have:
\begin{align*}
[\nu_i, \nu_j]
 {} = & -\frac{1}{48}\sum_{k=1}^n  \Big(S( E_{ki}, E_{i, j+1}, E_{j+1, k})
- S(E_{ik}, E_{k, j+1}, E_{j+1, i})
- S( E_{k, i+1}, E_{i+1, j+1}, E_{j+1, k}) \\
&  + S( E_{i+1, k}, E_{k, j+1}, E_{j+1, i+1}) - S( E_{ki}, E_{ij}, E_{jk})
+S(E_{ik}, E_{kj}, E_{ji}) \\
& +  S( E_{k, i+1}, E_{i+1, j}, E_{jk})- S( E_{i+1, k}, E_{kj}, E_{j, i+1})
\Big).
\end{align*}
Note that in the case $j=i+1$, the last two terms cancel, so
\begin{align*}
[\nu_i, \nu_{i+1}]
 = {} & -\frac{1}{48}
\sum_{k=1}^n \Big(
S(E_{ki}, E_{i, i+2}, E_{i+2, k})
- S(E_{ik}, E_{k, i+2}, E_{i+2, i})
- S(E_{k, i+1}, E_{i+1,i+2}, E_{i+2, k})
\\
& + S( E_{i+1, k}, E_{k, i+2}, E_{i+2, i+1})
- S( E_{ki}, E_{i, i+1}, E_{i+1, k})
+ S( E_{ik}, E_{k, i+1}, E_{i+1, i}).
\Big)\end{align*}
\end{lemma}
\begin{proof}
We know that
\[ -[\nu_i, \nu_j]=\sum_{a,b,c} \Big([H_i, x_a], \big[ [X_j^+, x_b], [X_j^-, x_c]\big]\Big)\{x^a, x^b, x^c\}. \] As dual bases of $\mfsl_n$, we choose $\{ E_{ij}, H_k \, | \, 1\le i\neq j \le n, \, 1\le k \le n-1  \}$ and $\{ E_{ji}, H_k^* \, | \, 1\le i\neq j \le n,  \, 1\le k \le n-1 \}$ where $H_k^*= \sum_{i=1}^k E_{ii} - \frac{k}{n} \sum_{i=1}^n E_{ii}$. 
We now compute the right-hand side as follows: 
\begin{align*}
\Big([H_i, x_a], \big[ [X_j^+, x_b], [X_j^-, x_c]\big]\Big)\{x^a, x^b, x^c\} 
 = {} &
\sum_{k\neq l} \sum_{p\neq q}\sum_{s\neq t}
\Big( [H_i, E_{kl}], \big[ [E_{j, j+1}, E_{pq}], [E_{j+1, j}, E_{st}]\big]\Big)
\{E_{lk}, E_{qp}, E_{ts}\}\\
  & +\sum_{k\neq l} \sum_{m}\sum_{s\neq t}
\Big([H_i, E_{kl}], \big[[E_{j, j+1}, H_m^*], [E_{j+1, j}, E_{st}]\big]\Big)
\{E_{lk}, H_m, E_{ts}\}\\
 & +\sum_{k\neq l} \sum_{p\neq q}\sum_{m}
\Big([H_i, E_{kl}], \big[[E_{j, j+1}, E_{pq}], [E_{j+1, j}, H_m^*]\big]\Big)
\{E_{lk}, E_{qp}, H_{m}\}\\
 = {} &
\sum_{k\neq l} \sum_{p\neq q}\sum_{s\neq t}
\Big( [H_i, E_{kl}], \big[ [E_{j, j+1}, E_{pq}], [E_{j+1, j}, E_{st}]\big]\Big)
\{E_{lk}, E_{qp}, E_{ts}\}\\
 & + \sum_{k\neq l} \sum_{m}
\Big([H_i, E_{kl}], \big[[E_{j, j+1}, H_m^*], [E_{j+1, j}, E_{lk}]\big]\Big)
\{E_{lk}, H_m, E_{kl}\}\\
 &  +\sum_{k\neq l} \sum_{m}
\Big([H_i, E_{kl}], \big[[E_{j, j+1}, E_{lk}], [E_{j+1, j}, H_m^*]\big]\Big)
\{E_{lk}, E_{kl}, H_{m}\}\\
 = {} & \sum_{k\neq l} \sum_{p\neq q}\sum_{s\neq t}
\Big( [H_i, E_{kl}], \big[ [E_{j, j+1}, E_{pq}], [E_{j+1, j}, E_{st}]\big]\Big)
\{E_{lk}, E_{qp}, E_{ts}\}.
\end{align*}
Now under the assumption that $i\neq j$, we obtain:
\begin{align*}
\sum_{k\neq l} \sum_{p\neq q}\sum_{s\neq t} &
\Big( [H_i, E_{kl}], \big[ [E_{j, j+1}, E_{pq}], [E_{j+1, j}, E_{st}]\big]\Big)\{E_{lk}, E_{qp}, E_{ts}\}\\
 = {} & \sum_{k\neq l} \sum_{p\neq q}\sum_{s\neq t} 
((\epsilon_i-\epsilon_{i+1}, \epsilon_k-\epsilon_l)E_{kl} , [
\delta_{j+1, p} E_{jq}-\delta_{jq} E_{p, j+1}, 
\delta_{js} E_{j+1, t}-\delta_{j+1, t} E_{sj}]
)\{E_{lk}, E_{qp}, E_{ts}\}\\
 = {} &  \sum_{k\neq l} \sum_{p\neq q}\sum_{s\neq t} 
\Bigg((\delta_{ik}-\delta_{il}-\delta_{i+1, k}+\delta_{i+1, l})
E_{kl} , \Big(
\delta_{j+1, p} \delta_{js}[E_{jq},  E_{j+1, t}]
-\delta_{j+1, p} \delta_{j+1, t}[E_{jq},E_{sj}] \\
&  -\delta_{jq} \delta_{js}[E_{p, j+1}, E_{j+1, t}]
+\delta_{jq}\delta_{j+1, t}[E_{p, j+1}, E_{sj}]
\Big)\Bigg)\{E_{lk}, E_{qp}, E_{ts}\}\\
\end{align*}
\begin{align*}
 = {} & \sum_{k\neq l} \sum_{p\neq q}\sum_{s\neq t} 
\Bigg((\delta_{ik}-\delta_{il}-\delta_{i+1, k}+\delta_{i+1, l})
E_{kl} , \Big(
\delta_{j+1, p} \delta_{js}
(\delta_{q, j+1} E_{jt}-\delta_{tj} E_{j+1, q})
-\delta_{j+1, p} \delta_{j+1, t}
(\delta_{qs} E_{jj}-E_{sq}) \\
& -\delta_{jq} \delta_{js}
(E_{pt}-\delta_{pt} E_{j+1, j+1})
+\delta_{jq}\delta_{j+1, t}
( \delta_{j+1, s} E_{pj}-\delta_{pj} E_{s, j+1})
\Big)\Bigg)\{E_{lk}, E_{qp}, E_{ts}\}\\
 = {} & \sum_{k\neq l} \sum_{p\neq q}\sum_{s\neq t} 
\Bigg((\delta_{ik}-\delta_{il}-\delta_{i+1, k}+\delta_{i+1, l})
E_{kl} , \delta_{j+1, p} \delta_{js}\delta_{q, j+1} E_{jt}
-\delta_{j+1, p} \delta_{js}\delta_{tj} E_{j+1, q}
+\delta_{j+1, p} \delta_{j+1, t}E_{sq} \\
&  -\delta_{jq} \delta_{js} E_{pt}
+\delta_{jq}\delta_{j+1, t}\delta_{j+1, s} E_{pj}
-\delta_{jq}\delta_{j+1, t}\delta_{pj} E_{s, j+1}
\Bigg)\{E_{lk}, E_{qp}, E_{ts}\}\\
 = {} & \sum_{k\neq l} \sum_{p\neq q}\sum_{s\neq t} 
\Bigg((\delta_{ik}-\delta_{il}-\delta_{i+1, k}+\delta_{i+1, l})
E_{kl} , \delta_{j+1, p} \delta_{j+1, t}E_{sq}
-\delta_{jq} \delta_{js}E_{pt}
\Bigg)\{E_{lk}, E_{qp}, E_{ts}\}\\
 = {} & \sum_{k\neq l} \sum_{p\neq q}\sum_{s\neq t} 
\Bigg((\delta_{ik}-\delta_{il}-\delta_{i+1, k}+\delta_{i+1, l}) \delta_{j+1, p} \delta_{j+1, t}\delta_{ls}\delta_{kq}
-(\delta_{ik}-\delta_{il}-\delta_{i+1, k}+\delta_{i+1, l})
\delta_{jq} \delta_{js}\delta_{lp}\delta_{kt}
\Bigg)\{E_{lk}, E_{qp}, E_{ts}\} \\
 = {} &\sum_{\stackrel{l=1}{l\neq i, j+1}}^n \{ E_{li}, E_{i, j+1}, E_{j+1, l}\}
-\sum_{\stackrel{k=1}{k\neq i,  j+1}}^n \{ E_{ik}, E_{k, j+1}, E_{j+1, i}\}
-\sum_{\stackrel{l=1}{l\neq i+1, j+1}}^n\{ E_{l, i+1}, E_{i+1, j+1}, E_{j+1, l}\}
+\sum_{\stackrel{k=1}{k\neq i+1,  j+1}}^n\{ E_{i+1, k}, E_{k, j+1}, E_{j+1, i+1}\}
\\
&-\sum_{\stackrel{l=1}{l\neq i, j}}^n \{ E_{li}, E_{ij}, E_{jl}\}
+\sum_{\stackrel{k=1}{k\neq i, j}}^n\{ E_{ik}, E_{kj}, E_{ji}\}
+\sum_{\stackrel{l=1}{l\neq i+1, j}}^n\{ E_{l, i+1}, E_{i+1, j}, E_{jl}\}
-\sum_{\stackrel{k=1}{k\neq i+1, j}}^n \{ E_{i+1, k}, E_{kj}, E_{j, i+1}\}\\
 = {} & \sum_{l=1}^n\{ E_{li}, E_{i, j+1}, E_{j+1, l}\}
-\sum_{k=1}^n \{ E_{ik}, E_{k, j+1}, E_{j+1, i}\}
-\sum_{l=1}^n\{ E_{l, i+1}, E_{i+1, j+1}, E_{j+1, l}\}
+\sum_{k=1}^n\{ E_{i+1, k}, E_{k, j+1}, E_{j+1, i+1}\}
\\
&-\sum_{l=1}^n\{ E_{li}, E_{ij}, E_{jl}\}
+\sum_{k=1}^n\{ E_{ik}, E_{kj}, E_{ji}\}
+\sum_{l=1}^n\{ E_{l, i+1}, E_{i+1, j}, E_{jl}\}
-\sum_{k=1}^n\{ E_{i+1, k}, E_{kj}, E_{j, i+1}\}.
\end{align*}
Denote by $S(x, y, z)$ the sum $S(S(x, y), z) + (S(x, z), y)+S(S(y, z), x)$. We then have the equality:
\[
\frac{1}{48}S(z_1, z_2, z_3)
=
\frac{1}{24}\sum_{\sigma\in \mathfrak{S}_3}z_{\sigma(1)}z_{\sigma(2)}z_{\sigma(3)}
=\{z_1, z_2, z_3\}.
\]
Therefore,
\begin{align*}
[\nu_i, \nu_j]
 {} = & -\frac{1}{48}\sum_{k=1}^n  \Big(S( E_{ki}, E_{i, j+1}, E_{j+1, k})
- S(E_{ik}, E_{k, j+1}, E_{j+1, i})
- S( E_{k, i+1}, E_{i+1, j+1}, E_{j+1, k}) \\
&  + S( E_{i+1, k}, E_{k, j+1}, E_{j+1, i+1}) - S( E_{ki}, E_{ij}, E_{jk})
+S(E_{ik}, E_{kj}, E_{ji}) \\
& +  S( E_{k, i+1}, E_{i+1, j}, E_{jk})- S( E_{i+1, k}, E_{kj}, E_{j, i+1})
\Big).
\end{align*}
\end{proof}

\subsection{Computations for Lemma \ref{PKPQ}}
\label{appendix2}
To continue the computations for $[P(H_{i+1,i+2}),K(E_{i+1,i+2})]$ started in \eqref{PHi1i2}, let us determine $ [P(H_{ik}),K(E_{i+1,i+2})]$ (in Step 1), $\big[E_{ik}, [K(E_{k, i+2}), P(E_{i+1, i})]\big]$, $\big[E_{k, i+1},  [P(E_{i+1, k}), K(E_{i+1, i+2})]\big]$ (both in Step 2), $[K(E_{i,i+2}),R_{i+1,i+2,i+2,i}] $ (in Step 3) and $[ \wt{W}_{i,i+2} - W_{i,i+2} ,K(E_{i+1,i+2})]$ (in Step 4).

\noindent \textbf{Step 1:} To obtain the formula for $[P(H_{ik}),K(E_{i+1,i+2})]$, recall from the proof of Theorem \ref{central ele sln} that $[W_{ik} ,K(E_{i+1,i+2})]=0$ because $i,k,i+1,i+2$ are all distinct.
\begin{align*}
[P(H_{ik}),K(E_{i+1,i+2})] = {} & [ \wt{W}_{ik} + [K(E_{ik}),Q(E_{ki})] - W_{ik} ,K(E_{i+1,i+2})] \\
= {} & [ \wt{W}_{ik},K(E_{i+1,i+2})] + \big[K(E_{ik}),[Q(E_{ki}),K(E_{i+1,i+2})]\big] \\
= {} &  [ \wt{W}_{ik},K(E_{i+1,i+2})] + \frac{\la}{2} [K(E_{ik}), S(E_{i+1,i},E_{k,i+2}) ] \\
 = {} & \frac{\la}{4} \big( S(K(E_{i+1,i}),E_{i,i+2}) - S(K(E_{i+1,k}),E_{k,i+2}) + S(E_{i+1,i},K(E_{i,i+2})) - S(E_{i+1,k},K(E_{k,i+2})) \big). 
\end{align*}
\noindent \textbf{Step 2:}  
Applying $[-, E_{ik}]$ to the conclusion of \textbf{Step 1} and reordering the indices, we have
\begin{equation*}
[K(E_{k, i+2}), P(E_{i+1, i})] = - \frac{\la}{4}\big( S(K(E_{ki}), E_{i+1, i+2})+ S(K(E_{i+1, i+2}), E_{ki} \big)
\end{equation*}
so 
\begin{equation*}
\big[E_{ik}, [K(E_{k, i+2}), P(E_{i+1, i})]\big] = - \frac{\la}{4}\big(
 S(K(H_{ik}), E_{i+1, i+2})+ S(K(E_{i+1, i+2}), H_{ik} )\big).
\end{equation*}
Applying $[-, E_{i+1, k}]$ to the conclusion of \textbf{Step 1}, we have
\begin{equation*}
[P(E_{i+1, k}), K(E_{i+1, i+2})]\big]= \frac{\la}{4} \big(S(K(E_{i+1, k}), E_{i+1, i+2})+ S(K(E_{i+1, i+2}), E_{i+1, k} ) \big)
\end{equation*}
so 
\begin{align*}
\big[E_{k, i+1},  [P(E_{i+1, k}), K(E_{i+1, i+2})]\big]  = {} & \frac{\la}{4} \Big(
  S(K(H_{k, i+1}), E_{i+1, i+2})+ S(K(E_{i+1, k}), E_{k, i+2} ) \\
  & +S(K(E_{k, i+2}), E_{i+1, k})+ S(K(E_{i+1, i+2}), H_{k, i+1} )
\Big).
\end{align*}

\noindent \textbf{Step 3:} We now compute  $ [K(E_{i,i+2}),R_{i+1,i+2,i+2,i}]$.
Recall that:
\[
R_{i+1,i+2,i+2,i}
=-\left(\beta-\frac{\la}{2}\right) E_{i+1, i}
-\frac{\la}{4} \left(
\sum_{q=1}^n S(E_{i+1, q}, E_{qi})
-2 S(E_{i+1, i}, E_{i+2, i+2})
\right).
\]
Thus, 
\begin{align}
 [K(E_{i,i+2}), & R_{i+1,i+2,i+2,i}] \notag \\
 = {} & \Bigg[\left(\beta-\frac{\la}{2}\right) E_{i+1, i}
+\frac{\la}{4} \left(
\sum_{q=1}^n S(E_{i+1, q}, E_{qi})
-2 S(E_{i+1, i}, E_{i+2, i+2})
\right), K(E_{i,i+2})\Bigg] \notag \\
= {} & \left(\beta-\frac{\la}{2}\right) K(E_{i+1, i+2})
+\frac{\la}{4} \Bigg(
\sum_{q=1}^n S(K(E_{q, i+2}), E_{i+1, q})
+S(K(E_{i+1, i+2}), E_{ii})
-S(K(E_{ii}), E_{i+1, i+2}) \notag \\
&-2\Big(S(K(E_{i+1, i+2}), E_{i+2, i+2})
-S(K(E_{i, i+2}), E_{i+1, i})
\Big)
\Bigg). \label{KER}
\end{align}

\noindent \textbf{Step 4:}  We now compute $[W_{i,i+2} ,K(E_{i+1,i+2})]$. By Proposition \ref{prop:two param}, we have
\[W_{i, i+2} = \frac{1}{2}Z_{i, i+2}=\frac{1}{2}Z_{ik}+
\left( \beta - \frac{\la}{2} \right)H_{i+2,k}.\]
Therefore,
\begin{equation}
[W_{i,i+2} ,K(E_{i+1,i+2})]
 =  \left[ \frac{1}{2}Z_{ik}+
\left( \beta - \frac{\la}{2} \right)H_{i+2,k}, K(E_{i+1,i+2})\right]
= -\left( \beta - \frac{\la}{2} \right)K(E_{i+1,i+2}). \label{WKE}
\end{equation}
Let us turn our attention to  $[ \wt{W}_{i,i+2}, K(E_{i+1,i+2})]$:
\begin{align}
[ \wt{W}_{i,i+2}, K(E_{i+1,i+2})] = {} &  -\frac{\la}{4}\left[\sum_{p=1}^n
S(E_{ip}, E_{p, i})+\sum_{p=1}^n
S(E_{i+2, p}, E_{p, i+2})
-2 S(E_{ii}, E_{i+2, i+2}), 
K(E_{i+1,i+2})\right] \notag \\
 = {} & -\frac{\la}{4}\Bigg(
S(K(E_{i, i+2}), E_{i+1, i})-S(K(E_{i+1, i}), E_{i, i+2})
+
S(K(E_{i+2, i+2}), E_{i+1, i+2})-S(K(E_{i+1, i+2}), E_{i+2, i+2}) \notag \\
& -\sum_{p=1}^n S(K(E_{i+1, p}), E_{p, i+2})
+2 S(K(E_{i+1, i+2}), E_{ii})\Bigg). \label{wtWKE}
\end{align}

\noindent \textbf{Step 5:}  Let us now simplify $[K(E_{i,i+2}),R_{i+1,i+2,i+2,i}]  + [ \wt{W}_{i,i+2} - W_{i,i+2} ,K(E_{i+1,i+2})]$ using \eqref{KER},\eqref{WKE} and \eqref{wtWKE}:

\begin{align*}
 [K& (E_{i,i+2})  ,R_{i+1,i+2,i+2,i}]  + [ \wt{W}_{i,i+2} - W_{i,i+2} ,K(E_{i+1,i+2})] \\
 = {} & \left(\beta-\frac{\la}{2}\right) K(E_{i+1, i+2})
+\frac{\la}{4} \Bigg(
\sum_{q=1}^{n} S(K(E_{q, i+2}), E_{i+1, q})
+S(K(E_{i+1, i+2}), E_{ii})
-S(K(E_{i, i}), E_{i+1, i+2})\\
&-2\Big(
S(K(E_{i+1, i+2}), E_{i+2, i+2})
-S(K(E_{i, i+2}), E_{i+1, i})
\Big)
\Bigg)\\
&-\frac{\la}{4}\Bigg(
S(K(E_{i, i+2}), E_{i+1, i})-S(K(E_{i+1, i}), E_{i, i+2})
+
S(K(E_{i+2, i+2}), E_{i+1, i+2})-S(K(E_{i+1, i+2}), E_{i+2, i+2})\\
&-\sum_{p=1}^n S(K(E_{i+1, p}), E_{p, i+2})
+2 S(K(E_{i+1, i+2}), E_{ii})\Bigg)
+\left( \beta - \frac{\la}{2} \right)K(E_{i+1,i+2})\\
 = {} & 2\left(\beta-\frac{\la}{2}\right) K(E_{i+1, i+2})
+\frac{\la}{4} \Bigg(
\sum_{q=1}^n S(K(E_{q, i+2}), E_{i+1, q})
+\sum_{p=1}^n S(K(E_{i+1, p}), E_{p, i+2})
-S(K(E_{i+1, i+2}), E_{ii})\\
& -S(K(E_{i, i}), E_{i+1, i+2})
-S(K(E_{i+1, i+2}), E_{i+2, i+2})
+S(K(E_{i, i+2}), E_{i+1, i})
+S(K(E_{i+1, i}), E_{i, i+2})
-S(K(E_{i+2, i+2}), E_{i+1, i+2})
\Bigg)
\end{align*}

\noindent \textbf{Step 6:} The expressions above implies that the right-hand side of \eqref{PHi1i2} can be expanded and then simplified in the following way:
\begin{align*}
[P(H_{i+1,i+2}),&K(E_{i+1,i+2})] 
= -2[P(H_{i}),K(E_{i+1,i+2})] + 2\left(\beta-\frac{\la}{2}\right) K(E_{i+1, i+2})
\\
 &+ \frac{\la}{4} \big( S(K(E_{i+1,i}),E_{i,i+2}) - S(K(E_{i+1,k}),E_{k,i+2}) + S(E_{i+1,i},K(E_{i,i+2})) - S(E_{i+1,k},K(E_{k,i+2})) \big)
 \\
&+\frac{\lambda}{4}
\Big(
  S(K(H_{i, i+1}), E_{i+1, i+2})+ S(K(E_{i+1, k}), E_{k, i+2} )
 +S(K(E_{k, i+2}), E_{i+1, k})+ S(K(E_{i+1, i+2}), H_{i, i+1} )
\Big)\\
&+\frac{\la}{4} \Bigg(
\sum_{q=1}^n S(K(E_{q, i+2}), E_{i+1, q})
+\sum_{p=1}^n S(K(E_{i+1, p}), E_{p, i+2})
-S(K(E_{i+1, i+2}), E_{ii})
-S(K(E_{i, i}), E_{i+1, i+2})\\
& -S(K(E_{i+1, i+2}), E_{i+2, i+2})
+S(K(E_{i, i+2}), E_{i+1, i})
+S(K(E_{i+1, i}), E_{i, i+2})
-S(K(E_{i+2, i+2}), E_{i+1, i+2})
\Bigg)\\
 = {} & -2[P(H_{i}),K(E_{i+1,i+2})] + 2\left(\beta-\frac{\la}{2}\right) K(E_{i+1, i+2})
\\
 &+ \frac{\la}{4} \Big
 ( S(K(E_{i+1,i}),E_{i,i+2})  + S(E_{i+1,i},K(E_{i,i+2})) 
+S(K(H_{i, i+1}), E_{i+1, i+2})
 + S(K(E_{i+1, i+2}), H_{i, i+1} )
\Big)\\
&+\frac{\la}{4} \left(
\sum_{q=1}^n S(K(E_{q, i+2}), E_{i+1, q})
+\sum_{p=1}^n S(K(E_{i+1, p}), E_{p, i+2})
-S(K(E_{i+1, i+2}), E_{ii})
-S(K(E_{ii}), E_{i+1, i+2}) \right.\\
& \left.
-S(K(E_{i+1, i+2}), E_{i+2, i+2})
+S(K(E_{i, i+2}), E_{i+1, i})
+S(K(E_{i+1, i}), E_{i, i+2})
-S(K(E_{i+2, i+2}), E_{i+1, i+2})
\right) \\
 = {} & -2[P(H_{i}),K(E_{i+1,i+2})] + 2\left(\beta-\frac{\la}{2}\right) K(E_{i+1, i+2})
\\
&+ \frac{\la}{4} \Big
 ( 2S(K(E_{i+1,i}),E_{i,i+2})  
 + 2S(E_{i+1,i},K(E_{i,i+2})) 
+S(K(-E_{i+1, i+1}-E_{i+2, i+2}), E_{i+1, i+2})
\\
&+ S(K(E_{i+1, i+2}), -E_{i+1, i+1}-E_{i+2, i+2})
\Big) +\frac{\la}{4} \left(\sum_{q=1}^n S(K(E_{q, i+2}), E_{i+1, q})
+\sum_{p=1}^n S(K(E_{i+1, p}), E_{p, i+2})
\right)
\end{align*}
\begin{align*}
 = {} & -2[P(H_{i}),K(E_{i+1,i+2})] + 2\left(\beta-\frac{\la}{2}\right) K(E_{i+1, i+2})
\\
&+ \frac{\la}{2} \Big
 ( S(K(E_{i+1,i}),E_{i,i+2})  
 + S(E_{i+1,i},K(E_{i,i+2})) 
\Big)
+\frac{\la}{4} \sum_{\stackrel{p=1}{p\neq i+1, i+2}}^n \Big(
S(K(E_{p, i+2}), E_{i+1, p})
+ S(K(E_{i+1, p}), E_{p, i+2})
\Big)\\
\end{align*}
Lemma \ref{PKPQ} follows directly from this last expression.

\subsection{Detailed proof of Lemma \ref{lem:P(H)}, Case 1}
\label{appendix3}
\noindent \textbf{Step 1:} The first term on the right-hand side of \eqref{PXiXj} is: 
\begin{align*}
\big[ [K(X_i^+), P(X_j^+)], Q(H_i)\big] = {} & -\frac{\la}{4}[ S(K(E_{i, j+1}), E_{j, i+1}) + S(K(E_{j, i+1}), E_{i, j+1}), Q(H_i)]\\
 = {} & -\frac{\la}{4}\Big(
S([ K(E_{i, j+1}), Q(H_i)] ,  E_{j, i+1})- S(K(E_{i, j+1}), Q(E_{j, i+1})) \\
& +S([ K(E_{j, i+1}), Q(H_i)], E_{i, j+1})-S(K(E_{j, i+1}), Q(E_{i, j+1}))\Big)
\end{align*}
The second term on the right-hand side of \eqref{PXiXj} is:
\begin{align*}
\big[ K(X_i^+), &[Q(H_i), P(X_j^+)] \big] \\
 = {} & \Big[ K(X_i^+), \big[ [Q(X_i^+), P(X_j^+)], X_i^- \big] \Big] \\
 = {} &  \frac{\la}{4}\Big[ K(X_i^+), 
[S(Q(E_{i, j+1}), E_{j, i+1})+ S(Q(E_{j, i+1}), E_{i, j+1}) , E_{i+1, i}\big] \Big] \\
 = {} &  \frac{\la}{4}[K(X_i^+), 
-S(Q(E_{i+1, j+1}),  E_{j, i+1})
+S(Q(E_{i, j+1}), E_{ji})
+S(Q(E_{ji}), E_{i, j+1})
-S(Q(E_{j, i+1}), E_{i+1, j+1})]\\
 = {} &  \frac{\la}{4}
\Big(
-S\big( [K(X_i^+), Q(E_{i+1, j+1})],  E_{j, i+1}\big)
+S\big( [K(X_i^+), Q(E_{i, j+1})], E_{j, i}\big)
-S(Q(E_{i, j+1}),  K(E_{j, i+1}))\\
&
+S\big( [K(X_i^+), Q(E_{j, i})], E_{i, j+1}\big)
-S\big([K(X_i^+), Q(E_{j, i+1})], E_{i+1, j+1}\big)
-S(Q(E_{j, i+1}), K(E_{i, j+1})
\Big)
\end{align*}
\noindent \textbf{Step 2:} The sum of the first and the second terms of \eqref{PXiXj} is:
\begin{align}
 \big[ [ K & (X_i^+),  P(X_j^+)], Q(H_i)]
+[K(X_i^+), [Q(H_i), P(X_j^+)]\big] \notag \\
= {} & 
\frac{-\la}{4}\Big(
S([ K(E_{i, j+1}), Q(H_i)] ,  E_{j, i+1}) - S(K(E_{i, j+1}), Q(E_{j, i+1}))
+ S([ K(E_{j, i+1}), Q(H_i)], E_{i, j+1}) - S(K(E_{j, i+1}), Q(E_{i, j+1}))\Big) \notag \\
&+\frac{\la}{4}
\Bigg(
-S([K(X_i^+), Q(E_{i+1, j+1})],  E_{j, i+1})
+S([K(X_i^+), Q(E_{i, j+1})], E_{ji})
-S(Q(E_{i, j+1}),  K(E_{j, i+1})) \notag \\
&
+S([K(X_i^+), Q(E_{ji})], E_{i, j+1})
-S([K(X_i^+), Q(E_{j, i+1})], E_{i+1, j+1})]
-S(Q(E_{j, i+1}), K(E_{i, j+1})
\Bigg) \notag \\
= {} & \frac{\la}{4}\Big(
-S([ K(E_{i, j+1}), Q(H_i)] ,  E_{j, i+1})
-S([K(X_i^+), Q(E_{i+1, j+1})],  E_{j, i+1})\Big) \notag \\
&+\frac{\la}{4}\Big(
-S([ K(E_{j, i+1}), Q(H_i)], E_{i, j+1})
+S([K(X_i^+), Q(E_{ji})], E_{i, j+1})
\Big) \notag \\
&+\frac{\la}{4}
\Bigg(
+S([K(X_i^+), Q(E_{i, j+1})], E_{ji})
-S([K(X_i^+), Q(E_{j, i+1})], E_{i+1, j+1})
\Bigg). \label{KPQKQP}
\end{align}
We need the following relations
(the first and the third equation follows from Lemma \ref{lem:KEab QHcd}, the rest are the defining relations):
\begin{align*}
[ K(E_{i, j+1}), Q(H_i)]&
=-P(E_{i, j+1})+\left(\beta-\frac{\la}{2}\right) E_{i, j+1} +\frac{\la}{4}
\left( \sum_{k=1}^n S(E_{k, j+1}, E_{ik})
-2 S(E_{i, j+1}, E_{ii})+2S(E_{i, i+1}, E_{i+1, j+1})
\right)\\
[K(X_i^+), Q(E_{i+1, j+1})]&
=P(E_{i, j+1})+\left(\beta-\frac{\la}{2}\right) E_{i, j+1} +\frac{\la}{4}
\left( 
\sum_{l=1}^n S(E_{il}, E_{l, j+1})
-2S(E_{i, j+1}, E_{i+1, i+1})
\right)
\end{align*}
\begin{align*}
[ K(E_{j, i+1}), Q(H_i)]&
=-P(E_{j, i+1})-\left(\beta-\frac{\la}{2}\right) E_{j, i+1} +\frac{\la}{4}
\left(
-\sum_{l=1}^n S(E_{jl}, E_{l, i+1}) +2S(E_{j, i+1}, E_{i+1, i+1})-2S(E_{ji}, E_{i, i+1})
\right) \\
[K(X_i^+), Q(E_{j, i})]&
=-P(E_{j, i+1})+\left(\beta-\frac{\la}{2}\right) E_{j, i+1}
+\frac{\la}{4}\left(
-2S(E_{ii}, E_{j, i+1})+\sum_{k=1}^n S(E_{k, i+1}, E_{jk})
\right)
\\
[K(X_i^+), Q(E_{i, j+1})]&
=-\frac{\la}{2} S(E_{i, j+1}, E_{i, i+1}), \;\; 
[K(X_i^+), Q(E_{j, i+1})]
=-\frac{\la}{2} S(E_{i, i+1}, E_{j, i+1})
\end{align*}
Using these and \eqref{KPQKQP}, we get a formula for the sum of the first and the second term of \eqref{PXiXj}: 
\begin{align}
 [[K( & X_i^+),  P(X_j^+)], Q(H_i)]
  +[K(X_i^+), [Q(H_i), P(X_j^+)]] \notag \\
 = {} & -\frac{\la^2}{16}
S\left( \sum_{k=1}^n S(E_{k, j+1}, E_{ik})
-2 S(E_{i, j+1}, E_{ii})+2S(E_{i, i+1}, E_{i+1, j+1})
+ \sum_{l=1}^n S(E_{il}, E_{l, j+1})
-2S(E_{i, j+1}, E_{i+1, i+1}) ,  E_{j, i+1}\right) \notag \\
&+\frac{\la^2}{16}
S\left( \sum_{l=1}^n S(E_{jl}, E_{l, i+1}) -2S(E_{j, i+1}, E_{i+1, i+1}) + 2S(E_{ji}, E_{i, i+1})
-2S(E_{ii}, E_{j, i+1})+\sum_{k=1}^n S(E_{k, i+1}, E_{jk}), E_{i, j+1} \right) \notag \\
& - \frac{\la^2}{8}
\Bigg(
+S(S(E_{i, j+1}, E_{i, i+1}), E_{ji})
-S(S(E_{i, i+1}, E_{j, i+1})
, E_{i+1, j+1})
\Bigg) \notag \\
 = {} & -\frac{\la^2}{8}
S\Bigg(
 \sum_{k} S(E_{k, j+1}, E_{ik})
- S(E_{i, j+1}, E_{ii})+S(E_{i, i+1}, E_{i+1, j+1})
-S(E_{i, j+1}, E_{i+1, i+1})
 ,  E_{j, i+1}\Bigg) \notag \\
&-\frac{\la^2}{8}
S\Bigg(
-\sum_{l} S(E_{jl}, E_{l, i+1}) +S(E_{j, i+1}, E_{i+1, i+1})-S(E_{ji}, E_{i, i+1})+S(E_{ii}, E_{j, i+1})
, E_{i, j+1}
\Bigg) \notag \\
&-\frac{\la^2}{8}
\Bigg(
+S( S(E_{i, j+1}, E_{i, i+1}), E_{ji})
-S( S(E_{i, i+1}, E_{j, i+1})
, E_{i+1, j+1})
\Bigg) \notag \\
 = {} & 
-\frac{\la^2}{8}
S\Bigg(
 \sum_{k} S(E_{k, j+1}, E_{ik}), E_{j, i+1}\Bigg)
+ \frac{\la^2}{8}
S\Bigg(
\sum_{l} S(E_{jl}, E_{l, i+1}), E_{i, j+1} \Bigg). \label{KPQKQP2}
\end{align}
where the last equality follows from the equality $S(S(A, B), C)-S(S(A, C), B)=[A, [B, C]]$.

\noindent \textbf{Step 3:} The third term  of \eqref{PXiXj} is, up to a scalar: 
\begin{equation*}
\sum_{\alpha\in\Delta} [S([X_i^+, X_{\alpha}], [X_{-\alpha}, H_i]), P(X_j^+)] = -2[ S(E_{i, i+1}, H_i),  P(E_{j, j+1})] = 0.
\end{equation*}
We can now deduce from \eqref{PXiXj} and \eqref{KPQKQP2} the following formula for $[P(X_i^+), P(X_j^+)]$: 
\begin{equation} 
[P(X_i^+), P(X_j^+)] = \frac{\la^2}{16} S\Bigg(  \sum_{k} S(E_{k, j+1}, E_{ik}), E_{j, i+1}\Bigg)
- \frac{\la^2}{16} S\Bigg( \sum_{l} S(E_{jl}, E_{l, i+1}), E_{i, j+1} \Bigg). \label{PXiXj2} 
\end{equation}
\noindent \textbf{Step 4:} We will also use the following two formulas:
\begin{align}
 \sum_{k=1}^n \Big[\big[S( S(E_{k, j+1}, E_{i, k}), E_{j, i+1}),  X_j^-\big],  X_i^-\Big] = {} & \sum_{k=1}^n \big[ S(S(E_{kj}, E_{ik}), E_{j, i+1}), E_{i+1, i}\big]
 -\sum_{k=1}^n \big[ S(S(E_{k, j+1}, E_{ik}), E_{j+1, i+1}), E_{i+1, i}\big] \notag \\
 = {} & \sum_{k=1}^n S(S(E_{kj}, E_{ik}), E_{ji})
 -\sum_{k=1}^n S(S(E_{kj}, E_{i+1, k}), E_{j, i+1}) \notag   \\
 {} & +\sum_{k=1}^n S(S(E_{k, j+1}, E_{i+1, k}), E_{j+1, i+1})
 -\sum_{k=1}^n S(S(E_{k, j+1}, E_{ik}), E_{j+1, i}). \label{EEEXX1}
\end{align}
Exchanging $i\leftrightarrow j$ and using $[X_i^-,X_j^-]=0$, we obtain:

\begin{align}
\sum_{l=1}^n \Big[ \big[ S(
 S(E_{jl}, E_{l, i+1}), E_{i, j+1}), X_j^- \big], X_i^- \big] = {} & \sum_{l=1}^n 
S(S(E_{j+1, l}, E_{l, i+1}), E_{i+1, j+1})
-\sum_{l=1}^n
S(S(E_{j+1, l}, E_{li}), E_{i, j+1}) \notag \\
& +\sum_{l=1}^n
S(S(E_{jl}, E_{li}), E_{ij})
-\sum_{l=1}^n
S(S(E_{jl}, E_{l, i+1}), E_{i+1, j}).  \label{EEEXX2}
\end{align}
The right-hand side of \eqref{PXiXj2} can be substituted into the right-hand side of $[P(H_i), P(H_j)] = \Big[\big[ [P(X_i^+), P(X_j^+)], X_j^-\big] ,  X_i^-\Big]$ and then formulas \eqref{EEEXX1} and \eqref{EEEXX2} can be applied to finally obtain:
\begin{align*}
[P & (H_i), P(H_j)] = \frac{\la^2}{16}
\sum_{k=1}^n \Big[\big[ S(S(E_{k, j+1}, E_{ik}), E_{j, i+1}), X_j^-\big] ,  X_i^-\Big]
- \frac{\la^2}{16} \sum_{l=1}^n \Big[\big[ S(S(E_{jl}, E_{l, i+1}), E_{i, j+1} ), X_j^-\big] ,  X_i^-\Big]
\\
 = {} & \frac{\la^2}{16} \sum_{k=1}^n
\big( S(S(E_{kj}, E_{ik}), E_{ji})
 - S(S(E_{kj}, E_{i+1, k}), E_{j, i+1})
 + S(S(E_{k, j+1}, E_{i+1, k}), E_{j+1, i+1})
 - S(S(E_{k, j+1}, E_{ik}), E_{j+1, i})
\big)\\
&- \frac{\la^2}{16} \sum_{l=1}^n  \big(
S(S(E_{j+1, l}, E_{l, i+1}), E_{i+1, j+1})
-S(S(E_{j+1, l}, E_{li}), E_{i, j+1})
+S(S(E_{jl}, E_{li}), E_{ij})
-S(S(E_{jl}, E_{l, i+1}), E_{i+1, j})\big).
\end{align*}
Comparing this with the formula for $[\nu_i,\nu_j]$ given in Lemma \ref{lem:nu2} completes the proof of Lemma \ref{lem:P(H)} when $|i-j|>1$. For this, one needs the observation that for any three elements $x,y,z$ of $\mfU\mfsl_n$: \[ \{ x,y,z \} = \frac{1}{24} \big( xzy+yzx+ yxz+zxy + \frac{1}{2}(yxz + [x,y]z + xzy + x[y,z] + yzx + [z,y]x + zxy + z[y,x]) \big). \]

\subsection{Remainder of the proof of Lemma \ref{lem:P(H)}, Case 2}
\label{appendix4}
Here are more details to show that the right-hand side of \eqref{PHiPHi1} equals $-\la^2[v_i, v_{i+1}]$.

\noindent \textbf{Step 1:} From Lemma \ref{PKPQ}, we know:
\begin{align*}
S_i = {} &
 \left(\beta-\frac{\la}{2}\right) K(E_{i+1, i+2})
+ \frac{\la}{4} \Big
 ( S(K(E_{i+1,i}),E_{i,i+2})  
 + S(E_{i+1,i},K(E_{i,i+2})) 
\Big)\\
&+\frac{\la}{8} \sum_{\stackrel{p=1}{p\neq i+1, i+2}}^n\Bigg(
S(K(E_{p, i+2}), E_{i+1, p})
+ S(K(E_{i+1, p}), E_{p, i+2})
\Bigg)
 \end{align*} 
and
\begin{align*}
\wt{S}_i= {} &
 -\left(\beta-\frac{\la}{2}\right) Q(E_{i+2, i+1})
- \frac{\la}{4} \Big
 ( S(Q(E_{i, i+1}),E_{i+2, i})  
 + S(E_{i, i+1},Q(E_{i+2, i})) 
\Big)\\
&-\frac{\la}{8} \sum_{\stackrel{p=1}{p\neq i+1, i+2}}^n\Bigg(
S(Q(E_{i+2, p}), E_{p, i+1})
+ S(Q(E_{p, i+1}), E_{i+2, p})
\Bigg).
 \end{align*} 
 Therefore,
  \begin{align*}
  [S_i,Q(E_{i+2,i+1})] + {} & [K(E_{i+1,i+2}),\wt{S}_i] \notag \\
 = {} & 
 \frac{\la}{4} \Big[\big
 ( S(K(E_{i+1,i}),E_{i,i+2})  
 + S(E_{i+1,i},K(E_{i,i+2})) 
\big), Q(E_{i+2,i+1})\Big] \notag \\
&+\frac{\la}{8} \sum_{\stackrel{p=1}{p\neq i+1, i+2}}^n \Big[\big(
S(K(E_{p, i+2}), E_{i+1, p})
+ S(K(E_{i+1, p}), E_{p, i+2})
\big), Q(E_{i+2,i+1})\Big] \notag \\
&
- \frac{\la}{4} \Big[K(E_{i+1,i+2}), \big
 ( S(Q(E_{i, i+1}),E_{i+2, i})  
 + S(E_{i, i+1},Q(E_{i+2, i})) 
\big)\Big] \notag \\
&-\frac{\la}{8} \sum_{\stackrel{p=1}{p\neq i+1, i+2}}^n \Big[K(E_{i+1,i+2}), \big(
S(Q(E_{i+2, p}), E_{p, i+1})
+ S(Q(E_{p, i+1}), E_{i+2, p})
\big)\Big] \notag \\
\end{align*}
\begin{align}
 = {} & 
 \frac{\la}{4} \Big( S\big( [K(E_{i+1,i}),Q(E_{i+2,i+1})], E_{i,i+2}\big)  
 + S\big( [K(E_{i,i+2}),Q(E_{i+2,i+1})], E_{i+1,i}\big) 
\Big) \notag \\
&+\frac{\la}{8} \sum_{\stackrel{p=1}{p\neq i+1, i+2}}^n \Big(
S\big( [K(E_{p, i+2}), Q(E_{i+2,i+1})], E_{i+1, p}\big)
+ S\big( [K(E_{i+1, p}), Q(E_{i+2,i+1})], E_{p, i+2}\big)
\Big) \notag \\
&
- \frac{\la}{4}\Big( 
S\big( [K(E_{i+1,i+2}), Q(E_{i, i+1})],E_{i+2, i}\big)  
 + S\big( [K(E_{i+1,i+2}), Q(E_{i+2, i})], E_{i, i+1}\big) 
\Big) \notag \\
&-\frac{\la}{8} \sum_{\stackrel{p=1}{p\neq i+1, i+2}}^n \Big(
S\big( [K(E_{i+1,i+2}), Q(E_{i+2, p})], E_{p, i+1}\big)
+ S\big( [K(E_{i+1,i+2}), Q(E_{p, i+1})], E_{i+2, p}\big)
\Big). \label{SQKS}
 \end{align}
 
\noindent \textbf{Step 2:} Using the following relations of the deformed double current algebra of $\mathfrak{sl}_n$, 
 \begin{align*}
&[K(E_{ab}), Q(E_{bd})]
=P(E_{ad})+\left(\beta-\frac{\la}{2}\right) E_{ad}+\frac{\la}{4}
\sum_{t=1}^n S(E_{at}, E_{td})-\frac{\la}{2}S(E_{ad}, E_{bb})
\\
&[K(E_{ab}), Q(E_{ca})]
=-P(E_{cb})+\left(\beta-\frac{\la}{2}\right) E_{cb}+\frac{\la}{4}
\sum_{t=1}^n S(E_{ct}, E_{tb})-\frac{\la}{2}S(E_{aa}, E_{cb}),
\end{align*}
we get the following equalities:
\begin{align*}
&[K(E_{i+1, i}), Q(E_{i+2, i+1})]
=-P(E_{i+2, i})+\left(\beta-\frac{\la}{2}\right)E_{i+2, i}+\frac{\la}{4}
\sum_{t=1}^n S(E_{ti}, E_{i+2, t})-\frac{\la}{2}S(E_{i+1, i+1}, E_{i+2, i})\\
&[K(E_{i,i+2}),Q(E_{i+2,i+1})]
=P(E_{i, i+1})+\left(\beta-\frac{\la}{2}\right)E_{i, i+1}+\frac{\la}{4}
\sum_{t=1}^n S(E_{it}, E_{t, i+1})-\frac{\la}{2}S(E_{i, i+1}, E_{i+2, i+2})
\\
&[K(E_{p, i+2}), Q(E_{i+2,i+1})]
=P(E_{p, i+1})+\left(\beta-\frac{\la}{2}\right)E_{p, i+1}+\frac{\la}{4}
\sum_{t=1}^n S(E_{pt}, E_{t, i+1})-\frac{\la}{2}S(E_{p, i+1}, E_{i+2, i+2})\\
&[K(E_{i+1, p}), Q(E_{i+2,i+1})]
=-P(E_{i+2,p})+\left(\beta-\frac{\la}{2}\right)E_{i+2,p}+\frac{\la}{4}
\sum_{t=1}^n S(E_{i+2,t}, E_{tp})-\frac{\la}{2}S(E_{i+1, i+1}, E_{i+2,p})\\
&[K(E_{i+1,i+2}), Q(E_{i, i+1})]
=-P(E_{i, i+2})+\left(\beta-\frac{\la}{2}\right)E_{i, i+2}+\frac{\la}{4}
\sum_{t=1}^n S(E_{it}, E_{t, i+2})-\frac{\la}{2}S(E_{i+1, i+1}, E_{i, i+2})\\
&[K(E_{i+1,i+2}), Q(E_{i+2, i})]
=P(E_{i+1, i})+\left(\beta-\frac{\la}{2}\right)E_{i+1,i}+\frac{\la}{4}
\sum_{t=1}^n S(E_{i+1, t}, E_{ti})-\frac{\la}{2}S(E_{i+1, i}, E_{i+2, i+2})\\
&[K(E_{i+1,i+2}), Q(E_{i+2, p})]
=P(E_{i+1, p})+\left(\beta-\frac{\la}{2}\right)E_{i+1, p}+\frac{\la}{4}
\sum_{t=1}^n S(E_{i+1, t}, E_{tp})-\frac{\la}{2}S(E_{i+1, p}, E_{i+2, i+2}) \\
&[K(E_{i+1,i+2}), Q(E_{p, i+1})]
=-P(E_{p, i+2})+\left(\beta-\frac{\la}{2}\right)E_{p, i+2}+\frac{\la}{4}
\sum_{t=1}^n S(E_{pt}, E_{t, i+2})-\frac{\la}{2}S(E_{i+1, i+1}, E_{p, i+2}).
\end{align*}

\noindent \textbf{Step 3:} Substituting the above equalities into \eqref{SQKS}, we obtain a long expression. The sum of the terms that involve $P(X)$ in that long expression is:
\begin{align*}
 \frac{\la}{4} \Big
 ( - & S(P(E_{i+2, i}), E_{i,i+2})  
 +S(P(E_{i,i+2}),E_{i+2, i})  
 + S(P(E_{i,i+1}), E_{i+1,i}) 
  - S(P(E_{i+1,i}), E_{i, i+1}) 
\Big) \notag \\
&+\frac{\la}{8} \sum_{\stackrel{p=1}{p\neq i+1, i+2}}^n \Big(
S(P(E_{p, i+1}), E_{i+1, p})
-S(P(E_{i+1, p}), E_{p, i+1})
- S(P(E_{i+2, p}), E_{p, i+2})
+ S(P(E_{p,i+2}), E_{i+2, p})
\Big) \notag 
\end{align*}
\begin{align}
 = {} & 
 \frac{\la}{4} \Big
 ( -S(P(E_{i+2, i}), E_{i,i+2})  
 +S(P(E_{i,i+2}),E_{i+2, i})  
 + S(P(E_{i,i+1}), E_{i+1,i}) 
  - S(P(E_{i+1,i}), E_{i, i+1}) 
\Big) \notag \\
&+\frac{\la}{8} \sum_{p=1}^n\Big(
S(P(E_{p, i+1}), E_{i+1, p})
-S(P(E_{i+1, p}), E_{p, i+1})
- S(P(E_{i+2, p}), E_{p, i+2})
+ S(P(E_{p,i+2}), E_{i+2, p})
\Big). \label{SQKSP}
\end{align}
\noindent \textbf{Step 4:}  In the long expression for \eqref{SQKS}, the sum of the terms involving $\left(\beta-\frac{\la}{2}\right) E_{ab}$ for some $a,b$ is zero. The sum of the terms involving $\sum_{t=1}^n S(E_{at},E_{tb})$ for various $a,b$ is: 
\begin{align}
 \frac{\la^2}{16}\sum_{t=1}^n  \Big( S\big( S(E_{ti}, E_{i+2, t}), & E_{i,i+2}\big)  
 + S\big( S(E_{it}, E_{t, i+1}), E_{i+1,i}\big) \Big) \notag \\
+\frac{\la^2}{32} \sum_{\stackrel{p=1}{p\neq i+1, i+2}}^n & \sum_{t=1}^n \Big( 
S\big( S(E_{pt}, E_{t, i+1}), E_{i+1, p}\big)
+ S\big( S(E_{i+2,t}, E_{tp}), E_{p, i+2}\big)
\Big) \notag \\
- \frac{\la^2}{16}\sum_{t=1}^n \Big( 
S\big( S(E_{it}, E_{t, i+2}), & E_{i+2,i}\big)  
 + S\big( S(E_{i+1, t}, E_{ti}), E_{i, i+1}\big) 
\Big) \notag \\
-\frac{\la^2}{32} \sum_{\stackrel{p=1}{p\neq i+1, i+2}}^n & \sum_{t=1}^n  \Big( 
S\big( S(E_{i+1, t}, E_{tp}), E_{p, i+1}\big)
+ S\big(S(E_{pt}, E_{t, i+2}), E_{i+2, p}\big)
\Big)  \notag \\
 = {} & \frac{\la^2}{16} \sum_{t=1}^n \Big
 ( S\big( S(E_{ti}, E_{i+2, t}), E_{i,i+2}\big)  
 + S\big( S(E_{it}, E_{t, i+1}), E_{i+1,i}\big) 
\Big) \notag \\
&+\frac{\la^2}{32} \sum_{p=1}^n \sum_{t=1}^n \Big(
S\big( S(E_{pt}, E_{t, i+1}), E_{i+1, p}\big)
+ S\big( S(E_{i+2,t}, E_{tp}), E_{p, i+2}\big)
\Big) \notag \\
&-\frac{\la^2}{32} (\sum_{t=1}^n \Big(
S\big( S(E_{i+1,t}, E_{t, i+1}), E_{i+1, i+1}\big)
+ S\big( S(E_{i+2,t}, E_{t, i+1}), E_{i+1, i+2}\big)
\Big) \notag \\
&-\frac{\la^2}{32} \sum_{t=1}^n \Big(
S\big( S(E_{i+2,t}, E_{t, i+1}), E_{i+1, i+2}\big)
+ S\big( S(E_{i+2,t}, E_{t, i+2}), E_{i+2, i+2}\big)
\Big) \notag \\
&
- \frac{\la^2}{16} \sum_{t=1}^n \Big( 
S\big( S(E_{it}, E_{t, i+2}),E_{i+2, i}\big)  
 + S\big( S(E_{i+1, t}, E_{ti}), E_{i, i+1}\big) 
\Big) \notag \\
&-\frac{\la^2}{32}  \sum_{p=1}^n \sum_{t=1}^n \Big(
S\big( S(E_{i+1, p}, E_{pt}), E_{t, i+1}\big)
+ S\big( S(E_{tp}, E_{p, i+2}), E_{i+2, t}\big)
\Big) \notag \\
&+\frac{\la^2}{32} \sum_{t=1}^n \Big(
S\big( S(E_{i+1, t}, E_{t, i+1}), E_{i+1, i+1}\big)
+ S\big( S(E_{i+1, t}, E_{t, i+2}), E_{i+2, i+1}\big)
\Big) \notag \\
&+\frac{\la^2}{32} \sum_{t=1}^n\Big(
S\big( S(E_{i+1, t}, E_{t, i+2}), E_{i+2, i+1}\big)
+ S\big( S(E_{i+2, t}, E_{t, i+2}), E_{i+2, i+2}\big)
\Big)  \notag \\
 = {} &
 \frac{\la^2}{16} \sum_{t=1}^n \Big
 ( S\big( S(E_{ti}, E_{i+2, t}), E_{i,i+2}\big)  
 + S\big( S(E_{it}, E_{t, i+1}), E_{i+1,i}\big) 
 - S\big( S(E_{i+2,t}, E_{t, i+1}), E_{i+1, i+2}\big)
\Big) \notag \\
&
- \frac{\la^2}{16} \sum_{t=1}^n\Big( 
S\big( S(E_{it}, E_{t, i+2}),E_{i+2, i}\big)  
 + S\big( S(E_{i+1, t}, E_{ti}), E_{i, i+1}\big) 
 - S\big( S(E_{i+1, t}, E_{t, i+2}), E_{i+2, i+1}\big)
\Big). \label{SQKSEE}
 \end{align}
\noindent \textbf{Step 5:} In the long expression for \eqref{SQKS}, the sum of the terms involving $S(E_{aa},E_{cb})$ for various $a,b,c$ is the following:
\begin{align*}
- \frac{\la^2}{8} \Big
 ( S\big( S(E_{i+1, i+1}, E_{i+2, i}), E_{i,i+2}\big)  
 & + S\big( S(E_{i, i+1}, E_{i+2, i+2}), E_{i+1,i}\big) 
\Big) \notag \\
-\frac{\la^2}{16} \sum_{\stackrel{p=1}{p\neq i+1, i+2}}^n & \Big(
S\big( S(E_{p, i+1}, E_{i+2, i+2}), E_{i+1, p}\big)
+ S\big( S(E_{i+1, i+1}, E_{i+2,p}), E_{p, i+2}\big)
\Big) \notag 
\end{align*}
\begin{align}
+ \frac{\la^2}{8}\Big( 
S\big( S(E_{i+1, i+1}, E_{i, i+2}),E_{i+2, i}\big)  
 & + S\big( S(E_{i+1, i}, E_{i+2, i+2}), E_{i, i+1}\big) 
\Big) \notag \\
+\frac{\la^2}{16} \sum_{\stackrel{p=1}{p\neq i+1, i+2}}^n & \Big(
S\big( S(E_{i+1, p}, E_{i+2, i+2}), E_{p, i+1}\big)
+ S\big( S(E_{i+1, i+1}, E_{p, i+2}), E_{i+2, p}\big)
\Big) \notag \\
 = {} &
- \frac{\la^2}{8} \Big(
 \big[ E_{i+1, i+1}, [E_{i+2, i}, E_{i,i+2}]\big]
 + \big[ E_{i+2, i+2}, [E_{i, i+1}, E_{i+1, i}]\big]
\Big) \notag \\
&-\frac{\la^2}{16}  \sum_{\stackrel{p=1}{p\neq i+1, i+2}}^n \Big(
\big[ E_{i+2, i+2}, [E_{p, i+1}, E_{i+1, p}]\big]
+
\big[ E_{i+1, i+1}, [E_{i+2,p},  E_{p, i+2}]\big]
\Big) \notag \\
 = {} & 0. \label{SEaaEcb}
 \end{align}

\noindent \textbf{Step 6:} As a consequence, from \eqref{SQKS}, \eqref{SQKSP}, \eqref{SQKSEE} and \eqref{SEaaEcb}, we obtain:
  \begin{align}
  [S_i,Q(E_{i+2,i+1})] & + [K(E_{i+1,i+2}),\wt{S}_i] \notag \\
 = {} &
 \frac{\la}{4} \Big
 ( -S(P(E_{i+2, i}), E_{i,i+2})  
 +S(P(E_{i,i+2}),E_{i+2, i})  
 + S(P(E_{i,i+1}), E_{i+1,i}) 
  - S(P(E_{i+1,i}), E_{i, i+1}) 
\Big) \notag \\
&+\frac{\la}{8} \sum_{p=1}^n\Big(
S(P(E_{p, i+1}), E_{i+1, p})
-S(P(E_{i+1, p}), E_{p, i+1})
- S(P(E_{i+2, p}), E_{p, i+2})
+ S(P(E_{p,i+2}), E_{i+2, p})
\Big) \notag \\
&+ \frac{\la^2}{16} \sum_{t=1}^n \Big
 ( S\big( S(E_{ti}, E_{i+2, t}), E_{i,i+2}\big)  
 + S\big( S(E_{it}, E_{t, i+1}), E_{i+1,i}\big)
 - S\big( S(E_{i+2,t}, E_{t, i+1}), E_{i+1, i+2}\big) \Big) \notag \\
& - \frac{\la^2}{16}(\sum_{t=1}^n \Big( 
S \big( S(E_{it}, E_{t, i+2}),E_{i+2, i}\big)  
 + S\big( S(E_{i+1, t}, E_{ti}), E_{i, i+1}\big) 
 - S\big( S(E_{i+1, t}, E_{t, i+2}), E_{i+2, i+1}\big)
\Big.) \label{SQKwtS}
  \end{align}

\noindent \textbf{Step 7:} To continue with our computation of the right-hand side of \eqref{PHiPHi1}, we now determine $[P(H_i)+\frac{1}{2}P(H_{i+1}), \wt{W}_{i+1,i+2}]$:

Recall that
\[
\wt{W}_{i+1,i+2}=
-\frac{\la}{4}\left(
\sum_{p=1}^n S(E_{i+1, p}, E_{p, i+1})+
\sum_{p=1}^n S(E_{i+2, p}, E_{p, i+2})
-2 S(E_{i+1, i+1}, E_{i+2, i+2})
\right). 
\]
Therefore, 
\begin{align}
[P(H_i)+ & \frac{1}{2}P(H_{i+1}), \wt{W}_{i+1,i+2}]  \notag \\
= {} & -\frac{\la}{4} \sum_{p=1}^n 
[P(H_i)+\frac{1}{2}P(H_{i+1}), S(E_{i+1, p}, E_{p, i+1})+
S(E_{i+2, p}, E_{p, i+2})]  \notag \\
 = {} & -\frac{\la}{4} 
\Bigg(
-S(P(E_{i+1, i}), E_{i, i+1})+S( P(E_{i, i+1}), E_{i+1, i})
-\sum_{p=1}^n S(P(E_{i+1, p}), E_{p, i+1})
+\sum_{p=1}^n S( P(E_{p, i+1}), E_{i+1, p}) \notag \\
&-S(P(E_{i+2, i}), E_{i, i+2})+S( P(E_{i, i+2}), E_{i+2, i})
+S(P(E_{i+2, i+1}), E_{i+1, i+2})-S( P(E_{i+1, i+2}), E_{i+2, i+1}) \notag \\
&+
\frac{1}{2}\sum_{p=1}^n \Big( S(P(E_{i+1, p}), E_{p, i+1})
- S(P(E_{p, i+1}), E_{i+1, p})
- S(P(E_{i+2, p}), E_{p, i+2})
+ S(P(E_{p, i+2}), E_{i+2, p}) \Big) \notag \\
&+ S(P(E_{i+1, i+2}), E_{i+2, i+1})
- S(P(E_{i+2, i+1}), E_{i+1, i+2})
\Bigg) \notag \\
 = {} & -\frac{\la}{4} 
\Big(
-S(P(E_{i+1, i}), E_{i, i+1})+S( P(E_{i, i+1}), E_{i+1, i})
-S(P(E_{i+2, i}), E_{i, i+2})+S( P(E_{i, i+2}), E_{i+2, i})\Big)
 \notag \\
&
-\frac{\la}{8} \sum_{p=1}^n \Big(
- S(P(E_{i+1, p}), E_{p, i+1})
+S( P(E_{p, i+1}), E_{i+1, p})
- S(P(E_{i+2, p}), E_{p, i+2})
+ S(P(E_{p, i+2}), E_{i+2, p})\Big). \label{PPW} 
\end{align}

\noindent \textbf{Step 8:}  We see from \eqref{PPW} that we can cancel $[P(H_i)+\frac{1}{2}P(H_{i+1}), \wt{W}_{i+1,i+2}]$ with the first two lines of \eqref{SQKwtS}. As a conclusion, going back to \eqref{PHiPHi1} and using \eqref{SQKwtS}, we finally obtain:
  \begin{align*}
  [P(H_i), P(H_{i+1})] = {} &[S_i,Q(E_{i+2,i+1})]  + [K(E_{i+1,i+2}),\wt{S}_i]+[P(H_i)+\frac{1}{2}P(H_{i+1}), \wt{W}_{i+1,i+2}-W_{i+1,i+2}]\\
 = {} & \frac{\la^2}{16} \sum_{t=1}^n \Big
 ( S\big( S(E_{ti}, E_{i+2, t}), E_{i,i+2}\big)  
 + S\big( S(E_{it}, E_{t, i+1}), E_{i+1,i}\big)
 - S\big( S(E_{i+2,t}, E_{t, i+1}), E_{i+1, i+2}\big) \Big)\\
& - \frac{\la^2}{16}(\sum_{t=1}^n \Big( 
S \big( S(E_{it}, E_{t, i+2}),E_{i+2, i}\big)  
 + S\big( S(E_{i+1, t}, E_{ti}), E_{i, i+1}\big) 
 - S\big( S(E_{i+1, t}, E_{t, i+2}), E_{i+2, i+1}\big).
\Big)
\end{align*}

Comparing this with the formula for $[\nu_i,\nu_j]$ given in Lemma \ref{lem:nu2} completes the proof of Lemma \ref{lem:P(H)} when $j=i+1$. 

\newcommand{\arxiv}[1]
{\texttt{\href{http://arxiv.org/abs/#1}{arXiv:#1}}}
\newcommand{\doi}[1]
{\texttt{\href{http://dx.doi.org/#1}{doi:#1}}}
\renewcommand{\MR}[1]
{\href{http://www.ams.org/mathscinet-getitem?mr=#1}{MR#1}}

\end{document}